\documentclass[10pt, leqno]{article}
\usepackage{amsmath,amsfonts,amsthm,amscd,amssymb}
\numberwithin{equation}{section}
\theoremstyle{plain} 
\newtheorem{theo}{Theorem}[section]
\newtheorem{prop}[theo]{Proposition}
\newtheorem{coro}[theo]{Corollary} 
\newtheorem{lemm}[theo]{Lemma}
\theoremstyle{definition}
\newtheorem{defi}[theo]{Definition}
\newtheorem{rema}[theo]{Remark} 
\newtheorem{theo-defi}[theo]{Theorem-Definition}
\newtheorem{prop-defi}[theo]{Proposition-Definition}
\newtheorem{rema-defi}[theo]{Remark-Definition}

\newtheorem{exem}[theo]{Example}
\newtheorem{conj}[theo]{Conjecture}
\newtheorem{prob}[theo]{Problem}

\def \bet{\beta}
\def \bul{\bullet}
\def \col{\colon}
\def \Del{\Delta}

\def \Gam{\Gamma}

\def \kap{\kappa}
\def \Lam{\Lambda}
\def \lam{\lambda}

\def \Lo{\Longrightarrow}
\def \lo{\longrightarrow}
\def \lom{\longmapsto}
\def \mab{\mathbb}

\def \Om{\Omega}
\def \om{\omega}
\def \ol{\overline}
\def \os{\overset}
\def \parno{\par\noindent}

\def \sus{\subset}

\def \us{\underset}

\def \vp{\varpi}

\def \vpl{\varprojlim}

\def \wt{\widetilde}
\bigskip

\newcommand{\getsfrom}
{\ensuremath{\longleftarrow\kern-.
52em\lower-.1ex\hbox%
{$\shortmid\,$}}}

\begin{document}

\title{Log ordinarity and log plurigenera of 
a proper SNCL scheme over a log point  in characterisitic $p>0$} 
%in characteristic $p>0$
%Hyodo's criterion for 
%a proper SNCL scheme over a family of log points in characterisitic $p>0$} 
%in characteristic $p>0$}
\author{Yukiyoshi Nakkajima
\date{}\thanks{2010 Mathematics subject 
classification number: 14F30, 14F40, 14J32. 
The author is supported from two JSPS
Grant-in-Aid's for Scientific Research (C)
(Grant No.~80287440, 18K03224).\endgraf}}
\maketitle

$${\bf Abstract}$$
In this article we give two applications of 
the spectral sequence of 
the log Hodge-Witt cohomology of 
a proper SNCL scheme over the log point of a perfect field of  
characterisitic $p>0$.
%prove Hyodo's criterion for 
%a proper SNCL scheme over a family of log points in characteristic $p>0$. 

\section{Introduction} 
This article consists of two parts. One is a criterion of  
the log ordinarity of a proper SNCL(=simple normal crossing log) 
scheme over a family of log points in characteristic $p>0$ 
and the other is the lower semi-continuity of the log geometric plurigenera and 
the log Iitaka-Kodaira dimension of a proper SNCL
scheme over a log point of any characteristic. 
Though arithmetic geometers have not yet shown their interest in  
log geometric plurigenera and log Iitaka-Kodaira dimensions, 
it seems that they are important notions as in the theory of 
log geometric plurigenera and log Iitaka-Kodaira dimensions
for the case of the ``horizontal'' log structure developed in \cite{ii2} 
and \cite{ii}. 
%(I am more interested in the second part than the first part.) 
\par 
Let $g\col Y\lo T$ be a morphism of fine log schemes 
in the sense of Fontaine-Illusie-Kato (\cite{klog1}). 
Let $\Om^{\bul}_{Y/T}$ be the log de Rham complex of $Y/T$
(this was denoted by $\om^{\bul}_{Y/T}$ in \cite{klog1} and \cite{hk}). 
Let $\os{\circ}{Y}$ be the underlying scheme of $Y$. 
Set $B\Om^i_{Y/T}:={\rm Im}(d\col \Om^{i-1}_{Y/T}\lo \Om^i_{Y/T})$
($B\Om^i_{Y/T}$ is only an abelian sheaf on $\os{\circ}{Y}$).  
Assume that $T$ is of characteristic $p>0$ and that 
$g$ is proper and log smooth. 
In this article, we say that $g$ (or simply $Y/T$) is 
{\it log ordinary} ({\it in the sense of Bloch-Kato-Illusie}) 
if $R^jg_*(B\Om^i_{Y/T})=0$ for any $i,j\in {\mab Z}_{\geq 0}$. 
This is the variation of the log ordinarity defined in \cite{hp}, \cite{io}, 
\cite{nil} and \cite{lodw}. (The log ordinarity in [loc.~cit.] is the log 
version of the ordinarity defined in \cite{bk} and \cite{ir}.) 
%In this article, we omit to say ``log'' in the terminology ``log ordinary''.
\par 
Let $\kap$ be a perfect field of characteristic $p>0$. 
Let $B$ be the spectrum of 
a (complete) discrete valuation ring of 
mixed characteristics with perfect residue field $\kap$. 
Let ${\cal X}/B$ be a proper strict semistable family. 
Let $X/s$ be the log special fiber of ${\cal X}/B$ 
with canonical log structure. 
Let ${\cal W}\Om^{\bul}_X$ be the log de Rham-Witt complex of $X/s$. 
%Let $\os{\circ}{X}{}$ be the underlying scheme of $X$. 
Let $k$ be a nonnegative integer. 
Let $\os{\circ}{X}{}^{(k)}$ be the disjoint union of the $(k+1)$-fold 
intersections of the different irreducible components of $\os{\circ}{X}$. 
In \cite{hp} Hyodo has claimed the following theorem without the proof of it
in the case where $B$ is of mixed characteristics: 

\begin{theo}[{\bf \cite{hp}, \cite{io}}]\label{theo:oal}
If $\os{\circ}{X}{}^{(k)}/\kap$ is ordinary 
for all $k\in {\mab Z}_{\geq 0}$, then $X/s$ is log ordinary.  
\end{theo} 

This theorem has an important role in several articles, e.~g., in \cite{rx} 
(of course in \cite{io})
by using Illusie's result in \cite{io}: 
the blow up of an ordinary proper smooth subscheme 
in an ordinary proper smooth scheme is ordinary. 
Following \cite{io}, 
we call this theorem {\it Hyodo's criterion}. 
In \cite{io} Illusie has given a proof of this fact 
including the case where $B$ is of equal characteristic. 
He has stated another very simple method to prove Hyodo's criterion  
under the assumption of the existence of the ``$p$-adic weight spectral sequence'' of 
$H^h(X,{\cal W}\Om^i_X)$ $(h,i\in {\mab Z}_{\geq 0})$ which is 
compatible with the operator $F$, 
which was not constructed at that time. 
In the special case $i=0$, this is classically well-known (cf.~\cite[\S2]{rzs}): 
for a proper SNC(=simple normal crossing) scheme $Z/\kap$, 
there exists the following spectral sequence: 
\begin{align*} 
E_1^{-k, h+k}=
H^{h+k}(Z^{(-k+1)},{\cal W}({\cal O}_{Z^{(-k+1)}}))
\Lo
H^h(Z,{\cal W}({\cal O}_{Z})) \quad (-k,h\in {\mab Z}_{\geq 0}),  
\tag{1.1.1}\label{ali:wrzn}
\end{align*}  
In the general $i$, 
the spectral sequence has been already constructed in 
\cite{ndw}  (cf.~\cite{msemi}):

\begin{theo}[{\bf \cite[(4.1.1)]{ndw}}]\label{theo:hwwt}
Let $s$ be the log point of $\kap$, that is, 
$s=({\rm Spec}(\kap),{\mab N}\oplus \kap^*\lo \kap)$. 
Here $1\in {\mab N}$ goes to $0 \in \kap$ by the morphism 
${\mab N}\oplus \kap^*\lo \kap$. 
Let $X/s$ be a proper SNCL scheme 
defined in {\rm \cite[(1.1.15)]{nb}} 
{\rm (}{\it cf}.~{\rm \cite[(1.1.10.2)]{nlk3})}. 
Let $i$ be a nonnegative fixed integer. 
Then there exists the following spectral sequence$:$ 
%\begin{equation*}
%E_1^{-k, h+k}(W_n\Om_X^i)=
%\us{j\geq {\rm max}\{-k,0\}}{\bigoplus}
%H^{h-i-j}(\os{\circ}{X}{}^{(2j+k+1)},
%W_n\Om^{i-j-k}_{\os{\circ}{X}{}^{(2j+k+1)}})
%(-j-k)
%\tag{4.1.1;$n$}\label{eqn:hwpwt}
%\end{equation*}
%\begin{equation*} \Lo
%H^{h}(X,W_n\Om_X^i\{-i\})=H^{h-i}(X,W_n\Om_X^i),
%\end{equation*}
%$($Recall the notation $\{-i\}$ 
%in the Convention $(1).)$
\begin{align*}
E_1^{-k, h+k}=&
\us{j\geq {\rm max}\{-k,0\}}{\bigoplus}
H^{h-i-j}(\os{\circ}{X}{}^{(2j+k)},
{\cal W}\Om^{i-j-k}_{\os{\circ}{X}{}^{(2j+k)}})
(-j-k)
\tag{1.2.1}\label{ali:infhwpwt}\\
&  
\Lo H^{h-i}(X,{\cal W}\Om_X^i) \quad (q\in {\mab N}).
\end{align*}
Here $(-j-k)$ means the usual Tate twist with respect to 
the Frobenius endomorphism of $X$. 
\end{theo}
\parno 
(In \cite[(4.7)]{ndw} we have proved 
that this spectral sequence degenerates at $E_2$.)
Because this spectral sequence is compatible with the operator $F$, 
we can prove Hyodo's criterion in a very simple way as follows 
by following Illusie's idea. 
\par 
By a fundamental theorem in \cite{ir}, 
a proper smooth scheme $Y/\kap$  is ordinary if and only if 
the operator $F$ on $H^h(Y,{\cal W}\Om_Y^i)$ is bijective.   
Hence $F$ is bijective on $H^{h-i}(X,{\cal W}\Om_X^i)$ by (\ref{ali:infhwpwt}). 
Furthermore, by the log version of this fundamental theorem due to 
Lorenzon (\cite{lodw}), 
this implies that $X/s$ is log ordinary.   
This is the first proof of Hyodo's criterion in this article. 
We also give the proof of Hyodo's criterion 
by following Illusie's original proof in \cite{io} as possible. 
This is the second proof of Hyodo's criterion in this article.   
In this article we simplify his proof 
%slightly 
and generalize it.  
We give the three proofs of the following theorem  
(the variation of Hyodo's criterion): 

\begin{theo}\label{theo:twols}
Let $S$ be a family of log points in characteristic $p>0$ defined in 
{\rm \cite[(1.1)]{nb}}. Let $f\col X\lo S$ 
be a proper SNCL scheme 
defined in {\rm \cite[(1.1.15)]{nb}}. 
Let $T\lo S$ be a morphism of fine log schemes. 
Set $X_T:=X\times_ST$ and 
let $f_T\col X_T\lo T$ be the base change of $f$. 
Let $\os{\circ}{X}{}^{(k)}$ $(k\in {\mab Z}_{\geq 0})$ 
be a scheme over $\os{\circ}{S}$ 
defined in {\rm [loc.~cit.]}. 
If $\os{\circ}{X}{}^{(k)}/\os{\circ}{S}$ is ordinary for all $k$, 
then $f_T$ is log ordinary. 
\end{theo} 

\parno
In the case $S=s$, the $\os{\circ}{X}{}^{(k)}$ in (\ref{theo:twols}) 
is equal to the $\os{\circ}{X}{}^{(k)}$ in (\ref{theo:oal}). 
Hence (\ref{theo:twols})  is a generalization of (\ref{theo:oal}). 
%As already stated, the first proof of (\ref{theo:twols}) uses (\ref{ali:infhwpwt}). 
In \cite{io} Illusie has defined a nonstandard Poincar\'{e} residue isomorphism 
for $\Om^{\bul}_{X/S}$ in the case $S=s$
and he has used this isomorphism for the proof of (\ref{theo:oal}) 
(see (\ref{ali:cpsips}) below for the generalization of his isomorphism).
In this article we prefer to use only 
a standard Poincar\'{e} residue isomorphism 
for $\Om^{\bul}_{X/\os{\circ}{S}}$ in \cite[(1.3.14)]{nb} 
when we give the second proof of Hyodo's criterion.  
%to simplify Illusie's proof of the criterion in \cite{io}.  
The key lemma for the second and the third proofs of 
(\ref{theo:twols}) is the following simple one:

\begin{lemm}[{\rm {\bf Key lemma}}]\label{lemm:kel}
Let $\theta$ be a global section of $\Om^1_{X/\os{\circ}{S}}$ defined in 
{\rm \cite[(1.3)]{nb}}. 
Then the following sequence 
\begin{align*} 
0\lo B\Om^{i-1}_{X/S}\os{\theta \wedge }{\lo} 
B\Om^i_{X/\os{\circ}{S}}\lo B\Om^{i}_{X/S}\lo 0 
\quad (i\in {\mab N}) 
\end{align*} 
of $f^{-1}({\cal O}_S)$-modules is exact.  
\end{lemm}
\parno
The proof of this lemma is very easy. 
The second proof of (\ref{theo:twols}) consists of the following two steps. 
By using the assumption in (\ref{theo:twols}), 
the standard Poincar\'{e} residue isomorphism for $\Om^{\bul}_{X/\os{\circ}{S}}$
and a variant of the log Cartier isomorphism,   
we first prove that $R^jf_*(B\Om^i_{X/\os{\circ}{S}})=0$ 
for $\forall i, j\in {\mab N}$.  
Using induction on $i$, we next see that $R^jf_*(B\Om^i_{X/S})=0$ 
(and $R^jf_{T*}(B\Om^i_{X_T/T})=0$) 
very immediately.  
\par 
Let ${\cal W}\wt{\Om}{}^{\bul}_X$ be the log de Rham-Witt complex of $X/\os{\circ}{s}$ 
in the case $S=s$ defined in \cite{msemi} and \cite{ndw}. 
The third proof of (\ref{theo:twols}) uses (\ref{lemm:kel}) for the case $S=s$ 
and uses a spectral sequence of 
$H^j(X,{\cal W}\wt{\Om}{}^i_X)$ instead of (\ref{ali:infhwpwt}). 
(This spectral sequence has been constructed in \cite{nb} and it is compatible 
with the operator $F$.)
The third proof is the shortest one and makes us feel that 
the proof of (\ref{theo:twols}) is obvious if we notice (\ref{lemm:kel}). 

\par 
Set $X_{\os{\circ}{T}}:=X\times_{\os{\circ}{S}}\os{\circ}{T}$. 
By generalizing the second proof of (\ref{theo:twols}), 
we can generalize (\ref{theo:twols}) for a certain integrable connection 
\begin{align*} 
\nabla \col {\cal E}\lo {\cal E}\otimes_{{\cal O}_{X_{\os{\circ}{T}}}}
\Om^1_{X_{\os{\circ}{T}}/\os{\circ}{T}}
\end{align*}
where ${\cal E}$ is a flat quasi-coherent 
${\cal O}_{X_{\os{\circ}{T}}}$-module such that 
${\cal E}$ is ``locally generated by horizontal sections of $\nabla$'', 
((\ref{theo:twols}) is a special case of the generalized theorem for the case 
where $({\cal E},\nabla)$ is the usual derivative 
$d \col {\cal O}_{X_{\os{\circ}{T}}}\lo \Om^1_{X_{\os{\circ}{T}}/\os{\circ}{T}}$.)
See (\ref{defi:tr}) and (\ref{theo:otf}) below for this generalized statement. 
An example of our ${\cal E}$'s is given in the following way. 
For a log scheme $Y$ of characteristic $p>0$, 
let $F_Y\col Y\lo Y$ be the Frobenius endomorphism of $Y$. 
Set $X^{[p]}_{\os{\circ}{T}}:=X\times_{\os{\circ}{T},F_{\os{\circ}{T}}}\os{\circ}{T}$ 
and let $F\col X_{\os{\circ}{T}}\lo X^{[p]}_{\os{\circ}{T}}$ be 
the induced morphism by $F_{X_{\os{\circ}{T}}}$.   
%\par 
%Let ${\cal E}'$ be an integrable connection 
%\begin{align*} 
%\nabla \col {\cal E}'\lo {\cal E}'\otimes_{{\cal O}_{X^{[p]}_{\os{\circ}{T}}}}
%\Om^1_{X^{[p]}_{\os{\circ}{T}}/\os{\circ}{T}}
%\end{align*}
%``without log poles''. 
Let ${\cal E}'$ be a quasi-coherent ${\cal O}_{X^{[p]}_{\os{\circ}{T}}}$-module. 
%\begin{align*} 
%\nabla \col {\cal E}'\lo {\cal E}'\otimes_{{\cal O}_{X^{[p]}_{\os{\circ}{T}}}}
Set ${\cal E}:=F^*({\cal E}')
(=F^{-1}({\cal E}')\otimes_{{\cal O}_{X^{[p]}_{\os{\circ}{T}}}}
{\cal O}_{X_{\os{\circ}{T}}})$ and 
consider a morphism 
${\rm id}_{F^{-1}({\cal E}')}\otimes d \col 
{\cal E}\lo {\cal E}\otimes_{{\cal O}_{X_{\os{\circ}{T}}}}
\Om^1_{X_{\os{\circ}{T}}/\os{\circ}{T}}$ 
of $f^{-1}({\cal O}_T)$-modules. 
It is easy to check that ${\rm id}_{F^{-1}({\cal E}')}\otimes d$ is 
an integrable connection on  ${\cal E}$.
The connction $({\cal E},{\rm id}_{F^{-1}({\cal E}')}
\otimes d_{X_{\os{\circ}{T}}/\os{\circ}{T}})$ 
is an example of the integrable connection in our mind.

\par 
In the second part of this article, we introduce new invariants of 
a proper SNCL scheme $X$ over the log point $s$ of $\kap$ 
of pure dimension $d$: 
the log (geometric) plurigenera of $X/s$ and the log Iitaka-Kodaira dimension of it. 
\par 
Set $p_g(X/s,n,1):={\rm length}_{{\cal W}_n}H^0(X,{\cal W}_n\Om^d_X)$. 
We call this the {\it log geometric genus} of $X/s$ of {\it level} $n$. 
By using the finite length version of (\ref{ali:infhwpwt}), we prove the following: 

\begin{theo}\label{theo:pgx}
$p_g(\os{\circ}{X}{}^{(1)}/\kap,n,1)\leq p_g(X/s,n,1)$. 
\end{theo} 

\parno 
We call this inequality the {\it lower semi-continuity of the log geometric genus} 
of $X/s$. This inequality bounds the number of the irreducible components $\os{\circ}{X}_{\lam}$'s 
of $\os{\circ}{X}$ such that $p_g(\os{\circ}{X}_{\lam}, n)\geq 1$. 
\par 
More generally, set 
$p_g(X/s,n,r):={\rm length}_{{\cal W}_n}H^0(X,({\cal W}_n\Om^d_X)^{\otimes r})$ 
$(r\in {\mab Z}_{\geq 1})$. 
We call this the {\it log} ({\it geometric}) {\it plurigenus} of $X/s$ of {\it level} $n$. 
(In the following, we omit to say the adjective ``geometric'' in the terminology 
``log geometric plurigenus''.) 
In the case where the characteristic of the base field of $\kap$ is $0$, we set 
$p_g(X/s,r):={\rm dim}_{\kap}H^0(X,(\Om^d_{X/s})^{\otimes r})$ 
and we also call this the {\it log plurigenus} of $X/s$. 
We would like to prove the following lower semi-continuity of the log plurigenus of $X/s$ of {\it level} $n$: 

\begin{conj}
\begin{equation*} 
\begin{cases} 
p_g(\os{\circ}{X}{}^{(1)}/\kap,n,r)\leq p_g(X/s,n,r) & {\rm if}~{\rm ch}(k)=p>0, \\
p_g(\os{\circ}{X}{}^{(1)}/\kap,r)\leq p_g(X/s,r) & {\rm if}~{\rm ch}(k)=0. 
\end{cases}
%\tag{3.0.1}\label{eqn:pkdefpw}
\end{equation*}  
\end{conj}
It is clear that  the finite length version of (\ref{ali:infhwpwt}) itself is not useful 
for the proof of this lower semi-continuity in the case $r\geq 2$. 
%and the log Kodaira dimension. 
To prove this, we need a new idea:  
we construct the following canonical morphism of 
${\cal W}_n({\cal O}_X)$-modules: 
\begin{align*} 
\Psi_{X,n,r}\col a_*(({\cal W}_n\Om^d_{\os{\circ}{X}{}^{(1)}})^{\otimes r})
\lo ({\cal W}_n\Om^d_{X})^{\otimes r} \quad (n,r \in {\mab Z}_{\geq 1}). 
\tag{1.6.1}\label{ali:aolmd} 
\end{align*} 
Here $a\col \os{\circ}{X}{}^{(1)}\lo \os{\circ}{X}$ is the natural morphism. 
(One may think that this morphism is a strange morphism at first glance.)
We prove that $\Psi_{X,1,r}$ is injective. 
However we do not know whether $\Psi_{X,n,r}$ is injective 
for any $n\in {\mab Z}_{\geq 2}$. 
In the case ${\rm ch}(\kap)=0$, 
we also construct the following canonical morphism of 
${\cal O}_X$-modules: 
\begin{align*} 
\Psi_{X,r}\col a_*((\Om^d_{\os{\circ}{X}{}^{(1)}})^{\otimes r})
\lo (\Om^d_{X/s})^{\otimes r} \quad (r \in {\mab Z}_{\geq 1}) 
\tag{1.6.2}\label{ali:aomd} 
\end{align*} 
and we prove that this morphism is injective. 
Consequently we obtain the following: 

\begin{theo}\label{theo:ls}
Let $\kap$ be a field of any characteristic. 
Let $s$ be the log point of $\kap$.  
Let $X/s$ be a proper SNCL scheme of pure dimension $d$. 
Let $r$ be a positive integer. 
Then 
\begin{align*} 
p_g(\os{\circ}{X}{}^{(1)}/\kap,1,r)\leq p_g(X/s,1,r)
\quad {\rm if}~{\rm ch}(\kap)=p>0 
\tag{1.7.1}\label{ali:pgr}
\end{align*}  
and 
\begin{align*} 
p_g(\os{\circ}{X}{}^{(1)}/\kap,r)\leq p_g(X/s,r)
\quad {\rm if}~{\rm ch}(\kap)=0. 
\tag{1.7.2}\label{ali:pgpr}
\end{align*}  
\end{theo}

\parno 
We can also construct the obvious analogue of the morphism (\ref{ali:aomd}) for the analytic case
and can prove the obvious analogue of the inequality (\ref{ali:pgr}) in this case. 

This should be compared with the works of Noboru Nakayama (\cite{nn}), 
Siu (\cite{s}) and Hajime Tsuji (\cite{t}) using non-algebraic methods.   
They have conjectured and H.~Tsuji has claimed that the following holds, 
though he has not given the detail of the proof of the following theorem:

\begin{theo}[{\rm {\bf \cite{t}}}]\label{theo:t}
Let $\Del$ be the unit disk. 
Let ${\cal X}\lo \Del$ be the analytification of a projective strict semistable family. 
Let $t$ be an element of $\Del \setminus \{O\}$. 
Let $X$ and ${\cal X}_t$ be the special fiber and the generic fiber 
of this family, respectively. 
Then 
\begin{align*} 
p_g(\os{\circ}{X}{}^{(1)}/{\mab C},r)\leq p_g({\cal X}_{\rm t}/{\mab C},r)
\quad (r\in {\mab Z}_{\geq 1}).
\tag{1.8.1}\label{ali:pgcr}
\end{align*}   
\end{theo}
An algebraic proof of (\ref{theo:t}) has not been known.
\par 
Because the log plurigenera is upper semi-continuous, this  
implies the smooth deformation invariance of the log plurigenera 
of a proper smooth algebraic family in the case ${\rm ch}(\kap)=0$. 
On the other hand, 
the smooth deformation invariance of the log plurigenera of  
a proper smooth analytic case does not necessarily hold by \cite{in}, 
\cite{ni} and \cite{ueno} and our result for the analytic case.

%Though I have found this theorem about 1996 and I have finished 
%writing a quite long preprint (\cite{nmw}) 
%including this theorem about 1998, 
%I have used theory of log de Rham-Witt complexes and 
%I have found numerous mathematical mistakes 
%in the theory of log de Rham-Witt complexes 
%in published papers and I have had to correct all of them 
%and I think that I have finished the corrections now. 
%Hence I think that it is time to choose important theorems 
%in it and to announce them and give the proof of them.  
%The lower semi-continuity implies that of the log Kodaira dimension, 
%which we shall define in the text. 
In the case where the proper or projective SNCL scheme is the log special fiber of 
a proper or projective strict semifamily over a (complete) discrete valuation ring, 
we would like to raise the following problem: 

\begin{prob}\label{prob:gpc}
Let ${\cal V}$ be a (complete) discrete valuation ring of mixed characteristics or 
equal characteristic $p>0$. 
Let $K$ be the fraction field of ${\cal V}$. 
Let ${\cal X}/{\cal V}$ be a proper or projective strict semistable family. 
Let $X/s$ and ${\cal X}_K$ be the log special fiber of ${\cal X}/B$ 
and the generic fiber of ${\cal X}/B$, respectively. 
When does the following equality hold?
\begin{align*} 
p_g(X/s,n,r)=p_g({\cal X}_K/K,n,r) 
\quad (n,r\in {\mab Z}_{\geq 1}). 
\tag{1.9.1}\label{ali:pjxs}
\end{align*}  
\end{prob}
We think that nothing nontrivial about (\ref{prob:gpc}) has been known. 
In the analytic case, the analogous equality to (\ref{ali:pjxs}) does not hold in general by 
\cite{in},

%In the case where $B$ is of mixed characterstics, 
%we cannot formulate when the analogue of  (\ref{ali:pjxs}) holds. 
%(We only know that the inequality ``$\geq$'' holds.) 
\par 
The contents of this article are as follows. 
\par 
In \S\ref{sec:hc} we prove (\ref{theo:twols}) quickly by using (\ref{ali:infhwpwt}). 
We also prove the analogue of (\ref{theo:twols}) for an open log scheme. 
\par 
In \S\ref{sec:hpc} we prove (\ref{theo:twols}) in a similar way to that in \cite{io}. 
However we do not use Illusie's Poincar\'{e} residue isomorphism. 
We simplify his proof by using the Key lemma. 
In addition, we prove the Hyodo's criterion for 
the integrable connection $({\cal E},\nabla)$ as already stated. 
\par 
In \S\ref{sec:htc} we prove (\ref{theo:twols}) without using (\ref{ali:infhwpwt}) nor 
Illusie's Poincar\'{e} residue isomorphism. Instead we use the sheaf 
${\cal W}\wt{\Om}^i_X$ in \cite{msemi} and \cite{ndw} and a 
spectral sequence constructed in \cite{nb}. 
%The proof of (\ref{theo:twols}) in this section is the shortest one in this article. 
% than the proof in \S\ref{sec:hc} 
%because we do not use the $p$-adic Steenbrink complex in \cite{msemi} and \cite{ndw} 
%in \S\ref{sec:htc}. 
\par 
In \S\ref{sec:llss} we prove the lower semi-continuity of the log genus 
by using the finite length version of (\ref{ali:infhwpwt}). 
\par 
In \S\ref{sec:llgpss} we prove the lower semi-continuity of 
the log plurigenus of level 1. It is a future problem to show  
whether the lower semi-continuity of the log plurigenera of level $n$ holds for any $n$. 
\par 
In \S\ref{sec:gir} we generalize Illusie's Poincar\'{e} residue isomorphism. 
\par 
In \S\ref{sec:hfc} we give an easy criterion for the quasi-$F$-splitness for 
an SNCL scheme. 
\par 
In \S\ref{sec:em} we give an easy remark about the ordinarity at $0$ 
defined in \cite{nlfc}. 

\par
\bigskip
\parno
{\bf Acknowledgment.} 
I would like to express my gratitude to L.~Illusie for 
asking me an interesting question in \S\ref{sec:hfc} of this article   
at a workshop ``Arithmetic geometry 2018'' at Hakodate organized by 
T.~Yamazaki, Y.~Goto and M.~Aoki. 
His question has turned my eye to Hyodo's criterion 
in his article \cite{io}. 
%and has motivated me to write this paper partially. 
I would also like to express my gratitude to T.~Fujisawa for having removed 
my misunderstanding about the definition of 
Illusie's Poincar\'{e} residue isomorphism in \cite{io}.

\bigskip
\par\noindent
{\bf Notations.} 
(1) For a complex $({\cal F}^{\bul},d^{\bul})$ of sheaves on a topological space, 
we denote ${\rm Ker}(d^i \col {\cal F}^i\lo {\cal F}^{i+1})$ 
(resp.~${\rm Im}(d^{i-1} \col {\cal F}^{i-1}\lo {\cal F}^i)$) 
mainly by $Z^i({\cal F}^{\bul})$ and sometimes by 
$Z{\cal F}^i$ (resp.~mainly by $B^i({\cal F}^{\bul})$ and sometimes by 
$B{\cal F}^i$ ) depending on the length of the letter ${\cal F}$. 
\par 
(2) For a log scheme $Z$, we denote by $\os{\circ}{Z}$ 
the underlying scheme of $Z$. 
  
%\bigskip
%\par\noindent
%{\bf Conventions.} 

\section{The first proof of (\ref{theo:twols})}\label{sec:hc}
In this section we give the first proof of (\ref{theo:twols}) 
by using the $p$-adic weight spectral (\ref{ali:infhwpwt}) 
(\cite{ndw}, cf.~\cite{msemi}). 
The idea of the proof in this section is Illusie's one (\cite[Appendice (2.8)]{io}).

\par 
Let $\kap$ be a perfect field of characteristic $p>0$. 
Let $s$ be a fine log smooth scheme such that 
$\os{\circ}{s}={\rm Spec}(\kap)$. 
Let $Y/s$ be a log smooth scheme of Cartier type. 
Let ${\cal W}_{\star}\Om^{\bul}_{Y}$ $(\star=n\in {\mab Z}_{\geq 1}$ or nothing) 
be the log de Rham-Witt complex defined in \cite[(4.1)]{hk}. 
Let ${\cal W}_n(Y)$ be the canonical lift of $Y$ over ${\cal W}_n(\kap)$. 
Identify ${\cal O}_{{\cal W}_n(Y)}={\cal W}_n({\cal O}_Y)$ 
with ${\cal W}_n\Om^0_{Y}$ via the canonical isomorphism 
$s_n\col {\cal W}_n({\cal O}_Y)\os{\sim}{\lo} {\cal W}_n\Om^0_{Y}$ 
(\cite[(7.5)]{ndw}). 
First let us recall the following result. 
This is the log version of \cite[I (3.11)]{idw} 
if one forgets the structure of ${\cal W}_{n+1}({\cal O}_Y)$-modules 
of sheaves in (\ref{ali:mfee}) below. 
A variant of the following will be used in \S\ref{sec:htc}.

\begin{prop}[{\bf \cite[(3.10)]{nlfc}}]\label{prop:fzg}
Let $F_{{\cal W}_n(Y)} \col {\cal W}_n(Y) \lo {\cal W}_n(Y)$ be 
the Frobenius endomorphism of ${\cal W}_n(Y)$. 
Let $i$ be a nonnegative integer. 
Let 
$C^{-1}\col \Om^i_{Y/s}\os{\sim}{\lo} F_{Y*}({\cal W}_1\Om^i_Y)
=F_{Y*}({\cal H}^i(\Om^{\bul}_{Y/s}))$ 
be the log inverse Cartier isomorphism due to Kato {\rm (\cite{klog1})}. 
Let $R\col {\cal W}_{n+1}\Om^i_{Y}\lo {\cal W}_n\Om^i_{Y}$ 
be the projection. 
Then the following hold$:$ 
\par 
$(1)$ The following sequence  
\begin{align*} 
{\cal W}_{n+1}\Om^i_{Y}
\os{F}{\lo} F_{{\cal W}_n(Y)*}({\cal W}_n\Om^i_{Y})
\os{F_{{\cal W}_n(Y)*}(F^{n-1}d)}{\lo}  
B_n{\cal W}_1\Om^{i+1}_Y
\lo 0
\tag{2.1.1}\label{ali:mfee}
\end{align*}
is an exact sequence of ${\cal W}_{n+1}({\cal O}_Y)$-modules. 
Here $B_n{\cal W}_1\Om^{i+1}_Y=C^{-1}(B_n\Om^{i+1}_{Y/s})$, 
where $B_n\Om^{i+1}_{Y/s}$ is an $F^n_{Y*}({\cal O}_Y)$-module 
defined inductively by the following isomorphism 
$C^{-1}\col B_n\Om^{i+1}_{Y/s}\os{\sim}{\lo} 
B_{n+1}\Om^{i+1}_{Y/s}/F_{Y*}^n(B_1\Om^{i+1}_{Y/s})$ 
and $B_1\Om^{i+1}_{Y/s}:=F_{Y*}(d\Om^{i}_{Y/s})$. 
\par 
$(2)$ Let $C\col B_{n+1}{\cal W}_1\Om^{i+1}_Y\lo B_n{\cal W}_1\Om^{i+1}_Y$ 
be the following composite morphism$:$ 
\begin{align*} 
B_{n+1}{\cal W}_1\Om^{i+1}_Y&=C^{-1}(B_{n+1}\Om^{i+1}_{Y/s})
\os{C^{-1}({\rm proj})}{\lo} C^{-1}(B_{n+1}\Om^{i+1}_{Y/s}/B_1\Om^{i+1}_{Y/s})\\
& \os{C^{-1}(C^{-1}),\sim}{\longleftarrow} C^{-1}(B_n\Om^{i+1}_{Y/s})
=B_n{\cal W}_1\Om^{i+1}_Y. 
\end{align*}
Then the following diagram is commutative$:$
\begin{equation*} 
\begin{CD} 
F_{{\cal W}_{n+1}(Y)*}({\cal W}_{n+1}\Om^i_Y)
@>{F^nd}>>B_{n+1}{\cal W}_1\Om^{i+1}_Y\\
@V{R}VV @VV{C}V\\
F_{{\cal W}_n(Y)*}({\cal W}_n\Om^i_Y)
@>{F^{n-1}d}>>B_n{\cal W}_1\Om^{i+1}_Y. 
\end{CD} 
\tag{2.1.2}\label{cd:nmwee}
\end{equation*} 
\end{prop}

\begin{prop}\label{prop:oh}
Let $\kap$ be a perfect field of characteristic $p>0$.   
Then {\rm (\ref{theo:twols})} holds in the case 
where $S$ is the log point $s$ of $\kap$. 
\end{prop}
\begin{proof}
By using (\ref{ali:mfee}) and noting that 
$R\col {\cal W}_{n+1}\Om^i_Y\lo {\cal W}_n\Om^i_Y$ is surjective, 
we have the following exact sequence: 
\begin{align*} 
0\lo {\cal W}\Om^i_Y\lo F_{{\cal W}({\cal O}_Y)*}({\cal W}\Om^i_Y)
\lo \vpl_nB_n{\cal W}_1\Om^{i+1}_Y\lo 0. 
\end{align*} 
Hence, by the log version of \cite[IV (4.11.6), (4.13)]{ir}, 
$X/s$ is log ordinary if and only if  
the operator $F\col H^j(X,{\cal W}\Om^i_{X})\lo H^j(X,{\cal W}\Om^i_{X})$ 
is bijective for any $i$ and $j$. 
(This has also been obtained in \cite[(4.1)]{lodw}.)
Let $\{E_{1}^{ij}\}_{i,j\in {\mab Z}}$ be the $E_1$-terms of 
the spectral sequence (\ref{ali:infhwpwt}). 
Here we ignore the Tate twist $(-j-k)$. 
Let $({\cal W}_nA_X^{\bul \bul},P)$ $(n\in {\mab Z}_{\geq 1})$ 
be the filtterd double complex defined in \cite[\S2]{ndw} and denoted by 
$(W_nA_X^{\bul \bul},P)$ (cf.~\cite{msemi}). 
By the definition of the operator $F$ and 
by considering a local log smooth lift $X/s$ over ${\cal W}_n(s)$, 
the following two diagrams are commutative: 
\begin{equation*} 
\begin{CD} 
{\cal W}_{n+1}\Om_X^i@>{\theta \wedge,~\simeq}>> {\cal W}_{n+1}A_X^{i \bul}\\
@V{F}VV @VV{F}V \\
{\cal W}_{n}\Om_X^i@>{\theta \wedge,~\simeq}>> {\cal W}_nA_X^{i \bul}, 
\end{CD}
\end{equation*} 
\begin{equation*} 
\begin{CD} 
{\rm gr}_k^P{\cal W}_{n+1}A_X^{i \bul}@>{{\rm Res},~\simeq}>> 
\us{j\geq {\rm max}\{-k,0\}}
{\bigoplus}
{\cal W}_{n+1}\Om_{\os{\circ}{X}{}^{(2j+k)}}^{i-j-k}
\{-j\}\\
@V{F}VV @VV{F}V \\
{\rm gr}_k^P{\cal W}_{n}A_X^{i \bul}@>{{\rm Res},~\simeq}>> 
\us{j\geq {\rm max}\{-k,0\}}
{\bigoplus}
{\cal W}_{n}\Om_{\os{\circ}{X}{}^{(2j+k)}}^{i-j-k}
\{-j\}.  
\end{CD}
\end{equation*} 
Because the spectral sequence (\ref{ali:infhwpwt}) 
is obtained by these commutative diagrams, 
it is compatible with $F$'s.
By the assumption, 
the operator $F\col E_{1}^{-k, h+k}\lo 
E_{1}^{-k, h+k}$ is bijective. 
%Because the spectral sequence 
%is compatible with the operators $F$'s, 
Hence the operator
$F\col H^j(X,{\cal W}\Om^i_{X})\lo H^j(X,{\cal W}\Om^i_{X})$ 
is bijective. We complete the proof of (\ref{prop:oh}). 
\end{proof}

\begin{prop}[{\bf \cite[Appendice (1.2), (1.9)]{io}}]\label{prop:bc} 
Let $g\col Y\lo T$ be a proper log smooth scheme 
of fine log schemes. Let $i$ be an integer. 
Then the following are equivalent$:$
\par 
$(1)$ $R^jg_*(B\Om^i_{Y/T})=0$ for $\forall j\geq 0$.  
\par 
$(2)$ For any exact closed point $t\in T$ 
$(t$ is a closed point of $\os{\circ}{T}$ endowed with 
the inverse image of the log structure of $T)$, 
$H^j(Y_t,B\Om^i_{Y_t/t})=0$ for $\forall j\geq 0$. 
Here $Y_t:=Y\times_Tt$. 
\end{prop} 
\begin{proof} 
Set $Y':=Y\times_{T,F_T}T$ and let $F_{Y/T}\col Y\lo Y'$ 
be the relative Frobenius morphism of $Y/T$. 
Let $g'\col Y'\lo T$ be the structural morphism. 
For an exact closed point $t\in T'$, let $g'_t\col Y'_t\lo t$ 
be the base change morphism of $g'$. 
Set $B_1\Om^i_{Y/T}:=F_{Y/T*}(B\Om^i_{Y/T})$. 
Then $B_1\Om^i_{Y/T}$ is an ${\cal O}_{Y'}$-module. 
\par 
Because $\os{\circ}{F}_{Y/T}$ is a homeomorphism of topological spaces 
(\cite[XV Proposition 2 a)]{sga5-2}), 
it suffices to prove that the following are equivalent$:$
\par 
$(1)'$ $R^jg'_*(B_1\Om^i_{Y/T})=0$ for $\forall j\geq 0$.  
\par 
$(2)'$ For any exact closed point $t\in T$, 
$H^j(Y'_t,B_1\Om^i_{Y_t/t})=0$ for $\forall j\geq 0$.  
%Set $Y^{(p^n)}:=Y^{(p^{(n-1)}}\times_{T,F_T}T$. 
By \cite[(1.13)]{lodw} 
the sheaf $B_1\Om^i_{Y/T}$ $(m,i\in {\mab N})$ 
is a locally free sheaf of ${\cal O}_{Y'}$-modules of finite rank. 
By the log K\"{u}nneth formula (\cite[(6.12)]{klog1}), 
we have the following isomorphism
\begin{align*} 
Rg'_*(B_1\Om^i_{Y/T})\otimes^L_{{\cal O}_T}\kap_t
\os{\sim}{\lo} Rg'_{t*}(B_1\Om^i_{Y_t/t}). 
\end{align*} 
The rest of the proof is the same as that of 
\cite[Appendice (1.2), (1.9)]{io}. 
Indeed, the implication $(2)'\Lo (1)'$ follows from the fact that 
$Rg'_*(B_1\Om^i_{Y/T})$ is a perfect complex of ${\cal O}_T$-modules 
(since $g'$ is proper). 
The implication $(1)'\Lo (2)'$ follows from this perfectness and Nakayama's lemma. 
\end{proof} 

\par 
We prove (\ref{theo:twols}).
Let the notations be as in (\ref{theo:twols}).  
Set $S_{\os{\circ}{T}}:=S\times_{\os{\circ}{S}}{\os{\circ}{T}}$ and 
$X_{\os{\circ}{T}}:=X\times_SS_{\os{\circ}{T}}$. 
Let $f_{\os{\circ}{T}}\col X_{\os{\circ}{T}}\lo S_{\os{\circ}{T}}$ 
be the structural morphism. 
Because $B\Om^i_{X_{\os{\circ}{T}}/S_{\os{\circ}{T}}}$ commutes with 
the base change morphism $T\lo S_{\os{\circ}{T}}$ (\cite[(1.13)]{lodw}), 
\begin{align*} 
R^jf_{T*}(B\Om^i_{X_T/T})=
R^jf_{T*}
(B\Om^i_{X_{\os{\circ}{T}}/S_{\os{\circ}{T}}}\otimes_{{\cal O}_T}{\cal O}_T)
=R^jf_{\os{\circ}{T}*}
(B\Om^i_{X_{\os{\circ}{T}}/S_{\os{\circ}{T}}}).
\tag{2.3.1}\label{ali:ftbx}
\end{align*} 
(To obtain (\ref{ali:ftbx}), 
one may use a fact 
that $B\Om^i_{X_{\os{\circ}{T}}/S_{\os{\circ}{T}}}
=B\Om^i_{X_T/T}$ 
(since $\Om^j_{X_{\os{\circ}{T}}/S_{\os{\circ}{T}}}=\Om^j_{X_T/T}$ 
for $j=i,i+1$).) 
Hence we may assume that $T=S_{\os{\circ}{T}}$ and in fact, $T=S$. 
By (\ref{prop:bc}) we can assume that $S$ is the log point. 
Consider the perfection $\kap^{\rm pf}$ of $\kap$ 
and let $s^{\rm pf}$ be the log point 
whose underlying scheme is ${\rm Spec}(\kap^{\rm pf})$ 
and whose log structure is the inverse image of the log structure 
of $s$ by the natural morphism ${\rm Spec}(\kap^{\rm pf})
\lo {\rm Spec}(\kap)$. 
Set $X_{s^{\rm pf}}:=X\times_ss^{\rm pf}$. 
By (\ref{prop:oh}),  
$H^j(X_{s^{\rm pf}},
B\Om^i_{X_{s^{\rm pf}}/{s^{\rm pf}}})=0$ for $\forall i$ and $\forall j$. 
Because $H^j(X_{s^{\rm pf}},
B\Om^i_{X_{s^{\rm pf}}/{s^{\rm pf}}})=
H^j(X,B\Om^i_{X/s})\otimes_{\kap}\kap^{\rm pf}$, 
$H^j(X,B\Om^i_{X/s})=0$. 
By (\ref{prop:bc}), $R^jf_*(B\Om^i_{X/S})=0$ for $\forall i$ and $\forall j$. 
We complete the proof of (\ref{theo:twols}). 
\par 
We can also obtain the Hyodo's criterion for an open smooth scheme as follows.

\begin{theo}\label{theo:hco} 
Let $S$ be a scheme of characteristic $p>0$. 
Let $f\col (X,D)\lo S$ be a proper smooth scheme with relative SNCD 
{\rm (\cite[(2.1.7)]{nh2})}. Let $D^{(i)}$ $(i\in {\mab Z}_{\geq 0})$ 
be a scheme defined in {\rm [loc.~cit., (2.2.13.2)]}. 
Then if $D^{(i)}$ is ordinary for all $i$, then 
$f$ is log ordinary. 
\end{theo}
\par 
To prove this theorem, we have only to use the following spectral sequence 
(\cite[(5.7.1)]{ndw}): 
\begin{equation*}
E_1^{-k,h+k}=H^{h-i}(D^{(k)},
W\Om_{D^{(k)}}^{i-k})(-k) \Lo 
H^{h-i}(X,W\Om_X^i(\log D)).
\tag{2.4.1}\label{eqn:iuonss}
\end{equation*}
instead of the spectral sequence (\ref{ali:infhwpwt}) 
in the case $S={\rm Spec}(\kap)$. 
\par 
We say that $f$ is {\it log ordinary with compact suppport} 
if $R^jf_*(d\Om^i(-\log D))=0$ for $\forall i,j$.  
We can also obtain the following theorem:  

\begin{theo}\label{theo:ocss} 
Let the notations be as in {\rm (\ref{theo:hco})}. 
If $D^{(i)}$ is ordinary for all $i$, then 
$f$ is log ordinary with compact suppport. 
\end{theo}
To prove this theorem, we have only to use the following spectral sequence 
(\cite[(5.7.2)]{ndw}): 
\begin{equation*}
E_1^{k,q-k}=H^{h-i-k}(D^{(k)},
W\Om_{D^{(k)}}^i) \Lo 
H^{h-i}(X,W\Om_X^i(-\log D)). 
\tag{2.5.1}\label{eqn:icons}
\end{equation*}
in the case  $S={\rm Spec}(\kap)$.

\section{The second proof of (\ref{theo:twols})}\label{sec:hpc}
In this section we prove (\ref{theo:twols}), (\ref{theo:hco}) and (\ref{theo:ocss})  
without using 
the $p$-adic weight spectral sequences (\ref{ali:infhwpwt}), 
(\ref{eqn:iuonss}) and (\ref{eqn:icons}), respectively. 
The proof of (\ref{theo:twols}) in this section 
includes a simplification of the proof of Hyodo's criterion in \cite[Appendice]{io} 
as already stated in the Introduction: 
the filtration $P$ on $\Om^{\bul}_{X/S}$ 
induced by the filtration $P$ on $\Om^{\bul}_{X/\os{\circ}{S}}$ 
defined in \cite{nb} nor a generalization of the residue map in 
[loc.~cit., (2.1.3)] is unnecessary; 
we use only the filtration $P$ on $\Om^{\bul}_{X/\os{\circ}{S}}$.  
The filtration $P$ on $\Om^{\bul}_{X/\os{\circ}{S}}$ 
is more directly related with the complex 
$\Om^{\bul}_{\os{\circ}{X}{}^{(j)}/\os{\circ}{S}}$ 
$(j\in {\mab N})$ than the filtration $P$ on $\Om^{\bul}_{X/S}$.  
From the earlier part of this section,  
we consider a certain integrable connection without log poles including 
the derivative $d\col {\cal O}_X\lo \Om^1_{X/\os{\circ}{S}}$.
The main result in this section is (\ref{theo:otf}) below. 
This includes  (\ref{theo:twols}). 
%\par 
%First we point out a non-minor mistake in 
%the proof of \cite[Appendice]{io}. 
%\par 
%Let the notations be as in \cite[Appendice]{io}. 
%There does not a ``residue morphism'' 
%$P_n\om_{X/s}^{q+n}\lo L_{n+1}^h$.  
%Indeed, in the case $r=2$, $q=0$ and $n=1$, by the definition of 
%the morphism $P_1\om_{X/s}^1\lo L_1$, 
%the image of $d\log x_1$ (resp.~$d\log x_2$) 
%of $P_1\om_{X/s}^1$ in  $L_1$ is $(1,0)$ (resp.~$(0,1)$).
%Hence the image of  $d\log x_1+d\log x_2=(1,1)$. 
%However $ d\log x_1+ d\log x_2=0$ in $P_1\om_{X/s}^1$. 
%Similarly there does not exist a residue morphism 
%${\rm Res} \col P_n\Om_{X/S}^{q+n}\lo L_{n+1}^h$. 

\par 
Let $S$ be a family of log points (\cite[(1.1)]{nb}) 
and let $Y/S$ be a log smooth scheme. 
Let $g\col Y\lo S$ be the structural morphism. 
Let $e$ be a local section of $M_S$ such that 
the image of $e$ in $M_S/{\cal O}_S^*$ is the local generator. 
Set $\theta_S:= d\log e\in \Om^1_{S/\os{\circ}{S}}$. 
Then this section is independent of the choice of $e$ 
and it is globalized. 
Set $\theta:=g^*(d\log e) \in \Om^1_{Y/\os{\circ}{S}}$.
Let $\pi^{\bul} \col \Om^{\bul}_{Y/\os{\circ}{S}}\lo \Om^{\bul}_{Y/S}$ 
be the natural projection.  
Because $Y/S$ is log smooth, 
the following sequence
\begin{align*} 
0\lo g^*(\Om^1_{S/\os{\circ}{S}})
\os{\theta \wedge}{\lo} 
\Om^1_{Y/\os{\circ}{S}}
\os{\pi^1}{\lo} 
\Om^1_{Y/S}\lo 0
\tag{3.0.1}\label{ali:eg} 
\end{align*} 
is locally split (\cite[(3.12)]{klog1}). 
%Here the morphism 
%$\Om^1_{Y/\os{\circ}{S}}\lo \Om^1_{Y/S}$ is the natural projection. 
%It is clear that the following sequence 
%\begin{align*} 
%\Om^{\bul}_{X/\os{\circ}{S}}[-1]
%\os{\theta \wedge}{\lo} 
%\Om^{\bul}_{X/\os{\circ}{S}}\lo 
%\Om^{\bul}_{X/S}\lo 0
%\tag{3.2.4}\label{ali:cfs} 
%\end{align*} 
%is exact. Here the morphism 
%$\Om^{\bul}_{X/\os{\circ}{S}}\lo 
%\Om^{\bul}_{X/S}$ is the natural projection. 
By a very special case of \cite[(1.7.20.1)]{nb}, 
the following sequence is exact: 
%locally split: 
\begin{align*} 
0\lo \Om^{\bul}_{Y/S}[-1]
\os{\theta \wedge}{\lo} 
\Om^{\bul}_{Y/\os{\circ}{S}}\os{\pi^{\bul}}{\lo} 
\Om^{\bul}_{Y/S}\lo 0. 
\tag{3.0.2}\label{ali:cfos} 
\end{align*} 

\par 
The following lemma is a key lemma  
which enables us to simplify Illusie's proof for Hyodo's criterion in \cite{io}. 

\begin{lemm}\label{lemm:pete}
For each $i$, the resulting sequences of 
{\rm (\ref{ali:cfos})} by the operations $B^i$, $Z^i$ and ${\cal H}^i$ 
$(i\in {\mab Z}_{\geq 0})$ are exact. 
\end{lemm} 
\begin{proof}
%First we give the elementary and direct proof of this lemma. 
%Later we also give another proof it when $Y/S$ is of Cartier type 
%and $({\cal E},\nabla)$ is a certain connection. 
%\par 
We have only to prove that the following sequences are exact$:$ 
\begin{align*} 
0\lo B\Om^{i-1}_{Y/S}\os{\theta \wedge}{\lo} B\Om^i_{Y/\os{\circ}{S}}
\lo B\Om^i_{Y/S}\lo 0
\tag{3.1.1}\label{ali:cftois} 
\end{align*} 
and 
\begin{align*} 
0\lo Z\Om^{i-1}_{Y/S}
\os{\theta \wedge}{\lo} Z\Om^i_{Y/\os{\circ}{S}}\lo Z\Om^i_{Y/S}\lo 0. 
\tag{3.1.2}\label{ali:cftoeis} 
\end{align*} 
Because the problem is local, 
%We may assume that $S=(\os{\circ}{S},{\mab N}\oplus {\cal O}_S^*\lo {\cal O}_S)$, 
%where the morphism ${\mab N}\oplus {\cal O}_S^*\lo {\cal O}_S$ 
%is defined the morphisms ${\mab N}\owns 1\lom 0\in {\cal O}_S$ and 
%${\cal O}_S^*\subset {\cal O}_S$. 
%For two nonnegative integers $a$ and $b$ such that $a\leq b$,  
%set  
%\begin{equation*} 
%{\mab A}_{\os{\circ}{S}}(a,b):=
%\ul{{\rm Spec}}_{\os{\circ}{S}}
%({\cal O}_S[x_0, \ldots, x_b]/(\prod_{i=0}^ax_i)). 
%\end{equation*}  
%Let $M_{\os{\circ}{S}}(a,b)$ be the log structure on 
%${\mab A}_{\os{\circ}{S}}(a,b)$ associated to the following morphism 
%\begin{equation*} 
%{\mab N}^{{\oplus}(a+1)}\owns 
%(0, \ldots,0,\os{i}{1},0,\ldots, 0)\lom x_{i-1}\in 
%{\cal O}_{\os{\circ}{S}}[x_0, \ldots, x_b]/(\prod_{i=0}^ax_i).
%\end{equation*}   
%Let ${\mab A}_S(a,b)$ be the resulting log scheme over $S$.
%The diagonal immersion ${\mab N}\os{\sus}{\lo} {\mab N}^{{\oplus}(a+1)}$ 
%induces a morphism ${\mab A}_S(a,b)\lo S$ of log schemes. 
we may assume that there exists 
a basis $\{\theta,\{\om_i\}_{i=1}^d\}$ of  $\Om^1_{Y/\os{\circ}{S}}$.  
Then $\{\pi^1(\om_i)_{i=1}^d\}$ is a basis of $\Om^1_{Y/S}$. 
Let $\iota \col \Om^1_{Y/S}\lo \Om^1_{Y/\os{\circ}{S}}$ 
$(\pi^1(\om_i)\lom \om_i)$ be the splitting 
of the projection 
$\pi^1\col \Om^1_{Y/\os{\circ}{S}}\lo \Om^1_{Y/S}$.  
Let $\Om^i_{Y/S}\lo \Om^i_{Y/\os{\circ}{S}}$ $(i\in {\mab Z}_{\geq 1})$ 
be the induced splitting 
by $\iota$ above and 
denote it by $\iota \col \Om^1_{Y/S}\lo \Om^1_{Y/\os{\circ}{S}}$ again. 
%=\bigoplus_{i=0}^{a-1}{\cal O}_Yd\log x_i\oplus  
%\bigoplus_{i=a+1}^{d}{\cal O}_Ydx_i\simeq {\cal O}_Y^{\oplus d}$. 
%the following commutative diagram
%\begin{equation*} 
%\begin{CD} 
%{\cal O}_Y@={\cal O}_Y\\
%@V{d}VV @VV{d}V\\
%\Om^1_{Y/\os{\circ}{S}}@<{\iota}<<\Om^1_{Y/S}. 
%\end{CD}
%\end{equation*} 
Then $\Om^i_{Y/\os{\circ}{S}}=
\iota(\Om^i_{Y/S})\oplus \theta \wedge(\Om^{i-1}_{Y/S})$.
Let $d\col \Om^i_{Y/\os{\circ}{S}}\lo \Om^{i+1}_{Y/\os{\circ}{S}}$ 
and $d_{Y/S}\col \Om^i_{Y/S}\lo \Om^{i+1}_{Y/S}$ 
be the standard differentials. 
By expressing a local section of $\Om^i_{Y/\os{\circ}{S}}$ by the basis 
$\{\theta,\{\om_i\}_{i=1}^d\}$ of $\Om^i_{Y/\os{\circ}{S}}$ 
(cf.~the first proof of (\ref{lemm:pte}) below), 
it is obvious that the following diagram is commutative:   
\begin{equation*} 
\begin{CD}
\Om^i_{Y/\os{\circ}{S}}
@=\iota(\Om^i_{Y/S})
\oplus 
\theta \wedge (\Om^{i-1}_{Y/S})\\
@V{d}VV @VV{\iota(d_{Y/S})\oplus (-\theta \wedge d_{Y/S})}V\\
\Om^{i+1}_{Y/\os{\circ}{S}}
@=\iota(\Om^{i+1}_{Y/S})\oplus \theta \wedge (\Om^i_{Y/S}). 
\end{CD}
\tag{3.1.3}\label{cd:kttc}
\end{equation*}
%Here $(\nabla-\iota(\nabla_{Y/S}),\nabla_{Y/S})$ is 
%the following natural morphism 
%\begin{align*} 
%\iota({\cal E}\otimes_{{\cal O}_Y}\Om^i_{Y/S})
%\oplus 
%\{\theta \wedge ({\cal E}\otimes_{{\cal O}_Y}\Om^{i-1}_{Y/S})\}
%\lo 
%\theta \wedge ({\cal E}\otimes_{{\cal O}_Y}\Om^i_{Y/S}). 
%\end{align*} 
Hence 
\begin{align*} 
B\Om^i_{Y/\os{\circ}{S}}=
\iota(B\Om^i_{Y/S})\oplus \theta \wedge (B\Om^{i-1}_{Y/S}) 
\end{align*}   
and 
\begin{align*} 
Z\Om^i_{Y/\os{\circ}{S}}=\iota(Z\Om^i_{Y/S})\oplus \theta \wedge (Z\Om^{i-1}_{Y/S}).
\end{align*}  
This implies that the sequences (\ref{ali:cftois}) and (\ref{ali:cftoeis})  
are exact. 
We complete the proof of (\ref{lemm:pete}). 
\end{proof}
%In this article we use the following sequence 
%\begin{align*} 
%0\lo \Om^{\bul-1}_{Y/S}
%\os{\theta \wedge }{\lo} 
%\Om^{\bul}_{Y/\os{\circ}{S}}\os{\pi^{\bul}}{\lo} 
%\Om^{\bul}_{Y/S}\lo 0
%\tag{3.0.3}\label{ali:cfnaos} 
%\end{align*} 
%instead of (\ref{ali:cfos}) (cf.~(\ref{rema:tfm}) below). 
\par 
More generally, 
let ${\cal E}$ be a quasi-coherent ${\cal O}_Y$-module 
and let 
\begin{align*} 
\nabla \col {\cal E}\lo {\cal E}\otimes_{{\cal O}_Y}\Om^1_{Y/\os{\circ}{S}}
\tag{3.1.4}\label{ali:cfes} 
\end{align*} 
be an integrable connection.   
Let 
\begin{align*} 
\nabla_{Y/S} \col {\cal E}\lo {\cal E}\otimes_{{\cal O}_Y}\Om^1_{Y/S}
\tag{3.1.5}\label{ali:cfees} 
\end{align*} 
be the induced integrable connection by $\nabla$.   
Then we have the following morphism of complexes of $g^{-1}({\cal O}_S)$-modules: 
\begin{align*} 
\theta  \wedge \col {\cal E}\otimes_{{\cal O}_Y}\Om^{\bul}_{Y/\os{\circ}{S}}[-1]\owns 
e\otimes \om \lom e\otimes(\theta \wedge \om )\in   
{\cal E}\otimes_{{\cal O}_Y}\Om^{\bul}_{Y/\os{\circ}{S}}
\tag{3.1.6}\label{ali:cfs} 
\end{align*} 
which induces the following morphism of complexes of $g^{-1}({\cal O}_S)$-modules:
\begin{align*} 
\theta  \wedge \col {\cal E}\otimes_{{\cal O}_Y}\Om^{\bul}_{Y/S}[-1]\lo 
{\cal E}\otimes_{{\cal O}_Y}\Om^{\bul}_{Y/\os{\circ}{S}}. 
\tag{3.1.7}\label{ali:cfaos} 
\end{align*} 
Indeed, because we have the following equalities
\begin{align*} 
(\nabla \theta \wedge -\theta \wedge (-\nabla))(e\otimes \om)
&=(\nabla \theta \wedge+\theta \wedge \nabla)(e\otimes \om)
\tag{3.1.8}\label{ali:eom} \\
&= 
\{\nabla(e)\wedge \theta \wedge \om+e\otimes d(\theta \wedge \om)\}
+\{\theta \wedge \nabla(e) \wedge \om+e\otimes (\theta \wedge d\om)\}\\
&=\nabla(e)\wedge \theta \wedge \om+\theta \wedge \nabla(e) \wedge \om =0\\ 
&\quad (e\in {\cal E},  \om \in \Om^i_{Y/S}~(i\in {\mab N})), 
\end{align*} 
$\nabla \theta \wedge=\theta \wedge(-\nabla)$. 
Because the exact sequence (\ref{ali:cfos})  
is locally split,  
the following sequence of complexes of 
$g^{-1}({\cal O}_S)$-modules is exact$:$
\begin{align*} 
0\lo {\cal E}\otimes_{{\cal O}_Y}\Om^{\bul-1}_{Y/S}
\os{\theta \wedge }{\lo} 
{\cal E}\otimes_{{\cal O}_Y}\Om^{\bul}_{Y/\os{\circ}{S}}\lo 
{\cal E}\otimes_{{\cal O}_Y}\Om^{\bul}_{Y/S}\lo 0. 
\tag{3.1.9}\label{ali:cfeos} 
\end{align*}

%\begin{rema}\label{rema:tfm}
%The following morphisms  
%\begin{align*} 
%\theta \wedge \col {\cal E}\otimes_{{\cal O}_Y}\Om^{\bul}_{Y/\os{\circ}{S}}[-1]
%\owns e\otimes \om \lom e\otimes(\theta \wedge \om)\in   
%{\cal E}\otimes_{{\cal O}_Y}\Om^{\bul}_{Y/\os{\circ}{S}} 
%\quad (e\in {\cal E},\om \in \Om^{\bul}_{Y/\os{\circ}{S}})
%\tag{3.2.1}\label{ali:cfsee} 
%\end{align*} 
%and 
%\begin{align*} 
%\theta \wedge \col {\cal E}\otimes_{{\cal O}_Y}\Om^{\bul-1}_{Y/\os{\circ}{S}}
%\owns e\otimes \om \lom e\otimes(\theta \wedge \om)\in   
%{\cal E}\otimes_{{\cal O}_Y}\Om^{\bul}_{Y/\os{\circ}{S}} 
%\quad (e\in {\cal E},\om \in \Om^{\bul}_{Y/\os{\circ}{S}})
%\tag{3.2.2}\label{ali:cfsete} 
%\end{align*} 
%are not morphisms of complexes of $g^{-1}({\cal O}_S)$-modules 
%when $f\nabla(e)\not=0$ for a local section $e\in {\cal E}$ and 
%$f\in {\cal O}_Y$. 
%\end{rema}

\begin{defi}\label{defi:coe}
We say that ${\cal E}$ is {\it locally generated by horizontal sections} 
of $\nabla$ if 
there exist local sections $e_0,\ldots,e_m$ $(m\in {\mab N})$ 
which generate ${\cal E}$ locally as an ${\cal O}_Y$-module 
such that $\nabla(e_j)=0$ for $0\leq \forall j\leq m$. 
We say that the set $\{e_j\}_{j=0}^m$ a set of {\it local horizontal generators} of ${\cal E}$. 
\end{defi} 

\par 
We can generalize (\ref{lemm:pete}) as follows: 

\begin{lemm}\label{lemm:pte}
Assume that ${\cal E}$ is locally generated by horizontal sections. Then,  
for each $i$, the resulting sequences of 
{\rm (\ref{ali:cfeos})} by the operations $B^i$, $Z^i$ and ${\cal H}^i$ 
$(i\in {\mab Z}_{\geq 0})$ are exact. 
\end{lemm} 
\begin{proof}
First we give the elementary and direct two proof of this lemma. 
Later we give another proof it when $Y/S$ is of Cartier type 
and $({\cal E},\nabla)$ is a certain connection (see (\ref{rema:ci}) below). 
%in a certain special case (see (\ref{rema:ci0}) (2) below). 
\par 
We claim that the following sequences are exact$:$ 
\begin{align*} 
0\lo B^{i-1}({\cal E}\otimes_{{\cal O}_Y}\Om^{\bul}_{Y/S})
\os{\theta \wedge }{\lo} 
B^i({\cal E}\otimes_{{\cal O}_Y}\Om^{\bul}_{Y/\os{\circ}{S}})
\lo B^i({\cal E}\otimes_{{\cal O}_Y}\Om^{\bul}_{Y/S})\lo 0, 
\tag{3.3.1}\label{ali:cfois} 
\end{align*} 
\begin{align*} 
0\lo Z^{i-1}({\cal E}\otimes_{{\cal O}_Y}\Om^{\bul}_{Y/S})
\os{\theta \wedge }{\lo} 
Z^i({\cal E}\otimes_{{\cal O}_Y}\Om^{\bul}_{Y/\os{\circ}{S}})
\lo Z^i({\cal E}\otimes_{{\cal O}_Y}\Om^{\bul}_{Y/S})\lo 0. 
\tag{3.3.2}\label{ali:cfdeis} 
\end{align*} 
The problem is local. We give the two proofs of the exactness of (\ref{ali:cfois}). 
%We may assume that $S=(\os{\circ}{S},{\mab N}\oplus {\cal O}_S^*\lo {\cal O}_S)$, 
%where the morphism ${\mab N}\oplus {\cal O}_S^*\lo {\cal O}_S$ 
%is defined the morphisms ${\mab N}\owns 1\lom 0\in {\cal O}_S$ and 
%${\cal O}_S^*\subset {\cal O}_S$. 
%For two nonnegative integers $a$ and $b$ such that $a\leq b$,  
%set  
%\begin{equation*} 
%{\mab A}_{\os{\circ}{S}}(a,b):=
%\ul{{\rm Spec}}_{\os{\circ}{S}}
%({\cal O}_S[x_0, \ldots, x_b]/(\prod_{i=0}^ax_i)). 
%\end{equation*}  
%Let $M_{\os{\circ}{S}}(a,b)$ be the log structure on 
%${\mab A}_{\os{\circ}{S}}(a,b)$ associated to the following morphism 
%\begin{equation*} 
%{\mab N}^{{\oplus}(a+1)}\owns 
%(0, \ldots,0,\os{i}{1},0,\ldots, 0)\lom x_{i-1}\in 
%{\cal O}_{\os{\circ}{S}}[x_0, \ldots, x_b]/(\prod_{i=0}^ax_i).
%\end{equation*}   
%Let ${\mab A}_S(a,b)$ be the resulting log scheme over $S$.
%The diagonal immersion ${\mab N}\os{\sus}{\lo} {\mab N}^{{\oplus}(a+1)}$ 
%induces a morphism ${\mab A}_S(a,b)\lo S$ of log schemes. 
\par  
The first proof: 
\par 
Let $M_Y$ be the log structure of $Y$. 
Because the problem, 
we can fix a set $\{e_j\}_{j=0}^d$ of horizontal generators  
of ${\cal E}$ and fix a base 
$\{\theta,\om_1,\ldots, \om_d\}$ of  $\Om^1_{Y/\os{\circ}{S}}$ 
$(\om_i=d\log m_i$ for some $m_i\in M_Y$).  
For a nonnegative integer $k$, 
set ${\cal P}_k:=\{I\subset \{1,\ldots, d\}~\vert~\sharp I=k\}$. 
For a nonnegative integer $k$ and 
an element $I:=\{i_1,\ldots, i_k\} \in {\cal P}_k$, 
set $\om_{I}:=\om_{i_1}\wedge \cdots \wedge \om_{i_k}$, 
where $i_1<\cdots <i_k$. 
For an empty set $\phi$, set $\om_{\phi}=1\in {\cal O}_Y$.  
Any local section $e$ of ${\cal E}\otimes_{{\cal O}_Y}\Om^{i-1}_{Y/\os{\circ}{S}}$ 
can be expressed as follows: 
\begin{align*} 
e=\sum_{j=0}^m\{e_j\otimes 
(\sum_{I\in {\cal P}_{i-1}}f_I\om_{I}+
\sum_{J\in {\cal P}_{i-2}}g_J \theta \wedge \om_J)\} \quad (f_I,g_J\in {\cal O}_Y). 
\end{align*} 
Then 
\begin{align*} 
\nabla(e)=
\sum_{j=0}^m\{e_j\otimes 
(\sum_{I\in {\cal P}_{i-1}}df_I\wedge \om_{I}+
\sum_{J\in {\cal P}_{i-2}}dg_J \wedge \theta\wedge \om_{J})\}. 
\end{align*} 
Hence the following diagram is commutative:   
\begin{equation*} 
\begin{CD}
{\cal E}\otimes_{{\cal O}_Y}\Om^i_{Y/\os{\circ}{S}}
@=
\iota({\cal E}\otimes_{{\cal O}_Y}\Om^i_{Y/S})
\oplus 
\{\wedge \theta  ({\cal E}\otimes_{{\cal O}_Y}\Om^{i-1}_{Y/S})\}\\
@V{\nabla}VV @VV{\iota(\nabla_{Y/S})\oplus (\theta \wedge \nabla_{Y/S}  )}V\\
{\cal E}\otimes_{{\cal O}_Y}\Om^{i+1}_{Y/\os{\circ}{S}}
@=\{\iota({\cal E}\otimes_{{\cal O}_Y}\Om^{i+1}_{Y/S})\}
\oplus 
\{\theta \wedge  ({\cal E}\otimes_{{\cal O}_Y}\Om^i_{Y/S})\}. 
\end{CD}
\tag{3.3.3}\label{cd:kdsc}
\end{equation*}
Consequently 
\begin{align*} 
B^i({\cal E}\otimes_{{\cal O}_Y}\Om^{\bul}_{Y/\os{\circ}{S}})=
\iota(B^i({\cal E}\otimes_{{\cal O}_Y}\Om^{\bul}_{Y/S}))
\oplus \theta \wedge  (B^{i-1}({\cal E}\otimes_{{\cal O}_Y}\Om^{\bul}_{Y/S})), 
\end{align*}   
\begin{align*} 
Z^i({\cal E}\otimes_{{\cal O}_Y}\Om^{\bul}_{Y/\os{\circ}{S}})
=\iota(Z^i({\cal E}\otimes_{{\cal O}_Y}\Om^{\bul}_{Y/S}))
\oplus \theta \wedge  (Z^{i-1}({\cal E}\otimes_{{\cal O}_Y}\Om^{\bul}_{Y/S})).
\end{align*}  
This implies that the sequences (\ref{ali:cfois}) and 
(\ref{ali:cfdeis}) are exact. 
\par 
The second proof in the case where ${\cal E}$ is a flat 
quasi-coherent ${\cal O}_Y$-module:  
%(without using a local basis of $\Om^1_{Y/\os{\circ}{S}}$): 
\par 
%Consider the splitting 
%$\iota \col \Om^1_{Y/S}\lo \Om^1_{Y/\os{\circ}{S}}$ 
%$(\pi^1(\om_i)\lom \om_i)$ of $\pi^1$. 
%the inclusion 
%$\nabla({\cal E})\subset \iota({\cal E}\otimes_{{\cal O}_Y}\Om^1_{Y/S})$
%does not necessarily hold and the diagram (\ref{cd:kc}) below 
%is not necessarily commutative.
%We have to prove the exactness of 
%(\ref{ali:cfois}) and (\ref{ali:cfoeis}) directly.  
%\par 
Here we give the exactness of only (\ref{ali:cfois}). 
We can prove the exactness of (\ref{ali:cfdeis}) similarly. 
It suffices to prove that 
\begin{align*} 
{\rm Ker}(B^i({\cal E}\otimes_{{\cal O}_Y}\Om^{\bul}_{Y/\os{\circ}{S}})\lo 
B^i({\cal E}\otimes_{{\cal O}_Y}\Om^{\bul}_{Y/S}))
\subset (\theta \wedge )(B^{i-1}({\cal E}\otimes_{{\cal O}_Y}\Om^{\bul}_{Y/S})). 
\end{align*} 
This is equivalent to the following inclusion 
\begin{align*} 
\nabla ({\cal E}\otimes_{{\cal O}_Y}\Om^{i-1}_{Y/\os{\circ}{S}}) 
\cap 
(\theta \wedge ) 
({\cal E}\otimes_{{\cal O}_Y}
\Om^{i-1}_{Y/\os{\circ}{S}})) \subset
(\theta \wedge )\nabla
({\cal E}\otimes_{{\cal O}_Y}\Om^{i-2}_{Y/\os{\circ}{S}}).
\tag{3.3.4}\label{al:hoyw}\\   
\end{align*}
%Let us denote $(\theta \wedge )(\Om^j_{Y/\os{\circ}{S}})$ 
%$(j\in {\mab N})$ by $(\theta \wedge\Om^j_{Y/\os{\circ}{S}}$
To prove the inclusion (\ref{al:hoyw}), by the assumption
we have only to prove that the following inclusion holds:  
\begin{align*} 
({\cal E}\otimes_{{\cal O}_Y}d\Om^{i-1}_{Y/\os{\circ}{S}}) 
\cap 
({\cal E}\otimes_{{\cal O}_Y}
(\theta \wedge  )(\Om^{i-1}_{Y/\os{\circ}{S}}))) \subset
{\cal E}\otimes_{{\cal O}_Y}(\theta \wedge )(d\Om^{i-2}_{Y/\os{\circ}{S}}).\tag{3.3.5}\label{al:heoiyw}\\   
\end{align*}
Because ${\cal E}$ is a flat ${\cal O}_Y$-module, 
it suffices to prove that 
\begin{align*} 
d\Om^{i-1}_{Y/\os{\circ}{S}} 
\cap (\theta \wedge (\Om^{i-1}_{Y/\os{\circ}{S}})) \subset
\theta \wedge  (d\Om^{i-2}_{Y/\os{\circ}{S}}).
\tag{3.3.6}\label{al:hmyw}\\   
\end{align*}
By taking a local basis 
$\{\om_1,\ldots, \om_d,\theta\}$ of  $\Om^1_{Y/\os{\circ}{S}}$, 
by describing a local section of $\Om^{i-1}_{Y/\os{\circ}{S}}$ with the use of this basis 
and by noting that $d\theta =0$, 
(\ref{al:hmyw}) is an obvious inclusion. 
We complete the first proof of (\ref{lemm:pte}). 
\end{proof}

\begin{defi}\label{defi:tr} 
Let $h\col Z\lo T$ be a proper log smooth morphism of fine log schemes. 
Let $\nabla \col {\cal M}\lo {\cal M}\otimes_{{\cal O}_Y}\Om^1_{Z/T}$ 
be an integrable connection. 
We say that the connection $\nabla$ is {\it log ordinary} 
if $R^jh_*(B^i({\cal M}\otimes_{{\cal O}_Z}\Om^{\bul}_{Z/T}))=0$ 
for $\forall i,j\in {\mab N}$. If $\nabla$ is the derivative 
$d\col {\cal O}_Z\lo \Om^1_{Z/T}$ and if $d$ is log ordinary in the sense above, 
then we say that  $h$ or $Z/T$ is {\it log ordinary}.
\end{defi} 

The following is a first step of the proof of Hyodo's criterion for the case of 
an integrable connection.  

\begin{coro}\label{lemm:b}
Assume that ${\cal E}$ is locally generated by horizontal sections. 
The connection {\rm (\ref{ali:cfees})} 
is log ordinary if and only if 
%Consider the following conditions$\,:$
%\par {\rm (a)} 
$R^jg_*(B^i({\cal E}\otimes_{{\cal O}_Y}\Om^{\bul}_{Y/\os{\circ}{S}}))=0$ 
for $\forall i$ and $\forall j$. 
%\par 
%{\rm (b)} $R^jg_*({\cal F}\otimes_{{\cal O}_Y}B\Om^i_{Y/S})=0$ for 
%$\forall i$ and $\forall j$. 
%\par 
%Then {\rm (a)} and ${\rm (b)}$ are equivalent. 
\end{coro}
\begin{proof} 
The implication $\Lo$ follows from (\ref{lemm:pte}). 
The converse implication 
%${\rm (a)}\Lo {\rm (b)}$ 
also follows from the exact sequence (\ref{ali:cfois}) 
and induction on $i$. 
\end{proof}

\begin{rema}\label{rema:ci}
(1) If ${\cal E}$ is not locally generated by horizontal sections of $\nabla$, 
(\ref{ali:cfois}) is not necessarily exact. 
%the following diagram is not necessarily commutative for 
%$i\in {\mab N}$:   
%\begin{equation*} 
%\begin{CD}
%{\cal E}\otimes_{{\cal O}_Y}\Om^i_{Y/\os{\circ}{S}}
%@=
%\iota({\cal E}\otimes_{{\cal O}_Y}\Om^i_{Y/S})
%\oplus 
%\{\theta \wedge ({\cal E}\otimes_{{\cal O}_Y}\Om^{i-1}_{Y/S})\}\\
%@V{\nabla}VV @VV{\iota(\nabla_{Y/S})\oplus (-\theta \wedge \nabla_{Y/S})}V\\
%{\cal E}\otimes_{{\cal O}_Y}\Om^{i+1}_{Y/\os{\circ}{S}}
%@=\{\iota({\cal E}\otimes_{{\cal O}_Y}\Om^{i+1}_{Y/S})\}
%\oplus \{\theta \wedge ({\cal E}\otimes_{{\cal O}_Y}\Om^i_{Y/S})\}. 
%\end{CD}
%\tag{3.7.1}\label{cd:kc}
%\end{equation*} 
Indeed, consider the case $i=0$ and 
consider a log scheme $Y$ whose underlying scheme is 
${\rm Spec}(\kap[x,y]/(xy))$ and whose log structure is associated to a morphism 
${\mab N}^2\lo \kap[x,y]/(xy)$ defined by $(1,0)\lom x$ and $(0,1)\lom y$. 
The diagonal immersion ${\mab N}\lo {\mab N}^2$ induces a morphism 
$Y\lo s$ of log schemes. 
Let ${\cal E}$ be a free ${\cal O}_Y$-module of rank $1$ with basis $e$. 
Let $\nabla \col {\cal E}\lo {\cal E}\otimes_{{\cal O}_Y}\Om^1_{Y/\os{\circ}{S}}$ be an integrable connection defined by the following equalities: 
$\nabla(e)=e\otimes \theta$, 
$\nabla(fe)=f\nabla(e)+e\otimes df$ $(f\in {\cal O}_Y)$.  
Then $\nabla_{Y/S}(e)=0$. 
%Hence there does not a section $\iota \col \Om^1_{Y/S}\lo \Om^1_{Y/\os{\circ}{S}}$ 
%making the diagram (\ref{cd:kc}) for the case $i=0$ commutative. 
Hence the sequence (\ref{ali:cfois}) for $i=1$ is not exact. 
\par 
(2) We give the third proof of (\ref{lemm:pte}) in the following case. 
\par 
Assume that $Y/S$ is of Cartier type. 
Set $Y':=Y\times_{S,F_S}S$. 
Let $F\col Y\lo Y'$ be the relative Frobenius morphism over $S$. 
Let ${\cal E}'$ be a quasi-coherent flat ${\cal O}_{Y'}$-module.  
Set $({\cal E},\nabla):=(F^{*}({\cal E}'),{\rm id}_{{\cal E}'}\otimes d)$. 
Here $d\col {\cal O}_Y\lo \Om^1_{Y/\os{\circ}{S}}$ is the usual derivative. 
It is obvious that ${\cal E}$ is locally generated by horizontal sections of $\nabla$. 
By a special case of \cite[V (4.1.1)]{ob} and by \cite[(1.2.5)]{ofc}, 
there exists a morphism 
\begin{align*} 
C^{-1}\col {\cal E}'\otimes_{{\cal O}_Y'}\Om^i_{Y'/\os{\circ}{S}}
\lo 
F_*{\cal H}^i({\cal E}\otimes_{{\cal O}_Y}
\Om^{\bul}_{Y/\os{\circ}{S}})\quad (i\in {\mab N})
\tag{3.6.1}\label{ali:hfox}
\end{align*} 
of ${\cal O}_{Y'}$-modules fitting into the following commutative diagram:
\begin{equation*}
\begin{CD}
0 @>>> {\cal E}'\otimes_{{\cal O}_{Y'}}\Om^{i-1}_{Y'/S}
@>{\theta \wedge }>> 
{\cal E}'\otimes_{{\cal O}_{Y'}}\Om^i_{Y'/\os{\circ}{S}}@>>> 
{\cal E}'\otimes_{{\cal O}_{Y'}}\Om^i_{Y'/S}@>>> 0\\
@. @V{C^{-1}}V{\simeq}V @V{C^{-1}}VV @V{C^{-1}}V{\simeq}V \\
0 @>>> F_*{\cal H}^{i-1}({\cal E}\otimes_{{\cal O}_Y}\Om^{\bul}_{Y/S})
@>{\theta \wedge }>> 
F_*{\cal H}^i({\cal E}\otimes_{{\cal O}_Y}\Om^{\bul}_{Y/\os{\circ}{S}})@>>> 
F_*{\cal H}^i({\cal E}\otimes_{{\cal O}_Y}\Om^{\bul}_{Y/S})@>>> 0. 
\end{CD}
\tag{3.6.2}\label{ali:hfoox}
\end{equation*}   
Note that $C^{-1}(\theta)=\theta$. 
Hence the morphism (\ref{ali:hfox}) is an isomorphism. 
By \cite[XV Proposition 2 a)]{sga5-2},  
$\os{\circ}{F}$ is a homeomorphism. 
Hence the following sequence 
\begin{align*} 
0 \lo {\cal H}^{i-1}({\cal E}\otimes_{{\cal O}_Y}\Om^{\bul}_{Y/S})
\os{\theta \wedge }{\lo} 
{\cal H}^i({\cal E}\otimes_{{\cal O}_Y}\Om^{\bul}_{Y/\os{\circ}{S}})
\lo {\cal H}^i({\cal E}\otimes_{{\cal O}_Y}\Om^{\bul}_{Y/S}) \lo 0
\end{align*} 
is exact. 
Using induction on $i$, we see that 
the following sequences are exact: 
\begin{align*} 
0 \lo Z^{i-1}({\cal E}\otimes_{{\cal O}_Y}\Om^{\bul}_{Y/S})
\os{\theta \wedge }{\lo} 
Z^i({\cal E}\otimes_{{\cal O}_Y}\Om^{\bul}_{Y/\os{\circ}{S}})
\lo Z^i({\cal E}\otimes_{{\cal O}_Y}\Om^{\bul}_{Y/S}) \lo 0
\end{align*} 
and 
\begin{align*} 
0 \lo B^{i-1}({\cal E}\otimes_{{\cal O}_Y}\Om^{\bul}_{Y/S})
\os{\theta \wedge }{\lo} 
B^i({\cal E}\otimes_{{\cal O}_Y}\Om^{\bul}_{Y/\os{\circ}{S}})
\lo B^i({\cal E}\otimes_{{\cal O}_Y}\Om^{\bul}_{Y/S}) \lo 0. 
\end{align*} 
We complete the second proof of (\ref{lemm:pte}) in the case above. 
\end{rema} 

\par 
Let $Z$ be a fine log 
(formal) scheme over a fine log (formal) scheme $T$. 
Let $g\col Z\lo \os{\circ}{T}$ be the structural morphism. 
Let us recall the increasing filtration 
$P$ on the sheaf ${\Om}^i_{Z/\os{\circ}{T}}$ $(i\in {\mab N})$ 
of log differential forms on $Z_{\rm zar}$ (\cite[(1.3.0.1)]{nb}, cf.~\cite[(4.0.2)]{nh3}): 
\begin{equation*} 
P_k{\Om}^i_{Z/\os{\circ}{T}} =
\begin{cases} 
0 & (k<0), \\
{\rm Im}({\Om}^k_{Z/\os{\circ}{T}}{\otimes}_{{\cal O}_Z}
\Om^{i-k}_{\os{\circ}{Z}/\os{\circ}{T}}
\lo {\Om}^i_{Z/\os{\circ}{T}}) & (0\leq k\leq i), \\
{\Om}^i_{Z/\os{\circ}{T}} & (k > i).
\end{cases}
\tag{3.6.3}\label{eqn:pkdefpw}
\end{equation*}  
In \cite{nb} we have called this filtration $P$ 
the preweight filtration on $\Om^{\bul}_{Z/\os{\circ}{T}}$.  
For a flat quasi-coherent ${\cal O}_Z$-module ${\cal F}$, 
set 
\begin{align*} 
P_k({\cal F}\otimes_{{\cal O}_Z}\Om^i_{Z/\os{\circ}{T}})
:={\cal F}\otimes_{{\cal O}_Z}P_k\Om^i_{Z/\os{\circ}{T}} 
\quad (i\in {\mab N}, k\in {\mab Z}). 
\end{align*} 
\par 
Let ${\cal F}$ be a flat quasi-coherent ${\cal O}_Z$-module 
and let 
\begin{align*} 
\nabla \col {\cal F}\lo {\cal F}\otimes_{{\cal O}_Z}P_0\Om^1_{Z/\os{\circ}{T}}
\tag{3.6.4}\label{eqn:peefpw}
\end{align*} 
be an integrable connection. 
That is, $\nabla$ is a morphism of $g^{-1}({\cal O}_T)$-modules,  
$\nabla(a\om)=\om \otimes da+a\nabla(\om)$ $(a\in {\cal O}_Z, \om \in {\cal F})$ and  
the iteration of $\nabla$:  
$\nabla^2 \col {\cal F}\lo {\cal F}\otimes_{{\cal O}_Z}P_0\Om^2_{Z/\os{\circ}{T}}$
is zero. By abuse of notation, we denote the induced morphism  
${\cal F}\lo {\cal F}\otimes_{{\cal O}_Z}\Om^1_{Z/\os{\circ}{T}}$ 
by (\ref{eqn:peefpw}) also by 
\begin{align*} 
\nabla \col {\cal F}\lo {\cal F}\otimes_{{\cal O}_Z}\Om^1_{Z/\os{\circ}{T}}. 
\tag{3.6.5}\label{eqn:peefow}
\end{align*} 
Let 
\begin{align*} 
\nabla_{Z/T} \col {\cal F}\lo {\cal F}\otimes_{{\cal O}_Z}\Om^1_{Z/T}. 
\tag{3.6.6}\label{eqn:peefsow}
\end{align*} 
be the induced integrable connection by (\ref{eqn:peefow}). 
Then we have complexes  
${\cal F}\otimes_{{\cal O}_Z}\Om^{\bul}_{Z/\os{\circ}{T}}$ 
and 
${\cal F}\otimes_{{\cal O}_Z}\Om^{\bul}_{Z/T}$. 
Set 
\begin{align*} 
P_k({\cal F}\otimes_{{\cal O}_Z}\Om^{\bul}_{Z/\os{\circ}{T}})
:={\cal F}\otimes_{{\cal O}_Z}P_k\Om^{\bul}_{Z/\os{\circ}{T}} 
\quad (k\in {\mab Z}). 
\end{align*} 
By (\ref{eqn:peefpw}) we indeed the complex 
$P_k({\cal F}\otimes_{{\cal O}_Z}\Om^{\bul}_{Z/\os{\circ}{T}})$. 
Consequently we have a filtered complex 
$({\cal F}\otimes_{{\cal O}_Z}\Om^{\bul}_{Z/\os{\circ}{T}},P)$ of 
$g^{-1}({\cal O}_T)$-modules. 

\begin{rema}\label{rema:res}
Though $Z/\os{\circ}{T}$ is not necessarily log smooth, 
let us define a sheaf $R_{Z/\os{\circ}{T}}$ on $\os{\circ}{Z}$ as 
${\rm Coker}(\Om^1_{\os{\circ}{Z}/\os{\circ}{T}}\lo 
\Om^1_{Z/\os{\circ}{T}})$ following \cite[(1.3.0.1)]{ofc}. 
Let ${\cal F}'$ be a quasi-coherent ${\cal O}_Z$-module. 
Then it is obvious that an integrable connection 
\begin{align*} 
\nabla \col {\cal F}'\lo {\cal F}'\otimes_{{\cal O}_Z}\Om^1_{Z/\os{\circ}{T}}
\tag{3.7.1}\label{eqn:pnpw}
\end{align*}
induces a connection (\ref{eqn:peefpw}) if and only if 
the induced morphism 
\begin{align*} 
\nabla \col {\cal F}'\lo {\cal F}'\otimes_{{\cal O}_Z}R_{Z/\os{\circ}{T}}
\tag{3.7.2}\label{eqn:pnrpw}
\end{align*}
by (\ref{eqn:pnpw}) is zero. 
\end{rema}

\par 
Let $h\col Z\lo W$ be a morphism of log schemes over 
$\os{\circ}{T}$.  
Let ${\cal G}$ be a flat quasi-coherent ${\cal O}_W$-module 
and let 
\begin{align*} 
\nabla \col {\cal G}\lo {\cal G}\otimes_{{\cal O}_W}P_0\Om^1_{W/\os{\circ}{T}}
\end{align*} 
be an integrable connection fitting into the following commutative diagram 
\begin{equation*} 
\begin{CD} 
h_*({\cal F})@>{\nabla}>> h_*({\cal F}\otimes_{{\cal O}_Z}P_0\Om^1_{Z/\os{\circ}{T}})\\
@AAA @AAA \\
{\cal G}@>{\nabla}>> {\cal G}\otimes_{{\cal O}_W}P_0\Om^1_{W/\os{\circ}{T}}. 
\end{CD}
\end{equation*} 
Then we have the following morphism of filtered complexes: 
\begin{equation*} 
h^*\col ({\cal G}\otimes_{{\cal O}_W}\Om^{\bul}_{W/\os{\circ}{T}},P)
\lo 
h_*(({\cal F}\otimes_{{\cal O}_Z}\Om^{\bul}_{Z/\os{\circ}{T}},P)).  
\tag{3.7.3}\label{eqn:lyzpp}
\end{equation*}

\par 
%Let $S$ be a family of log points and  
%let  $\gam$ be a PD-strucuture 
%on a quasi-coherent ideal sheaf ${\cal I}$ of ${\cal O}_S$. 
%Assume that a prime number $p$ is locally nilpotent on $S$. 
%Set $S:=\ul{\rm Spec}_S^{\log}({\cal O}_S/{\cal I})$. 
Now let $X$ be an SNCL scheme over $S$. 
Let $\{\os{\circ}{X}_{\lam}\}_{\lam \in \Lam}$ be 
a decomposition of $\os{\circ}{X}$ 
by smooth components of $\os{\circ}{X}$ over $\os{\circ}{S}$ 
(\cite[(1.1.8)]{nb}). 
Then the set $\{\os{\circ}{X}_{\lam}\}_{\lam \in \Lam}$ gives the 
%$k$-fold zariskian 
orientation sheaf 
$\vp^{(k)}_{\rm zar}(\os{\circ}{X}/\os{\circ}{S})$ $(k\in {\mab N})$ 
as in \cite[p.~81]{nh2} (\cite[(1.1)]{nb}). 
Let $f\col X\lo S$ and 
$\os{\circ}{f}{}^{(k)} \col \os{\circ}{X}{}^{(k)}\lo \os{\circ}{S}$ 
be the structural morphisms. 
For a nonnegative integer $k$, set 
\begin{equation}
\os{\circ}{X}_{\{\lam_0, \lam_1,\ldots \lam_k\}} 
:=\os{\circ}{X}_{\lam_0}\cap \cdots \cap \os{\circ}{X}_{\lam_k} \quad 
(\lam_i \not= \lam_j~{\rm if}~i\not= j) 
\tag{3.7.4}\label{eqn:parlm}
\end{equation}
and  
\begin{equation}
\os{\circ}{X}{}^{(k)} =  
\us{\{\lam_0, \ldots,  \lam_{k}~\vert~\lam_i 
\not= \lam_j~(i\not=j)\}}{\coprod}
\os{\circ}{X}_{\{\lam_0, \lam_1, \ldots, \lam_k\}}.   
\tag{3.7.5}\label{eqn:kfdintd}
\end{equation} 
For a negative integer $k$, set  
$\os{\circ}{X}{}^{(k)}=\emptyset$.  
In \cite[(1.1.11)]{nb} we have proved that 
$\os{\circ}{X}{}^{(k)}$ is independent of 
the choice of the set $\{\os{\circ}{X}_{\lam}\}_{\lam \in \Lam}$. 
Denote the natural local closed immersion 
$\os{\circ}{X}_{\lam_0 \cdots \lam_k}\os{\sus}{\lo} 
\os{\circ}{X}$ by $a_{\lam_0\cdots \lam_k}$. 
Let $a^{(k)} \col \os{\circ}{X}{}^{(k)}\lo 
\os{\circ}{X}$ $(k\in {\mab N})$ 
be the morphism induced by the $a_{\lam_0\cdots \lam_k}$'s. 
\par

\par 
The following Poincar\'{e} residue isomorphism is 
a special case of \cite[(1.3.14)]{nb}:   

\begin{prop}\label{prop:perkl}  
Let $x$ be an exact closed point of $X$. 
Let $r$ be a nonnegative integer such that 
$M_{X,x}/{\cal O}_{X,x}^*\simeq {\mab N}^r$. 
Let $m_{1,x},\ldots, m_{r,x}$ be local sections of 
$M_X$ around $x$ 
whose images in $M_{X,x}/{\cal O}_{X,x}^*$ 
are generators of $M_{X,x}/{\cal O}_{X,x}^*$. 
Then, for a positive integer $k$, 
the following morphism  
\begin{equation*} 
{\rm Res} \col 
P_k{\Om}^{\bul}_{X/\os{\circ}{S}} \lo 
a^{(k-1)}_*(\Om^{\bul -k}_{\os{\circ}{X}{}^{(k-1)}
/\os{\circ}{S}}
\otimes_{\mab Z}\vp^{(k-1)}_{\rm zar}(\os{\circ}{X}/\os{\circ}{S}))
\tag{3.8.1}\label{eqn:mprrn}
\end{equation*} 
\begin{equation*} 
\omega d\log m_{\lam_0}\cdots d\log m_{\lam_{k-1}} 
\lom 
a^*_{\lam_0\cdots \lam_{k-1}}(\omega)
\otimes({\rm orientation}~(\lam_0\cdots \lam_{k-1})) 
\quad (\om \in P_0{\Om}^{\bul}_{X/\os{\circ}{S}})
\end{equation*} 
{\rm (cf.~\cite[(3.1.5)]{dh2})} 
induces the following ``Poincar\'{e} residue isomorphism''  
\begin{align*} 
{\rm gr}^P_k({\Om}^{\bul}_{X/\os{\circ}{S}}) 
& \os{\sim}{\lo} 
a^{(k-1)}_*(\Om^{\bul -k}_{\os{\circ}{X}{}^{(k-1)}
/\os{\circ}{S}}
\otimes_{\mab Z}\vp^{(k-1)}_{\rm zar}
(\os{\circ}{X}/\os{\circ}{S})).
\tag{3.8.2}\label{eqn:prvin}   
\end{align*} 
\end{prop} 

Let 
\begin{align*} 
\nabla \col {\cal E}\lo 
{\cal E}\otimes_{{\cal O}_X}\Om^1_{X/\os{\circ}{S}}
\tag{3.8.3}\label{ali:cops} 
\end{align*} 
be an integrable connection on $X/\os{\circ}{S}$ 
locally generated horizontal sections of $\nabla$. 
This connection induces 
the following integrable connections: 
\begin{align*} 
\nabla \col {\cal E}\lo 
{\cal E}\otimes_{{\cal O}_X}P_0\Om^1_{X/\os{\circ}{S}} 
\tag{3.8.4}\label{ali:copss} 
\end{align*} 
and 
\begin{align*} 
\nabla_{X/S} \col {\cal E}\lo 
{\cal E}\otimes_{{\cal O}_X}\Om^1_{X/S}. 
\tag{3.8.5}\label{ali:copts} 
\end{align*} 
%Here, as usual, we mean by ``the intgeral connection'' for the last morphism 
%a morphism $\nabla \col {\cal E}\lo 
%{\cal E}\otimes_{{\cal O}_X}P_0\Om^1_{X/\os{\circ}{S}}$ 
%of $f^{-1}({\cal O}_S)$-modules such that 
%$\nabla(a\om)=\om \otimes da+a\nabla(\om)$ $(a\in {\cal O}_X, \om \in {\cal E})$ and  
%the iteration of $\nabla$:  
%$\nabla^2 \col {\cal E}\lo {\cal E}\otimes_{{\cal O}_Y}P_0\Om^2_{X/\os{\circ}{S}}$
%is zero. 
In the following we always assume that ${\cal E}$ is a flat ${\cal O}_X$-module.

\begin{coro}\label{coro:fb}
$(1)$ 
%Let ${\cal E}_{\os{\circ}{X}{}^{(k-1)}}$ be the inverse image of 
%${\cal E}$ by $a^{(k-1)}$. 
The morphism {\rm (\ref{eqn:mprrn})} induces the following morphism 
\begin{equation*} 
{\rm Res} \col 
P_k({\cal E}\otimes_{{\cal O}_X}{\Om}^{\bul}_{X/\os{\circ}{S}}) \lo 
a^{(k-1)}_*(a^{(k-1)*}({\cal E})\otimes_{{\cal O}_{\os{\circ}{X}{}^{(k-1)}}}
\Om^{\bul -k}_{\os{\circ}{X}{}^{(k-1)}
/\os{\circ}{S}}
\otimes_{\mab Z}\vp^{(k-1)}_{\rm zar}(\os{\circ}{X}/\os{\circ}{S}))
\tag{3.9.1}\label{eqn:mprern}
\end{equation*} 
of complexes which induces 
the following Poincar\'{e} residue isomorphism 
\begin{align*} 
{\rm gr}^P_k({\cal E}\otimes_{{\cal O}_X}{\Om}^{\bul}_{X/\os{\circ}{S}}) 
& \os{\sim}{\lo} 
a^{(k-1)}_*(a^{(k-1)*}({\cal E})
\otimes_{{\cal O}_{\os{\circ}{X}{}^{(k-1)}}}
\Om^{\bul -k}_{\os{\circ}{X}{}^{(k-1)}
/\os{\circ}{S}}
\otimes_{\mab Z}\vp^{(k-1)}_{\rm zar}
(\os{\circ}{X}/\os{\circ}{S}))
\tag{3.9.2}\label{eqn:previn}   
\end{align*} 
of complexes. 
\par 
$(2)$  
\begin{align*}
B^i{\rm gr}^P_k({\cal E}\otimes_{{\cal O}_X}{\Om}^{\bul}_{X/\os{\circ}{S}})=
a^{(k-1)}_*(B^{i-k}(a^{(k-1)*}({\cal E})
\otimes_{{\cal O}_{\os{\circ}{X}{}^{(k-1)}}}
\Om^{\bul}_{\os{\circ}{X}{}^{(k-1)}
/\os{\circ}{S}}
\otimes_{\mab Z}\vp^{(k-1)}_{\rm zar}
(\os{\circ}{X}/\os{\circ}{S}))).  
\tag{3.9.3}\label{ali:bpk}
\end{align*} 
\end{coro}
\begin{proof} 
(1): By noting the connection (\ref{ali:cops}) 
induces the connection (\ref{ali:copss}), 
it is easy to check that  the morphism 
(\ref{eqn:mprern}) is a morphism of complexes. 
It is also easy to check that this morphism induces 
the isomorphism (\ref{eqn:previn}). 
\par 
(2): (2) follows from (1). 
\end{proof}

\par 

\par 
Next we recall the description of 
$P_0\Om^{\bul}_{X/\os{\circ}{S}}$ 
$(k\in {\mab Z})$ in \cite{nb}. 
The following has been proved in \cite[(1.3.21)]{nb} 
as a special case: 

\begin{prop}
[{\rm {\bf cf.~\cite[Lemma 3.15.1]{msemi}, 
\cite[(6.29)]{ndw}}}]\label{prop:csrl}  
The natural morphism 
$$\Om^{\bul}_{\os{\circ}{X}/\os{\circ}{S}} \lo a^{(0)}_*
(\Om^{\bul}_{\os{\circ}{X}{}^{(0)}/\os{\circ}{S}}\otimes_{\mab Z}
\vp^{(0)}_{\rm zar}(\os{\circ}{X}/\os{\circ}{S})))$$ 
induces a morphism 
$$P_0\Om^{\bul}_{X/\os{\circ}{S}}
\lo a^{(0)}_*(\Om^{\bul}_{\os{\circ}{X}{}^{(0)}/\os{\circ}{S}}\otimes_{\mab Z}
\vp^{(0)}_{\rm zar}(\os{\circ}{X}/\os{\circ}{S})))$$
and 
the following sequence 
\begin{equation*} 
0 \lo P_0\Om^{\bul}_{X/\os{\circ}{S}}
\lo a^{(0)}_*(\Om^{\bul}_{\os{\circ}{X}{}^{(0)}/\os{\circ}{S}}\otimes_{\mab Z}
\vp^{(0)}_{\rm zar}(\os{\circ}{X}/\os{\circ}{S})) 
\os{\iota^{(0)*}}{\lo} {a}{}^{(1)}_*(\Om^{\bul}_{\os{\circ}{X}{}^{(1)}/\os{\circ}{S}}
\otimes_{\mab Z}
\vp^{(1)}_{\rm zar}(\os{\circ}{X}/\os{\circ}{S}))) 
\os{\iota^{(1)*}}{\lo} \cdots  \tag{3.10.1}\label{eqn:cptle} 
\end{equation*} 
of complexes of $f^{-1}({\cal O}_S)$-modules is exact. 
Here 
$$\iota^{(k)*}\col {a}{}^{(k)}_*
(\Om^{\bul}_{\os{\circ}{X}{}^{(k)}/\os{\circ}{S}}
\otimes_{\mab Z}
\vp^{(k)}_{\rm zar}(\os{\circ}{X}/\os{\circ}{S})))
\lo {a}{}^{(k+1)}_*(\Om^{\bul}_{\os{\circ}{X}{}^{(k+1)}/\os{\circ}{S}}
\otimes_{\mab Z}
\vp^{(k+1)}_{\rm zar}(\os{\circ}{X}/\os{\circ}{S})))\quad (k\in {\mab N})$$ 
is the morphism defined in 
{\rm \cite[(1.3.20.5)]{nb}}. 
\end{prop} 

\parno
The following is not included in \cite[(1.3.21)]{nb}: 

\begin{coro}\label{coro:p0e}
The following sequence 
\begin{equation*} 
0 \lo P_0({\cal E}\otimes_{{\cal O}_X}\Om^{\bul}_{X/\os{\circ}{S}})
\lo a^{(0)}_*(a^{(0)*}({\cal E})
\otimes_{{\cal O}_{\os{\circ}{X}{}^{(0)}}}
\Om^{\bul}_{\os{\circ}{X}{}^{(0)}/\os{\circ}{S}}\otimes_{\mab Z}
\vp^{(0)}_{\rm zar}(\os{\circ}{X}/\os{\circ}{S})))
\tag{3.11.1}\label{eqn:cpetle}
\end{equation*} 
$$\os{\iota^{(0)*}}{\lo} {a}{}^{(1)}_*(a^{(1)*}({\cal E})
\otimes_{{\cal O}_{\os{\circ}{X}{}^{(1)}}}
\Om^{\bul}_{\os{\circ}{X}{}^{(1)}/\os{\circ}{S}}
\otimes_{\mab Z}
\vp^{(1)}_{\rm zar}(\os{\circ}{X}/\os{\circ}{S}))) 
\os{\iota^{(1)*}}{\lo} \cdots  $$ 
of complexes of $f^{-1}({\cal O}_S)$-modules is exact. 
\end{coro}
\begin{proof} 
This immediately follows from (\ref{prop:csrl}). 
\end{proof} 

\par 
Set 
$$\Om^{\bul}_k({\cal E}):= 
a^{(k)}_*(a^{(k)*}({\cal E})
\otimes_{{\cal O}_{\os{\circ}{X}{}^{(k)}}}
\Om^{\bul}_{\os{\circ}{X}{}^{(k)}/\os{\circ}{S}}\otimes_{\mab Z}
\vp^{(k)}_{\rm zar}(\os{\circ}{X}/\os{\circ}{S}))) 
\quad (k\in {\mab N}).$$ 
(Our numbering $k$ is different from 
the numbering for $L_k$ in \cite[p.~394]{io}.)
Let $F^{(k)} \col \os{\circ}{X}{}^{(k)}
\lo (\os{\circ}{X}{}^{(k)})'$ be the relative Fronbenius morphism 
of $\os{\circ}{X}{}^{(k)}/\os{\circ}{S}$. 
%and set 
%${\cal E}_{(\os{\circ}{X}{}^{(k)})'}:=F^{(k)*}({\cal E}_{\os{\circ}{X}{}^{(k)}})$. 
Let $F_{\os{\circ}{S}} \col \os{\circ}{S}\lo \os{\circ}{S}$ be 
the Frobenius endomorphism of $\os{\circ}{S}$ and 
set  
$X^{[p]}:=X\times_{\os{\circ}{S},F_{\os{\circ}{S}}}\os{\circ}{S}$.  
%and $(\os{\circ}{X}{}^{(k)})':=
%\os{\circ}{X}{}^{(k)}\times_{\os{\circ}{S},F_{\os{\circ}{S}}}\os{\circ}{S}$. 
In particular, $S^{[p]}:=S\times_{\os{\circ}{S},F_{\os{\circ}{S}}}\os{\circ}{S}$. 
Note that we do not consider the base change 
$X':=X\times_{S,F_S}S$ by the Frobenius endomorphsim 
$F_S\col S\lo S$ of $S$ since $X'$ is not an SNCL scheme.
Then $\os{\circ}{(X^{[p]})}{}^{(k)}=(\os{\circ}{X}{}^{(k)})'$. 
Let $a^{(k)'}\col  (\os{\circ}{X}{}^{(k)})'\lo (X^{[p]})^{\circ}$ 
be the analogous morphism to $a^{(k)}$. 
Let $F\col X\lo X^{[p]}$ be the abrelative Frobenius morphism of $X$ 
(\cite[(1.5.14)]{nb}) induced by the Frobenius endomorphism of $X$. 
Let ${\cal F}$ be a quasi-coherent flat ${\cal O}_{X^{[p]}}$-module. 
Set 
$$\Om^i_{k}({\cal F}):= 
a^{(k)'}_*(a^{(k)'*}({\cal F})
\otimes_{{\cal O}_{(\os{\circ}{X}{}^{(k)})'}}
\Om^i_{(\os{\circ}{X}{}^{(k)})'/\os{\circ}{S}}\otimes_{\mab Z}
\vp^{(k)}_{\rm zar}((X^{[p]})^{\circ}/\os{\circ}{S}))) 
\quad (k\in {\mab Z}_{\geq 0})$$  
for each $i$. 
Then the following sequence is exact: 
\begin{equation*} 
0 \lo P_0({\cal F}\otimes_{{\cal O}_{X^{[p]}}}\Om^i_{X^{[p]}/\os{\circ}{S}})
\lo \Om^i_0({\cal F}) \os{\iota'{}^{(0)}}{\lo} \Om^i_1({\cal F}) 
\os{\iota'{}^{(1)*}}{\lo} \cdots .
\tag{3.11.2}\label{eqn:cpetale} 
\end{equation*} 
Here $\iota'{}^{(k)}$ is the analogous morphism to 
$\iota^{(k)}$ for $X^{[p]}/S$. 
Consider the following four conditions (I) $\sim$ (IV):  
\par 
(I): There exists an isomorphism 
\begin{align*} 
C^{-1}_k\col  \Om^i_{k}({\cal F}) 
\os{\sim}{\lo} F_*{\cal H}^i(\Om^{\bul}_k({\cal E}))
\tag{3.11.3}\label{ali:kfk} 
\end{align*} 
of ${\cal O}_{(\os{\circ}{X}{}^{(k)})'}$-modules for any $i\in {\mab Z}_{\geq 0}$ 
and any $k\in {\mab Z}_{\geq 0}$ fitting into the following commutative diagram 
\begin{equation*} 
\begin{CD} 
F_*{\cal H}^i(\Om^{\bul}_k({\cal E}))@>>> 
F_*{\cal H}^i(\Om^{\bul}_{k+1}({\cal E}))\\
@A{C^{-1}_k}A{\simeq}A @A{\simeq}A{C^{-1}_{k+1}}A\\
\Om^i_{k}({\cal F}) @>>> \Om^i_{k+1}({\cal F}). 
\end{CD} 
\tag{3.11.4}\label{cd:kffk} 
\end{equation*} 
\par 
(II) There exists the following morphism 
\begin{align*} 
C^{-1} \col {\cal F}\otimes_{{\cal O}_{X^{[p]}}}\Om^i_{X^{[p]}/\os{\circ}{S}}
\lo F_*({\cal H}^i({\cal E}\otimes_{{\cal O}_X}\Om^{\bul}_{X/\os{\circ}{S}}))
\tag{3.11.5}\label{ali:pecin} 
\end{align*} 
of ${\cal O}_{X^{[p]}}$-modules
which induces the following morphism 
\begin{align*} 
C^{-1} \col P_k({\cal F}\otimes_{{\cal O}_{X^{[p]}}}\Om^i_{X^{[p]}/\os{\circ}{S}})
\lo F_*({\cal H}^i(P_k({\cal E}\otimes_{{\cal O}_X}\Om^{\bul}_{X/\os{\circ}{S}}))) 
\quad (k\in {\mab Z}). 
\tag{3.11.6}\label{ali:pecpn} 
\end{align*} 
\par 
(III) The following diagram 
\begin{equation*} 
\begin{CD} 
P_0({\cal F}\otimes_{{\cal O}_{X^{[p]}}}{\Om}^i_{X^{[p]}/\os{\circ}{S}}) @>>> 
\Om^i_0({\cal F})\\
@V{C^{-1}}VV @V{\simeq}V{C^{-1}_0}V\\
F_*{\cal H}^i(P_0({\cal E}\otimes_{{\cal O}_X}{\Om}^{\bul}_{X/\os{\circ}{S}}))
@>>>
F_*{\cal H}^i(\Om^i_0({\cal E}))
\end{CD} 
\tag{3.11.7}\label{cd:pzaccin} 
\end{equation*} 
is commutative.  
\par 
(IV) The following diagram is commutative for $k\in {\mab Z}_{\geq 1}$: 
\begin{equation*} 
\begin{CD} 
F_*{\cal H}^i({\rm gr}_k^P({\cal E}\otimes_{{\cal O}_X}{\Om}^{\bul}_{X/\os{\circ}{S}}))
@>{\simeq}>>
F_*{\cal H}^{i-k}(\Om^{\bul}_k({\cal E}))\\
@A{C^{-1}}A{\simeq}A @A{\simeq}A{C^{-1}_k}A\\ 
{\rm gr}_k^P({\cal F}\otimes_{{\cal O}_{X^{[p]}}}{\Om}^i_{X^{[p]}/\os{\circ}{S}}) 
@>{\simeq}>> 
\Om^{i-k}_k({\cal F}). 
\end{CD} 
\tag{3.11.8}\label{cd:pzccin} 
\end{equation*} 

\parno  
If the conditions (I) and (III) are satisfied, then 
%By (\ref{cd:kffk}) and (\ref{cd:pzaccin}), we see 
the following morphism is an isomorphism
\begin{align*} 
C^{-1} \col P_0({\cal F}\otimes_{{\cal O}_{X^{[p]}}}\Om^i_{X^{[p]}/\os{\circ}{S}})
\os{\sim}{\lo} 
F_*({\cal H}^i(P_0({\cal E}\otimes_{{\cal O}_X}\Om^{\bul}_{X/\os{\circ}{S}}))).  
\tag{3.11.9}\label{ali:pe0pn} 
\end{align*} 
 
\par 
We easily see that the morphism (\ref{ali:pecin}) is an isomorphism 
by (\ref{ali:pe0pn}) and (\ref{cd:pzccin}). 

\par 
The following is a generalization of \cite[Appendice (1.6)]{io}:

\begin{prop}\label{prop:iga} 
Assume that the condition ${\rm (I)}$ holds. 
Then the resulting sequences of 
{\rm (\ref{eqn:cpetle})} by the operations ${\cal H}^i$, $B^i$ and $Z^i$ 
$(i\in {\mab Z}_{\geq 0})$ are exact. 
\end{prop} 
\begin{proof} 
%Because the problem is local, we may assume that 
%there exists the  following Cartesian diagram: 
%\begin{equation*}
%\begin{CD} 
%X @>{\sus}>> \ol{X} \\
%@VVV @VVV \\
%{\mab A}_S(a,d) 
%@>{\sus}>> {\mab A}_{\ol{S}}(a,d) \\ 
%@VVV @VVV \\
%S@>{\sus}>> \ol{S}  
%\end{CD} 
%\tag{3.7.1}\label{cd:xwatx} 
%\end{equation*}
%where the vertical morphism 
%$\ol{X} \lo {\mab A}_{\ol{S}}(a,d)$ is solid 
%and \'{e}tale (\cite[(1.1.21)]{nb}).  
%Set $\ol{X}':=\ol{X}\times_{\os{\circ}{\ol{S}},F_{\os{\circ}{\ol{S}}}}\os{\circ}{\ol{S}}$. 
%Let $F\col \os{\circ}{X}\lo \os{\circ}{X}{}'$ be 
%the relative Frobenius morphism. 
%By using the following Cartier isomorphisms 
%\begin{align*} 
%C^{-1}\col (\Om_k')^i\lo F_{*}(\Om_k), 
%\end{align*} 
%and 
%\begin{align*} 
%C^{-1}\col \Om^i_{\ol{X}'/\os{\circ}{\ol{S}}}(\log X')(-X')
%\os{\sim}{\lo} F_*{\cal H}^i(\Om^{\bul}_{\ol{X}/\os{\circ}{\ol{S}}}(\log X)(-X)), 
%\end{align*}  
The proof of (\ref{prop:iga}) is 
%simpler than 
the same as that of \cite[Appendice (1.6)]{io}. 
Indeed, the problem is local. Hence we may assume that 
$\os{\circ}{X}$ is quasi-compact. 
For ${\cal G}:={\cal E}$ or ${\cal F}$, 
set $K^{\bul}_k({\cal G})
:={\rm Ker}(\Om^{\bul}_k({\cal G})\lo \Om^{\bul}_{k+1}({\cal G}))$. 
Then the following sequences 
$$0\lo K^{\bul}_k({\cal E})\lo \Om^{\bul}_k({\cal G})\lo K^{\bul}_{k+1}({\cal E})\lo 0$$ 
and 
$$0\lo K^i_k({\cal F})\lo \Om^i_k({\cal F})\lo 
K^i_{k+1}({\cal F})\lo 0 \quad (i\in {\mab N})$$
are exact by (\ref{eqn:cpetle}) and (\ref{eqn:cpetale}), respectively. 
Hence we have the following exact sequence 
\begin{align*} 
\cdots \lo {\cal H}^i(K^{\bul}_k({\cal E}))\lo 
{\cal H}^i(\Om^{\bul}_k({\cal E})) \lo {\cal H}^i(K^{\bul}_{k+1}({\cal E}))
\lo \cdots 
\tag{3.12.1}\label{ali:fxskhlh}
\end{align*}  
and the following commutative diagram 
\begin{equation*} 
\begin{CD} 
0@>>> F_*{\cal H}^i(K^{\bul}_k({\cal E}))@>>> 
F_*{\cal H}^i(\Om^{\bul}_k({\cal E})) @>>>
F_*{\cal H}^i(K^{\bul}_{k+1}({\cal E}))@>>> 0\\
@.@AAA @A{C^{-1}}A{\simeq}A @AAA \\
0@>>>K^i_{k}({\cal F})@>>> \Om^i_{k}({\cal F}) 
@>>> K^i_{k+1}({\cal F})@>>>0 
\end{CD}
\tag{3.12.2}\label{ali:fomklh}
\end{equation*} 
of sequences. 
Let $u^i_k$ be the left vertical morphism. 
If $k\gg 0$, then $K^{\bul}_k({\cal G})=0$ since 
$\os{\circ}{X}{}^{(k)}=\emptyset$. 
Assume that $u^i_k$ is an isomorphism for 
$\forall k\geq k_0$ for some $k_0$ and $\forall i\in {\mab Z}_{\geq 0}$. 
Because $F_*{\cal H}^i(\Om^{\bul}_k({\cal E}))\lo 
F_*{\cal H}^i(K^{\bul}_{k+1}({\cal E}))$ is surjective for $k=k_0-1$, 
the morphism 
$F_*{\cal H}^{i+1}(K^{\bul}_{k_0-1}({\cal E}))
\lo F_*{\cal H}^{i+1}(\Om^{\bul}_{k_0-1}({\cal E}))$ is injective. 
Obviously the morphism 
$F_*{\cal H}^0(K^{\bul}_{k_0-1}({\cal E}))
\lo F_*{\cal H}^0(\Om^{\bul}_{k_0-1}({\cal E}))$ is also injective. 
Hence $u^i_{k_0-1}$ $(i\in {\mab Z}_{\geq 0})$ is an isomorphism. 
Descending induction on $k$ shows that 
the upper sequence of (\ref{ali:fomklh}) is exact. 
Because $\os{\circ}{F}$ is a homeomorphism of topological spaces 
(\cite[XV Proposition 2 a)]{sga5-2}), 
the following sequence 
\begin{align*} 
0\lo {\cal H}^i(K^{\bul}_k({\cal E}))\lo {\cal H}^i(\Om^{\bul}_k({\cal E})) \lo 
{\cal H}^i(K^{\bul}_{k+1}({\cal E}))\lo 0 \quad (i,k\in {\mab Z}_{\geq 0}) 
\end{align*} 
is exact. 
Now ascending induction on $i$ shows that 
the resulting sequences of 
(\ref{eqn:cptle}) by the operations $Z^i$ and $B^i$ 
$(i\in {\mab Z}_{\geq 0})$ are exact. 
\end{proof}

\begin{prop}[{\rm {\bf cf.~\cite[Corollaire (2.5)]{io}}}]\label{prop:ris}
Assume that the conditions {\rm (I)} $\sim$ {\rm (IV)} hold. 
Then the resulting sequences of the following sequence  
by the operations ${\cal H}^i$, $B^i$ and $Z^i$ 
$(i\in {\mab Z}_{\geq 0})$ are exact$:$ 
\begin{align*} 
0\lo P_{k-1}({\cal E}\otimes_{{\cal O}_X}\Om^{\bul}_{X/\os{\circ}{S}})
\lo 
P_k({\cal E}\otimes_{{\cal O}_X}\Om^{\bul}_{X/\os{\circ}{S}})
\lo 
{\rm gr}_k^P({\cal E}\otimes_{{\cal O}_X}\Om^{\bul}_{X/\os{\circ}{S}})\lo 0.
\tag{3.13.1}\label{ali:ppe}
\end{align*} 
\end{prop}
\begin{proof}
We use the argument of (\ref{prop:iga}) again. 
We may assume that $k\geq 1$. 
We have the following exact sequence  
\begin{align*} 
\cdots \lo {\cal H}^i(P_{k-1}({\cal E}\otimes_{{\cal O}_X}\Om^{\bul}_{X/\os{\circ}{S}}))
\lo 
{\cal H}^i(P_k({\cal E}\otimes_{{\cal O}_X}\Om^{\bul}_{X/\os{\circ}{S}}))
\lo 
{\cal H}^i({\rm gr}_k^P({\cal E}\otimes_{{\cal O}_X}\Om^{\bul}_{X/\os{\circ}{S}}))
\lo \cdots. 
\end{align*} 
First consider the case $k=1$. 
By (\ref{ali:pe0pn}) and the conditions (II) and (IV), 
we obtain the following commutative diagram 
\begin{equation*} 
\begin{CD} 
0 @>>> F_*{\cal H}^i(P_0({\cal E}\otimes_{{\cal O}_X}\Om^{\bul}_{X/\os{\circ}{S}}))
@>>> 
F_*{\cal H}^i(P_1({\cal E}\otimes_{{\cal O}_X}\Om^{\bul}_{X/\os{\circ}{S}}))
@>>>\\
@. @A{C^{-1}}A{\simeq}A @A{C^{-1}}AA \\
0 @>>> P_0({\cal F}\otimes_{{\cal O}_X}\Om^i_{X^{[p]}/\os{\circ}{S}})
@>>> 
P_1({\cal F}\otimes_{{\cal O}_X}\Om^i_{X^{[p]}/\os{\circ}{S}})@>>>
\end{CD} 
\tag{3.13.2}\label{cd:pfox} 
\end{equation*}
\begin{equation*} 
\begin{CD}  
F_*{\cal H}^i({\rm gr}_1^P({\cal E}\otimes_{{\cal O}_X}\Om^{\bul}_{X/\os{\circ}{S}}))
@>>> 0\\
@A{C_1^{-1}}A{\simeq}A @. \\ 
{\rm gr}_1^P({\cal F}\otimes_{{\cal O}_X}\Om^i_{X^{[p]}/\os{\circ}{S}})
@>>> 0. 
\end{CD} 
\end{equation*}
Hence the morphism 
$F_*{\cal H}^i(P_1({\cal E}\otimes_{{\cal O}_X}\Om^{\bul}_{X/\os{\circ}{S}}))
\lo 
F_*{\cal H}^i({\rm gr}_1^P({\cal E}\otimes_{{\cal O}_X}\Om^{\bul}_{X/\os{\circ}{S}}))$
is surjective for $\forall i$. 
Consequently 
the morphism 
$F_*{\cal H}^i(P_0({\cal E}\otimes_{{\cal O}_X}\Om^{\bul}_{X/\os{\circ}{S}}))
\lo 
F_*{\cal H}^i(P_1({\cal E}\otimes_{{\cal O}_X}\Om^{\bul}_{X/\os{\circ}{S}}))$ 
is injective for $\forall i$. 
Consequently the upper horizontal sequence is exact. 
%and hence the middle vertical morphism in (\ref{cd:pfox}) is an isomorphism. 
By \cite[XV Proposition 2 a)]{sga5-2},  
$\os{\circ}{F}$ is a homeomorphism. 
Hence the following sequence 
\begin{align*} 
0 \lo {\cal H}^i(P_0({\cal E}\otimes_{{\cal O}_X}\Om^{\bul}_{X/\os{\circ}{S}}))
\lo {\cal H}^i(P_1({\cal E}\otimes_{{\cal O}_X}\Om^{\bul}_{X/\os{\circ}{S}}))
\lo {\cal H}^i({\rm gr}_1^P({\cal E}\otimes_{{\cal O}_X}\Om^{\bul}_{X/\os{\circ}{S}}))
\lo 0
\end{align*} 
is exact. 
Similarly we see that 
the following sequence 
\begin{align*} 
0 \lo {\cal H}^i(P_{k-1}({\cal E}\otimes_{{\cal O}_X}\Om^{\bul}_{X/\os{\circ}{S}}))
\lo 
{\cal H}^i(P_k({\cal E}\otimes_{{\cal O}_X}\Om^{\bul}_{X/\os{\circ}{S}}))
\lo 
{\cal H}^i({\rm gr}_k^P({\cal E}\otimes_{{\cal O}_X}\Om^{\bul}_{X/\os{\circ}{S}}))
\lo 0 
\end{align*} 
is exact for $k\geq 2$. 
The rest of the proof is the same as that of (\ref{prop:iga}). 
\end{proof}

\begin{prop}[{\rm {\bf cf.~\cite[Appendice (2.6)]{io}}}]\label{prop:ahc}
Assume that the conditions {\rm (I)} $\sim$ {\rm (IV)} hold. 
Consider the following conditions$:$ 
\par 
{\rm (a)} $R^jf_*(B\Om_k^i({\cal E}))=0$ 
for $\forall i, \forall j$ and $\forall k$. 
%\par 
%{\rm (d)} $R^jf_*(B{\rm gr}^P_n\Om^i_{X/\os{\circ}{S}})=0$ 
%for $\forall i, \forall j, \forall n$. 
\par 
{\rm (b)} $R^jf_*(B^i(P_k({\cal E}\otimes_{{\cal O}_X}\Om^{\bul}_{X/\os{\circ}{S}})))=0$ 
and $R^jf_*(B^i({\cal E}\otimes_{{\cal O}_X}\Om^{\bul}_{X/\os{\circ}{S}}))=0$ 
for $\forall i, \forall j$ and $\forall k$.
\par 
{\rm (c)} $R^jf_*(B^i({\cal E}\otimes_{{\cal O}_X}\Om^{\bul}_{X/\os{\circ}{S}}))=0$ 
for $\forall i$ and $\forall j$.
\par 
{\rm (d)} $R^jf_*(B^i({\cal E}\otimes_{{\cal O}_X}\Om^{\bul}_{X/S}))=0$ 
for $\forall i$ and $\forall j$.
\parno 
Then the following hold$:$
\par 
$(1)$ {\rm (a)} and {\rm (b)} are equivalent. 
\par 
$(2)$ {\rm (b)} implies {\rm (c)}.  
%{\rm (f)}. 
\par 
$(3)$ {\rm (c)} and {\rm (d)} are equivalent. 
\par 
Consequently {\rm (a)} implies {\rm (d)}. 
\end{prop} 
\begin{proof} 
(1): Let us prove the implication (a)$\Lo$(b). 
The sheaf 
$R^jf_*(B^i({\cal E}\otimes_{{\cal O}_X}\Om^{\bul}_{X/\os{\circ}{S}}))$
is the associated sheaf to the following presheaf 
for open sub-log schemes $U$'s of $\os{\circ}{S}$: 
$U \lom 
R^jf_{U*}(B^i({\cal E}\otimes_{{\cal O}_X}\Om^{\bul}_{X\times_{\os{\circ}{S}}U/U}))$,  
where $f_U\col X\times_{\os{\circ}{S}}U\lo U$ is the structural morphism. 
Hence, to prove that 
$R^jf_*(B^i({\cal E}\otimes_{{\cal O}_X}\Om^{\bul}_{X/\os{\circ}{S}}))=0$, we can 
assume that  
$B^i({\cal E}\otimes_{{\cal O}_X}\Om^{\bul}_{X/\os{\circ}{S}})
=P_kB^i({\cal E}\otimes_{{\cal O}_X}\Om^{\bul}_{X/\os{\circ}{S}})$ 
for some $k\in {\mab N}$. 
We can also assume that $\os{\circ}{X}$ is quasi-compact. 
Now we obtain the implication (a)$\Lo$(b) 
by (\ref{prop:iga}), (\ref{ali:bpk}) and (\ref{prop:ris}).  
The converse implication follows from (\ref{ali:bpk}) and (\ref{prop:ris}). 
\par 
(2):  Obvious.
\par 
(3):  (3) follows from (\ref{lemm:b}).  
\end{proof} 

\par 
In the following, we give examples of 
the integrable connections $({\cal E},\nabla)$'s  
which satisfy the conditions (I) $\sim$ (IV). 
%The following tells us that, in the case (\ref{exem:triv}), 
%there exists a morphism (\ref{ali:pecin}) satisfying 
%(\ref{ali:pecpn}), (\ref{cd:pzaccin}) and  (\ref{cd:pzccin}). 

\begin{prop}\label{prop:ci}
$(1)$ $\Om^i_{X'/S}=\Om^i_{X^{[p]}/S^{[p]}}$ $(i\in {\mab N})$.
\par 
$(2)$ There exists a morphism 
\begin{align*} 
C^{-1}\col \Om^i_{X^{[p]}/\os{\circ}{S}}
\lo
F_*{\cal H}^i(\Om^{\bul}_{X/\os{\circ}{S}})\quad (i\in {\mab N})
\tag{3.15.1}\label{ali:fhox}
\end{align*} 
of ${\cal O}_{X^{[p]}}$-modules fitting into 
the following commutative diagram for $i\in {\mab N}:$
\begin{equation*}
\begin{CD}
0 @>>> \Om^{i-1}_{X^{[p]}/S^{[p]}}@>{\theta \wedge}>> 
\Om^i_{X^{[p]}/\os{\circ}{S}}@>>> \Om^i_{X^{[p]}/S^{[p]}}@>>> 0\\
@. @V{C^{-1}}V{\simeq}V @V{C^{-1}}VV @V{C^{-1}}V{\simeq}V \\
0 @>>> F_*{\cal H}^{i-1}(\Om^{\bul}_{X/S})@>{\theta \wedge}>> 
F_*{\cal H}^i(\Om^{\bul}_{X/\os{\circ}{S}})@>>> 
F_*{\cal H}^i(\Om^{\bul}_{X/S})@>>> 0. 
\end{CD}
\tag{3.15.2}\label{ali:fhoox}
\end{equation*}  
Here 
\begin{align*} 
C^{-1}\col \Om^i_{X^{[p]}/S^{[p]}}=\Om^i_{X'/S}\os{\sim}{\lo}
F_*{\cal H}^i(\Om^{\bul}_{X/S})
\tag{3.15.3}\label{ali:fcxoox}
\end{align*} 
is the log inverse Cartier isomorphism due to Kato {\rm (\cite[(4.12)]{klog1})}. 
Consequently the morphism {\rm (\ref{ali:fhox})} is an isomorphism. 
\par 
$(3)$ 
The isomorphism {\rm (\ref{ali:fhox})} 
induces the following isomorphism 
\begin{align*} 
C^{-1}\col P_k\Om^i_{X^{[p]}/\os{\circ}{S}}
\os{\sim}{\lo}
F_*{\cal H}^i(P_k\Om^{\bul}_{X/\os{\circ}{S}})\quad (i,k\in {\mab N}).  
\tag{3.15.4}\label{ali:fhhox}
\end{align*} 
\par 
$(4)$ The following diagram 
\begin{equation*} 
\begin{CD} 
{\rm gr}^P_k({\Om}^i_{X^{[p]}/\os{\circ}{S}}) 
@>{{\rm Res},~\sim}>> 
a^{(k-1)}_*(\Om^{i-k}_{\os{\circ}{X}{}^{[p]}{}^{(k-1)}
/\os{\circ}{S}}
\otimes_{\mab Z}\vp^{(k-1)}_{\rm zar}
(\os{\circ}{X}{}'/\os{\circ}{S}))\\
@V{C^{-1}}V{\simeq}V @V{\simeq}V{C^{-1}}V\\
F_*{\cal H}^i({\rm gr}^P_k({\Om}^{\bul}_{X/\os{\circ}{S}})) 
@>{{\rm Res},~\sim}>>
F_*a^{(k-1)}_*
{\cal H}^i(\Om^{\bul -k}_{\os{\circ}{X}{}^{(k-1)}
/\os{\circ}{S}}
\otimes_{\mab Z}\vp^{(k-1)}_{\rm zar}
(\os{\circ}{X}/\os{\circ}{S})).  
\end{CD} 
\tag{3.15.5}\label{eqn:pprrvin} 
\end{equation*} 
is commutative for $k\in {\mab Z}_{\geq 1}$. 
\par  
$(5)$ 
The following diagram 
\begin{equation*} 
\begin{CD} 
P_0{\Om}^i_{X^{[p]}/\os{\circ}{S}} @>>> 
a^{(0)}_*(\Om^{i}_{\os{\circ}{X}{}^{[p]}{}^{(0)}/\os{\circ}{S}}
\otimes_{\mab Z}\vp^{(k-1)}_{\rm zar}
(\os{\circ}{X}{}'/\os{\circ}{S}))\\
@V{C^{-1}}V{\simeq}V @V{\simeq}V{C^{-1}}V\\
F_*{\cal H}^i(P_0{\Om}^{\bul}_{X/\os{\circ}{S}})
@>>>
F_*a^{(0)}_*
{\cal H}^i(\Om^{\bul}_{\os{\circ}{X}{}^{(0)}
/\os{\circ}{S}}\otimes_{\mab Z}\vp^{(0)}_{\rm zar}
(\os{\circ}{X}/\os{\circ}{S}))  
\end{CD} 
\tag{3.15.6}\label{eqn:pzvin} 
\end{equation*} 
is commutative.  
\par 
Consequently the isomorphism {\rm (\ref{ali:fhox})} 
satisfies the conditions {\rm (I)} $\sim$ {\rm (IV)} for 
the case ${\cal E}={\cal O}_X$ 
and ${\cal F}={\cal O}_{X^{[p]}}$. 
\end{prop}
\begin{proof} 
(1):  Because the morphism $X\lo S$ is integral, 
$(X')^{\circ}=(X^{[p]})^{\circ}$. (1) follows from \cite[(1.7)]{klog1}. 
\par 
(2):  The existence of (\ref{ali:fhox}) is a special case of \cite[V (4.1.1)]{ob}.  
The commutativity of (\ref{ali:fhoox}) follows from [loc.~cit.] and \cite[(4.12)]{klog1}. 
$($Note that $C^{-1}(\theta)=\theta)$.
\par 
(3), (4), (5):  By the characterization of $C^{-1}$ in \cite[V (4.1.1)]{ob}, 
the morphism (\ref{ali:fhox}) induces the following morphism 
\begin{align*} 
C^{-1}\col P_k\Om^i_{X^{[p]}/\os{\circ}{S}}
\lo 
F_*{\cal H}^i(P_k\Om^{\bul}_{X/\os{\circ}{S}}) \quad (k\in {\mab Z}). 
\end{align*} 
We also have the commutative diagram (\ref{eqn:pprrvin}). 
Hence decsending induction on $k$ shows that 
the morphism above is an isomorphism. 
We also have the commutative diagram 
(\ref{eqn:pzvin}).   
\end{proof}

\begin{exem}\label{exem:triv}
Let ${\cal E}'$ be a quasi-coherent flat 
${\cal O}_{X^{[p]}}$-module.  
%(\ref{eqn:peefpw}) on $X^{[p]}/S^{[p]}$. 
Let $F\col X\lo X^{[p]}$  be the induced morphism of 
the Frobenius endomorphism of $X$. 
%(X^{[p]})^{\circ}\lo \os{\circ}{X}$ be the projection. 
Set $({\cal E},\nabla):=(F^{*}({\cal E}'),{\rm id}_{{\cal E}'}\otimes d)$. 
%Set $X^{(p)}:=X\times_{S,F_S}S$. 
%Note that $\os{\circ}{X}{}^{(p)}=\os{\circ}{X}{}^{[p]}$. 
By (\ref{prop:ci}) (1) and \cite[(1.2.5)]{ofc}
we have the following isomorphism 
\begin{align*} 
C^{-1}\col {\cal E}'\otimes_{{\cal O}_{X^{[p]}}}\Om^i_{X^{[p]}/S}
\os{\sim}{\lo} F_*({\cal H}^i({\cal E}\otimes_{{\cal O}_{X^{[p]}}}\Om^i_{X^{[p]}/S})) 
\quad (i\in {\mab N}). 
\end{align*} 
Hence, by the same proof as that of (\ref{prop:ci}), 
we have the following isomorphism 
\begin{align*} 
C^{-1}\col 
{\cal E}'\otimes_{{\cal O}_{X^{[p]}}}\Om^i_{X^{[p]}/\os{\circ}{S}}
\os{\sim}{\lo} 
F_*({\cal H}^i({\cal E}
\otimes_{{\cal O}_{X^{[p]}}}
\Om^i_{X^{[p]}/\os{\circ}{S}}))  
\quad (i\in {\mab N}) 
\tag{3.16.1}\label{ali:fhe}
\end{align*} 
of ${\cal O}_{X^{[p]}}$-modules fitting into the following commutative diagram 
for $i\in {\mab N}:$
\begin{equation*}
\begin{CD}
0 @>>> {\cal E}'\otimes_{{\cal O}_{X^{[p]}}}\Om^{i-1}_{X^{[p]}/S^{[p]}}@>{\theta \wedge}>> 
{\cal E}'\otimes_{{\cal O}_{X^{[p]}}}\Om^i_{X^{[p]}/\os{\circ}{S}}@>>> 
{\cal E}'\otimes_{{\cal O}_{X^{[p]}}}\Om^i_{X^{[p]}/S^{[p]}}@>>> 0\\
@. @V{C^{-1}}V{\simeq}V @V{C^{-1}}V{\simeq}V @V{C^{-1}}V{\simeq}V \\
0 @>>> F_*{\cal H}^{i-1}({\cal E}\otimes_{{\cal O}_{X}}\Om^{\bul}_{X/S})
@>{\theta \wedge}>> 
F_*{\cal H}^i({\cal E}\otimes_{{\cal O}_{X}}\Om^{\bul}_{X/\os{\circ}{S}})@>>> 
F_*{\cal H}^i({\cal E}\otimes_{{\cal O}_{X}}\Om^{\bul}_{X/S})@>>> 0. 
\end{CD}
%\tag{3.13.2}\label{ali:fhoxox}
\end{equation*}  
One can check that the isomorphism (\ref{ali:fhe}) 
satisfies the conditions (I) $\sim$ (IV) without difficulty. 
%In particular, if ${\cal F}={\cal O}_{X^{[p]}}$, 
%the classical inverse Cartier isomorphism gives us 
%the isomorphism (\ref{ali:kfk}). 
\par 
Note that the log $p$-curvature of $({\cal E},\nabla)$ in \cite{ofc} is zero.
%:=(F^{*}({\cal E}'),{\rm id}_{{\cal E}'}\otimes d)$. 
\end{exem}

\begin{rema}\label{rema:ci0}
(1) The inverse Cartier morphism 
$C^{-1} \col \Om^i_{\os{\circ}{X}/\os{\circ}{S}}\lo 
F_*({\cal H}^i(\Om^{\bul}_{\os{\circ}{X}/\os{\circ}{S}}))$ 
fits into the following commutative diagram 
\begin{equation*} 
\begin{CD}
\Om^i_{\os{\circ}{X}{}'/\os{\circ}{S}}
@>>>P_0\Om^i_{X'/\os{\circ}{S}} \\
@V{C^{-1}}VV @VV{C^{-1}}V\\
F_*(\Om^i_{\os{\circ}{X}/\os{\circ}{S}})@>>> 
F_*{\cal H}^i(P_0\Om^{\bul}_{X/\os{\circ}{S}}).
\end{CD}
\tag{3.18.1}\label{cd:remmb} 
\end{equation*} 
\par 
(2)  By using the log inverse Cartier isomorphism (\ref{ali:fhox}), 
the commutative diagram (\ref{ali:fhoox}) and 
the argument of the proof of (\ref{prop:iga}), 
we can also obtain (\ref{lemm:pte}) for $Y=X$ and 
$({\cal E},\nabla)$ in (\ref{exem:triv}). 
\end{rema} 

%\par 
%Let the notations be as in {\rm (\ref{exem:triv}) (2)}. 
%To prove that there exists a morphism {\rm (\ref{ali:pecin})} satisfying 
%{\rm (\ref{ali:pecpn})}, {\rm (\ref{cd:pzaccin})} and {\rm (\ref{cd:pzccin})} in this case, 
%we recall two results in {nb}. 
%\par 
%The following is a very special case of \cite[(2.3.19)]{nb}: 

%\begin{coro}\label{coro:cfi}
%Let the notations be as in {\rm (\ref{exem:triv})}. 
%Then there exists a canonical filtered isomorphism 
%\begin{equation*} 
%C^{-1} \col ({\cal E}'\otimes_{{\cal O}_{X^{[p]}}}\Om^i_{X^{[p]}/\os{\circ}{S}},P) 
%\os{\sim}{\lo} 
%(F_*{\cal H}^i({\cal E}\otimes_{{\cal O}_X}\Om^{*}_{X/\os{\circ}{S}}),P) 
%\quad (i\in {\mab Z}). 
%\tag{3.13.1}\label{eqn:ywony}
%\end{equation*}  
%\end{coro} 
%\begin{proof} 
%By using (\ref{coro:fb}), the same proof as that of (\ref{prop:ci}) works.  
%\end{proof} 

%\begin{prop}\label{prop:cdif}
%Let the notations be as in {\rm (\ref{exem:triv}) (3)}. 
%Then there exists a morphism {\rm (\ref{ali:pecin})} satisfying 
%{\rm (\ref{ali:pecpn})}, {\rm (\ref{cd:pzaccin})} and {\rm (\ref{cd:pzccin})}. 
%\end{prop} 
%\begin{proof}
%By the definitions in {\rm (\ref{exem:triv}) (3)}, there is nothing to prove. 
%\end{proof} 

\par 
Now we can give the second proof of (\ref{theo:twols}). 
\par 
Replace $S$ by $S_{\os{\circ}{T}}:=S\times_{\os{\circ}{S}}{\os{\circ}{T}}$. 
Set $X_{\os{\circ}{T}}:=X\times_SS_{\os{\circ}{T}}$. 
Let $f_{\os{\circ}{T}}\col X_{\os{\circ}{T}}\lo S_{\os{\circ}{T}}$ 
be the structural morphism. 
By the assumption, 
the condition (a) in (\ref{prop:ahc}) for ${\cal E}={\cal O}_X$ is satisfied. 
Hence $R^jf_{\os{\circ}{T}*}
(B\Om^i_{X_{\os{\circ}{T}}/S_{\os{\circ}{T}}})=0$ 
for $\forall i$ and $\forall j$. 
By (\ref{ali:ftbx}), $R^jf_{T*}(B\Om^i_{X_T/T})=0$. 
More generally, we obtain the following: 

\begin{theo}\label{theo:otf}
Let $({\cal E},\nabla)$ be the integrable connection 
{\rm (\ref{ali:cops})} satisfying the conditions {\rm (I)} $\sim$ {\rm (IV)}. 
Let $p_T\col X_T\lo X$ be the projection. 
If $a^{(k)*}({\cal E},\nabla)$ for $\forall k\in {\mab N}$ is ordinary, 
then the induced connection 
$$p_T^*({\cal E})\lo p_T^*({\cal E})\otimes_{{\cal O}_{X_T}}\Om^1_{X_T/T}$$ 
by $({\cal E},\nabla_{X/S})$ is log ordinary. 
\end{theo}
\begin{proof} 
One has to use only (\ref{prop:ahc}) and the argument before 
(\ref{theo:otf}). 
\end{proof} 

\par 
In the following we give another proof of the existence of 
the log inverse Cartier isomorphism satisfying the conditions (I)$\sim$(IV) 
for a unit root $F$-crystal on $\os{\circ}{X}/{\cal W}(\os{\circ}{S})$ 
when $\os{\circ}{S}$ is a perfect scheme of characteristic $p>0$. 
\par 
Assume that $\os{\circ}{S}$ is a perfect scheme of characteristic $p>0$. 

Let $E$ be a unit root $F$-crystal of 
${\cal O}_{\os{\circ}{X}/{\cal W}(\os{\circ}{S})}$-module. 
Set ${\cal E}:=E_{\os{\circ}{X}}$ and ${\cal F}:=q^*({\cal E})$. 
Then, by \cite[(3.4.1)]{et}, we have the isomorphism (\ref{ali:kfk}).  
In the following we review a log version of Etesse's results 
(\cite[II (3.2.3), (3.3.1), (3.4.1), (4.2.1)]{et}) 
proved in \cite{nb}. 
\par 
Let $Y/S$ be a log smooth scheme. 
Let $\star$ be $'$ or nothing. 
%Let $({\cal W}_n\wt{\Om}^{\bul}_Y)^{\star}$ be the complex defined in 
%\cite[p.~311]{msemi} and \cite[(2.2.0.1)]{nb}. 
Let $E$ be a unit root $F$-crystal of 
${\cal O}_{\os{\circ}{Y}/{\cal W}(\os{\circ}{S})}$-module. 
Set $E_n:=E_{({\cal W}_n(Y),V({\cal O}_{{\cal W}_n(Y)}), [~])}$. 
Let $E_n\otimes_{{\cal W}_n({\cal O}_{Y})^{\star}}({\cal W}_n\wt{\Om}^{\bul}_{Y})^{\star}$
be the complex defined in \cite[(2.2.11.10)]{nb}. 
Let $F\col {\cal W}_n(Y)\lo {\cal W}_n(Y)$ be the Frobenius endomorphism. 
Though we have assumed that $\os{\circ}{S}={\rm Spec}(\kap)$ in [loc.~cit.], 
the same argument as that in [loc.~cit.] works for the case 
where  $\os{\circ}{S}$ is a perfect scheme of characteristic $p>0$.

%\begin{prop}[{\rm {\bf \cite[(2.2.14)]{nb}}}] 
%Let $S$ be as in {\rm (\ref{exem:triv}) (2)}. 
%Let $Y/S$ be a log smooth scheme. 
%\begin{align*} 
%& {\rm Fil}^n(E_{n+r}
%\otimes_{{\cal W}_{n+r}({\cal O}_{Y})^{\star}}
%({\cal W}_{n+r}\wt{\Om}^i_{Y})^{\star})
%= \tag{3.13.1}\label{eqn:fyau}\\
%&V^n(E_r\otimes_{{\cal W}_r({\cal O}_{Y})^{\star}}
%({\cal W}_r\wt{\Om}^i_{Y})^{\star})+
%\nabla V^n(E_r\otimes_{{\cal W}_r({\cal O}_{Y})^{\star}}
%({\cal W}_r\wt{\Om}^{i-1}_{Y})^{\star}).
%\end{align*} 
%\end{prop} 

\begin{prop}[{\rm {\bf \cite[(2.2.15)]{nb}}}] 
The following formula holds$:$ 
\begin{equation*} 
F^r({\rm Fil}^n(E_{n+r}
\otimes_{{\cal W}_{n+r}({\cal O}_{Y})^{\star}}
({\cal W}_{n+r}\wt{\Om}^i_{Y})^{\star}))= 
F^r\nabla V^n(E_r
\otimes_{{\cal W}_r({\cal O}_{Y})^{\star}}
({\cal W}_r\wt{\Om}^{i-1}_{Y})^{\star}). 
\tag{3.19.1}\label{eqn:frfyu}
\end{equation*} 
Consequently the morphism 
$F^r \col E_{n+r}
\otimes_{{\cal W}_{n+r}({\cal O}_{Y})^{\star}}
({\cal W}_{n+r}\wt{\Om}^i_{Y})^{\star}\lo 
F^r_*\{E_n\otimes_{{\cal W}_n({\cal O}_{Y})^{\star}}
({\cal W}_n\wt{\Om}^i_{Y})^{\star}\}$ 
of ${\cal W}_n({\cal O}_Y)^{\star}$-modules 
induces the following morphism 
\begin{align*} 
\check{F}{}^r \col & 
E_n\otimes_{{\cal W}_n({\cal O}_{Y})^{\star}}
({\cal W}_n\wt{\Om}^i_{Y})^{\star}\lo \tag{3.19.2}\label{eqn:fsyu}\\
& F^r_*\{F^r(E_{n+r}
\otimes_{{\cal W}_{n+r}({\cal O}_{Y})^{\star}}
({\cal W}_{n+r}\wt{\Om}^i_{Y})^{\star})
/F^r\nabla V^n(E_r
\otimes_{{\cal W}_r({\cal O}_{Y})^{\star}}
({\cal W}_r\wt{\Om}^{i-1}_{Y})^{\star})\}
\end{align*} 
of ${\cal W}_n({\cal O}_Y)$-modules. 
For the case $r=n$, $\check{F}{}^n$ induces the following morphism
\begin{align*}
\check{F}{}^n \col 
E_n\otimes_{{\cal W}_n({\cal O}_{Y})^{\star}}
({\cal W}_n\wt{\Om}^i_{Y})^{\star} \lo 
F^n_*{\cal H}^i(E_n
\otimes_{{\cal W}_n({\cal O}_{Y})^{\star}}
({\cal W}_n\wt{\Om}^{\bul}_{Y})^{\star})
\tag{3.19.3}\label{eqn:fyru}
\end{align*}
of ${\cal W}_n({\cal O}_Y)^{\star}$-modules. 
\end{prop}

\begin{theo}[{\rm {\bf \cite[(2.2.16)]{nb}}}] \label{theo:cin}  
%Let $E$ be a flat coherent log $F$-crystal 
%on $Y/{\cal W}(s)$. 
%Assume that $E$ is a unit root $F$-crystal. 
%Then 
%Let $S$ be as in {\rm (\ref{exem:triuv}) (2)}. 
The morphism $V^r\col E_n
\otimes_{{\cal W}_n({\cal O}_{Y})^{\star}}
({\cal W}_n\wt{\Om}^i_{Y})^{\star} \lo 
E_{n+r}\otimes_{{\cal W}_{n+r}({\cal O}_{Y})^{\star}}
({\cal W}_{n+r}\wt{\Om}^i_Y)^{\star}$ 
of ${\cal W}_n({\cal O}_Y)^{\star}$-modules 
induces the following morphism 
\begin{align*} 
\check{V}{}^r \col & 
F^r_*\{E_n
\otimes_{{\cal W}_n({\cal O}_{Y})^{\star}}
({\cal W}_n\wt{\Om}^i_{Y})^{\star}/
F^r\nabla V^n(E_r\otimes_{{\cal W}_r({\cal O}_{Y})^{\star}}
({\cal W}_r\wt{\Om}^{i-1}_{Y})^{\star})\} 
\tag{3.20.1}\label{eqn:fyu}\\
& \lo 
E_{n+r}
\otimes_{{\cal W}_{n+r}({\cal O}_{Y})^{\star}}
({\cal W}_{n+r}\wt{\Om}^i_{Y})^{\star}. 
\end{align*}
\parno 
There exists a generalized Cartier isomorphism 
\begin{align*} 
\check{C}{}^r\col & 
F^r_*\{F^r(E_{n+r}
\otimes_{{\cal W}_{n+r}({\cal O}_{Y})^{\star}}
({\cal W}_{n+r}\wt{\Om}^i_{Y})^{\star})/
F^r\nabla V^n(E_r
\otimes_{{\cal W}_r({\cal O}_{Y})^{\star}}
({\cal W}_r\wt{\Om}^{i-1}_{Y})^{\star})\} \\
& \os{\sim}{\lo}  
E_n
\otimes_{{\cal W}_n({\cal O}_{Y})^{\star}}
({\cal W}_n\wt{\Om}^i_{Y})^{\star} 
\tag{3.20.2}\label{eqn:wlfoyu}
\end{align*} 
of ${\cal W}_n({\cal O}_Y)^{\star}$-modules,  
which is the inverse of $\check{F}{}^r$. 
The morphism  $\check{C}{}^r$ satisfies a relation 
${\bf p}^r\circ \check{C}{}^r=\check{V}{}^r$. 
In particular, there exist the following isomorphisms 
\begin{align*} 
\check{F}{}^n \col 
E_n
\otimes_{{\cal W}_n({\cal O}_{Y})^{\star}}
({\cal W}_n\wt{\Om}^i_{Y})^{\star}  
\os{\sim}{\lo} 
F^n_*({\cal H}^i(E_n
\otimes_{{\cal W}_n({\cal O}_{Y})^{\star}}
({\cal W}_n\wt{\Om}^{\bul}_{Y})^{\star})) 
\tag{3.20.3}\label{eqn:wlfyu}
\end{align*}  
and 
\begin{align*} 
\check{C}{}^n \col 
F^n_*({\cal H}^i(E_n
\otimes_{{\cal W}_n({\cal O}_{Y})^{\star}}
({\cal W}_n\wt{\Om}^{\bul}_{Y})^{\star})) 
\os{\sim}{\lo} 
E_n
\otimes_{{\cal W}_n({\cal O}_{Y})^{\star}}
({\cal W}_n\wt{\Om}^i_{Y})^{\star}   
\tag{3.20.4}\label{eqn:wlcyu}
\end{align*}  
of ${\cal W}_n({\cal O}_Y)^{\star}$-modules,  
which are the inverse of another. 
\end{theo}

\begin{prop}\label{prop:cif}
%Let the notations be as in {\rm (\ref{exem:triuv}) (2)}. 
There exist morphisms {\rm (\ref{ali:kfk})} and {\rm (\ref{ali:pecin})} satisfying 
{\rm I}, {\rm II}, {\rm III} and {\rm IV} for the case ${\cal E}={\cal F}=E_1$. 
\end{prop} 
\begin{proof}
%Let ${\cal Y}$ be a local lift of $Y$ over ${\cal W}_{n+r}(S)$. 
%The morphism $F^r\col {\cal W}_{n+r}\wt{\Om}^{\bul}_{Y})\lo 
%{\cal W}_n\wt{\Om}^{\bul}_Y)$ induced by the natural inclusion morphism 
%${\cal X}\otimes_{{\cal W}_{n+r}(S)}{\cal W}_n(S)\os{\sus}{\lo} {\cal X}$, 
%it is obvious that the isomorphism 
%\begin{align*} 
%\check{F}{}^n \col 
%E_1
%\otimes_{{\cal W}_1({\cal O}_{Y})^{\star}}
%({\cal W}_1{\Om}^i_{Y})^{\star} 
%\os{\sim}{\lo} 
%F_*({\cal H}^i(E_1\otimes_{{\cal W}_1({\cal O}_{Y})^{\star}}
%({\cal W}_1{\Om}^{\bul}_{Y})^{\star})) 
%\tag{3.23.1}\label{eqn:wlcynu}
%\end{align*}
%induces the following morphism 
%\begin{align*} 
%\check{F} \col P_k(E_1\otimes_{{\cal W}_1({\cal O}_{Y})}
%({\cal W}_1{\Om}^i_{Y})) \os{\sim}{\lo} 
%F^n_*({\cal H}^i(P_k(E_1\otimes_{{\cal W}_1({\cal O}_{Y})}
%({\cal W}_1{\Om}^{\bul}_{Y})^{\star}))). 
%\tag{3.23.2}\label{eqn:wlcaynu}
%\end{align*}
We define an isomorphism 
\begin{align*} 
C^{-1}\col  {\cal E}\otimes_{{\cal O}_X}\Om^i_{X}
\os{\sim}{\lo} 
F_*({\cal H}^i({\cal E}
\otimes_{{\cal O}_X}{\Om}^{\bul}_{X})) 
\tag{3.21.1}\label{eqn:wlcybnu}
\end{align*} 
of ${\cal O}_Y$-modules as the following composite isomorphism: 
%making the following diagram commutative: 
%\begin{equation*} 
%\begin{CD} 
%\check{F}{}^n \col E_1\otimes_{{\cal W}_1({\cal O}_{Y})^{\star}}
%({\cal W}_1{\Om}^i_{Y})^{\star} \os{\sim}{\lo} 
%F^n_*({\cal H}^i(E_1
%\otimes_{{\cal W}_1({\cal O}_{Y})^{\star}}
%({\cal W}_1{\Om}^{\bul}_{Y})^{\star})) 
%\end{CD} 
%\end{equation*} 
\begin{align*} 
C^{-1}\col  {\cal E}\otimes_{{\cal O}_X}\Om^i_{X/\os{\circ}{S}}
\os{{\rm id}\otimes C^{-1}}{\lo} 
{\cal E}\otimes_{{\cal O}_X}F_*({\cal H}^i(\Om^{\bul}_{X/\os{\circ}{S}}))
\os{\check{F}}{\lo} 
F_*({\cal H}^i({\cal E}
\otimes_{{\cal O}_X}{\Om}^{\bul}_{X})) 
\end{align*} 
This composite morphism preserve the filtration $P$. 
Because $\check{F}$ is the induced morphism by ``${\rm id}$'', 
the commutativity of the diagrams 
{\rm (\ref{cd:pzaccin})} and {\rm (\ref{cd:pzccin})} are obvious. 
\end{proof} 
 
\par 
In the rest of this section, we give the open analogue of (\ref{theo:otf}). 
%\begin{defi}\label{defi:sdd}
%Let the notations be as in (\ref{theo:hco}). 
%\end{defi} 

\begin{lemm}\label{lemm:opanc}
Let the notations be as in {\rm (\ref{theo:hco})}. 
Let ${\cal E}$ be a quasi-coherent flat ${\cal O}_X$-module. 
Let $\nabla \col {\cal E}\lo {\cal E}\otimes_{{\cal O}_X}
\Om^1_{(X,D)/S}$ be an integrable connection. 
$($Note that $\Om^1_{(X,D)/S}=\Om^1_{X/S}(\log D))$. 
For $k\in {\mab Z}_{\geq 0}$, 
let $D^{(k)}$ be the closed subscheme of $X$ defined in 
{\rm \cite[(2.2.13.2)]{nh2}} 
and let  $a^{(k)}\col D^{(k)}\lo X$ 
be the natural morphism. 
Assume that ${\cal E}$ is locally generated by horizontal sections of $\nabla$. 
Set $P_k({\cal E}\otimes_{{\cal O}_X}{\Om}^{\bul}_{(X,D)/S})
:={\cal E}\otimes_{{\cal O}_X}P_k\Om^{\bul}_{(X,D)/S}$. 
Set 
$$\Om^i_{k}({\cal E}):= 
a^{(k)}_*(a^{(k)*}({\cal E})
\otimes_{{\cal O}_{D^{(k)}}}
\Om^i_{D^{(k)}/S}\otimes_{\mab Z}\vp^{(k)}_{\rm zar}(D/S)) 
\quad (k\in {\mab Z}_{\geq 0})$$  
for each $i$. 
Then the following hold$:$ 
\par 
$(1)$ There exists the following Poincar\'e residue isomorphism of complexes$:$ 
\begin{align*} 
{\rm gr}^P_k({\cal E}\otimes_{{\cal O}_X}{\Om}^{\bul}_{(X,D)/S})  
& \os{\sim}{\lo} 
\Om^{\bul -k}_{k}({\cal E}). 
\tag{3.22.1}\label{eqn:peovin}   
\end{align*} 
\par 
$(2)$ The following sequence is exact$:$
\begin{align*} 
0 \lo {\cal E}\otimes_{{\cal O}_X}\Om^{\bul}_{(X,D)/S}(-D)
\lo \Om^{\bul}_{0}({\cal E})
\os{\iota^{(0)*}}{\lo} \Om^{\bul}_1({\cal E})
\os{\iota^{(1)*}}{\lo} \cdots  
\tag{3.22.2}\label{eqn:povin}   
\end{align*} 
of complexes of $f^{-1}({\cal O}_S)$-modules is exact. 
\end{lemm}
\begin{proof}
(1): By \cite[(2.2.21.3)]{nh2} 
we have the following Poincar\'{e} residue isomorphism: 
\begin{align*} 
{\rm Res}:  
{\rm gr}_k^{P}(\Om^{\bul}_{(X,D)/S}) 
\os{\sim}{\lo} a_*^{(k)}(\Om^{\bul-k}_{D^{(k)}/S}
\otimes_{\mab Z}\vp^{(k)}_{\rm zar}(D/S)). 
\end{align*} 
(\ref{eqn:peovin}) immediately follows from this isomorphism. 
\par 
(2): By \cite[(4.2.2) (a), (c)]{di} 
the sequence (\ref{eqn:povin}) for the case ${\cal E}={\cal O}_X$ is exact. 
Now we see that the sequence (\ref{eqn:povin}) for the general case is exact. 
\end{proof}

\begin{defi}\label{defi:eox}
We say that $({\cal E},\nabla)$ is log ordinary with compactly support 
if $R^jf_*(B^i({\cal E}\otimes_{{\cal O}_X}\Om^{\bul}_{(X,D)/S}(-D)))=0$ for 
any $i\in {\mab Z}_{\geq 0}$ and $j\in {\mab Z}_{\geq 0}$.  
\end{defi}

The following (1) is the open analogue of (\ref{theo:otf}). 

\begin{theo}\label{theo:opp}
Let the notations be as in {\rm (\ref{theo:hco})}. 
Let $F\col (X,D)\lo (X',D')$ be 
the relative Frobenius morphism over $S$. 
Let ${\cal E}'$ be a quasi-coherent ${\cal O}_{X'}$-module. 
%Let $\star$ be nothing or $'$. 
For $k\in {\mab Z}_{\geq 0}$, 
let  $a'{}^{(k)}\col D'{}^{(k)}\lo X'$ 
be the natural morphism. 
Set 
$$\Om^i_{k}({\cal E}'):= 
a'{}^{(k)}_*(a'{}^{(k)*}({\cal E}')
\otimes_{{\cal O}_{D'{}^{(k)}}}
\Om^i_{D'{}^{(k)}/S}\otimes_{\mab Z}\vp^{(k)}_{\rm zar}(D'/S)) 
\quad (k\in {\mab Z}_{\geq 0})$$  
for each $i$. 
Assume that the following two conditions {\rm (I)} and {\rm (II)} 
are satisfied$:$  
\par 
{\rm (I)}: There exists an isomorphism 
\begin{align*} 
C^{-1}_k\col  \Om^i_{k}({\cal E}') 
\os{\sim}{\lo} F_*{\cal H}^i(\Om^{\bul}_k({\cal E}))
\tag{3.24.1}\label{ali:kfko} 
\end{align*} 
of ${\cal O}_{D^{'(k)}}$-modules 
for any $i\in {\mab Z}_{\geq 0}$ 
and any $k\in {\mab Z}_{\geq 0}$ 
fitting into the following commutative diagram 
\begin{equation*} 
\begin{CD} 
F_*{\cal H}^i(\Om^{\bul}_k({\cal E}))@>>> 
F_*{\cal H}^i(\Om^{\bul}_{k+1}({\cal E}))\\
@A{C^{-1}_k}A{\simeq}A @A{\simeq}A{C^{-1}_{k+1}}A\\
\Om^i_{k}({\cal E}') @>>> \Om^i_{k+1}({\cal E}'). 
\end{CD} 
\tag{3.24.2}\label{cd:kffok} 
\end{equation*} 
\par 
{\rm (II)} The following diagram is commutative for $k\in {\mab Z}_{\geq 0}$: 
\begin{equation*} 
\begin{CD} 
F_*{\cal H}^i({\rm gr}_k^P({\cal E}\otimes_{{\cal O}_X}{\Om}^{\bul}_{(X,D)/S}))
@>{\simeq}>>
F_*{\cal H}^{i-k}(\Om^{\bul}_k({\cal E}))\\
@A{C^{-1}}A{\simeq}A @A{\simeq}A{C^{-1}_k}A\\ 
{\rm gr}_k^P({\cal E}'\otimes_{{\cal O}_{X'}}{\Om}^i_{(X',D')/S}) 
@>{\simeq}>> 
\Om^{i-k}_k({\cal E}'). 
\end{CD} 
\tag{3.24.3}\label{cd:pzcocin} 
\end{equation*} 
Then the following hold$:$
\par 
$(1)$ If $a^{(k)*}({\cal E},\nabla)$ for $\forall k\in {\mab N}$ is ordinary, 
then $({\cal E},\nabla)$ is log ordinary. 
\par 
$(2)$ If $a^{(k)*}({\cal E},\nabla)$ for $\forall k\in {\mab N}$ is ordinary, 
then $({\cal E},\nabla)$ is log ordinary with compact support. 
\end{theo}
\begin{proof} 
(1): Because the proof of (1) by using (\ref{lemm:opanc}) (1) 
is easier than that of (\ref{theo:otf}), 
we leave the detail of the proof to the reader. 
\par 
(2): Because the proof of (2) by using (\ref{lemm:opanc}) (2)
is easier than that of (\ref{theo:otf}), 
we leave the detail of the proof to the reader. 
\end{proof}

\section{The third proof of (\ref{theo:twols})}\label{sec:htc}
In this section we give the proof of (\ref{theo:twols}) 
without using the $p$-adic weight spectral sequence (\ref{ali:infhwpwt}) 
nor the filtration $P$ on $\Om^{\bul}_{X/\os{\circ}{S}}$. 
Instead we use the $p$-adic weight spectral sequence for 
${\cal W}\wt{\Om}^i_X$ in the case $S=s$. 
To obtain this spectral sequence, we use the filtration $P$ on 
${\cal W}\wt{\Om}^i_X$. 
\par 
First assume that $S=s$. 
Then we have proved the following in \cite{nb} 
as a very special case of  \cite[(2.3.37)]{nb}.

\begin{theo}\label{theo:pwt}
Let $i$ be a nonnegative integer. 
Set 
\begin{equation*} 
E_1^{-k,h+k}=
\begin{cases} 
0 & (k<0), \\
H^h(X,P_0{\cal W}\wt{\Om}^i_{X}) & (k=0), \\
H^{h-k}(\os{\circ}{X}{}^{(k-1)}, 
{\cal W}\Om^{i-k}_{\os{\circ}{X}{}^{(k-1)}/\kap}\otimes_{\mab Z}
\vp^{(k-1)}_{\rm zar}(\os{\circ}{X}/\kap))(-k) & (k > 0).
\end{cases} 
\end{equation*}  
%\begin{equation*} 
%E_1^{-k,h+k} := 
%H^{q}(X,{\rm gr}^P_k{\cal W}\wt{\Om}^{\bul}_X)\\ 
%\end{equation*}
Then there exists the following spectral sequence 
\begin{align*} 
E_1^{-k,h+k}\Lo 
H^h(X,{\cal W}\wt{\Om}^{i}_X). 
\tag{4.1.1}\label{ali:xnt}
\end{align*} 
There also exists the following spectral sequence 
\begin{align*}
E_1^{k,h-k}=H^{h-k}(\os{\circ}{X}{}^{(k-1)},
{\cal W}\Om^{i}_{\os{\circ}{X}{}^{(k-1)}} 
\otimes_{\mab Z}
\vp^{(k-1)}_{\rm zar}(\os{\circ}{X}/\kap))
\Lo 
H^h(X,P_0{\cal W}\wt{\Om}^i_{X}).
\tag{4.1.2}\label{ali:xp0nt}
\end{align*}
The spectral sequences  {\rm (\ref{ali:xnt})} and {\rm (\ref{ali:xp0nt})} 
are compatible with the operator $F$. 
\end{theo}

\begin{prop}[{\bf \cite[(3.4)]{nlfc}}]\label{prop:fag}
The obvious analogue of {\rm (\ref{prop:fzg})} 
holds for ${\cal W}\wt{\Om}^i_{X}$. 
\end{prop} 

\par 
Now we give the third proof of (\ref{theo:twols}) as follows.
\par  
First assume that $S=s$. 
By (\ref{theo:pwt}) the operator 
$F\col H^j(X,{\cal W}\wt{\Om}^i_{X})\lo H^j(X,{\cal W}\wt{\Om}^i_{X})$ 
is bijective for any $i$ and $j$. 
By (\ref{prop:fag}) the operator  
$F\col H^j(X,{\cal W}\wt{\Om}^i_{X})\lo H^j(X,{\cal W}\wt{\Om}^i_{X})$ 
is bijective for any $i$ and $j$ if and only if 
$H^j(X,B{\cal W}\wt{\Om}{}^i_{X})=0$ (cf.~\cite[IV (4.13)]{ir}). 
Let $R$ be the 
Cartier-Dieudonn\'{e}-Raynaud algebra of $\kap$ (\cite{ir}). 
Set $R_n:=R/(V^nR+dV^nR)$.
Then the following three facts hold by \cite[(6.21.1), (6.16.1), (6.27)]{ndw}: 
\par 
(i)  ${\rm Im}(F^n\col {\cal W}_{2n}\wt{\Om}^i_X\lo {\cal W}_n\wt{\Om}^i_X)
={\rm Ker}(d\col {\cal W}_n\wt{\Om}^i_X\lo {\cal W}_n\wt{\Om}^i_X)$, 
\par 
(ii) $d^{-1}(p^n{\cal W}\wt{\Om}^{i+1}_X)=F^n{\cal W}\wt{\Om}^i_X$,
\par 
(iii)  $R_n\otimes_R{\cal W}\wt{\Om}^{\bul}_X={\cal W}_n\wt{\Om}^{\bul}_X$.
\par 
As noted in the proof of \cite[(4.1)]{lodw} 
(for the case ${\cal W}_n\Om^{\bul}_Y$ 
for a log smooth scheme over a fine log scheme $Y$ whose underlying scheme is 
${\rm Spec}(\kap)$), 
these imply that the following sequence 
\begin{align*} 
0&\lo  H^j(X,{\cal W}\wt{\Om}^{i-1}_X)/(F^n+V^n)
H^j(X, {\cal W}\wt{\Om}^{i-1}_X)\os{d}{\lo}
H^j(X,B{\cal W}_n\wt{\Om}^i_X)\\
&\lo (V^n)^{-1}F^n 
H^{j+1}(X, {\cal W}\wt{\Om}^{i-1}_X)/
F^nH^{j+1}(X, {\cal W}\wt{\Om}^{i-1}_X)\lo 0.  
\end{align*} 
is exact by the argument of the log version of \cite[IV (4.13)]{ir}. 
Hence $H^j(X,B{\cal W}_n\wt{\Om}{}^i_{X})=0$ for 
any $i$, $j$  and $n$. 
In particular, 
$H^j(X,B{\cal W}_1\wt{\Om}{}^i_{X})=0$ for 
any $i$ and $j$. 
This is equivalent to the vanishing of $H^j(X,B\Om^i_{X/\os{\circ}{s}})=0$. 
By (\ref{lemm:pte}), 
$H^j(X,B\Om^i_{X/s})=0$ for 
any $i$ and $j$. 
\par 
In the case of the general $S$, 
the rest of the proof is the same as the proof after (\ref{prop:bc}).  

\par  
Let $n$ be a positive integer. 
Let $Y/s$ be a log smooth scheme. 
By \cite[(2.2.3.1)]{nb} we have the following exact sequence:  
\begin{align*} 
0\lo {\cal W}_n\Om^{\bul}_Y[-1]
\os{\theta \wedge}{\lo} 
{\cal W}_n\wt{\Om}^{\bul}_Y
\lo {\cal W}_n\Om^{\bul}_Y\lo 0. 
\tag{4.2.1}\label{ali:cfownwis} 
\end{align*} 
Because this exact sequence is compatible with projections, 
we have the following exact sequence: 
\begin{align*} 
0\lo {\cal W}\Om^{\bul}_Y[-1]
\os{\theta \wedge}{\lo} 
{\cal W}\wt{\Om}^{\bul}_Y
\lo {\cal W}\Om^{\bul}_Y\lo 0. 
\tag{4.2.2}\label{ali:cfonnwis} 
\end{align*} 
One can generalize (\ref{lemm:pete}) in the case $S=s$ as follows, 
which is of independent interest: 

\begin{prop}\label{prop:pwte}
For each $i$, the resulting sequences of 
{\rm (\ref{ali:cfownwis})} by the operations $B^i$, $Z^i$ and ${\cal H}^i$, 
$(i\in {\mab Z}_{\geq 0})$ are exact. 
Consequently the resulting sequences of 
{\rm (\ref{ali:cfonnwis})} by the operations 
$B^i$,  $Z^i$ and ${\cal H}^i$ 
$(i\in {\mab Z}_{\geq 0})$ are exact. 
\end{prop} 
\begin{proof}
This is a local problem. 
We may assume 
%that there exists a solid and \'{e}tale morphism 
%$X\lo {\mab A}_s(a,d)$ and we fix this morphism. 
%Hence we may assume 
that there exists a log smooth lift 
${\cal Y}$ of $Y$ over ${\cal W}_{2n}(s)$. 
%with solid and \'{e}tale morphism 
%${\cal X}\lo {\mab A}_{{\cal W}_{2n}(s)}(a,d)$ 
%fitting into the following lift 
%\begin{equation*} 
%\begin{CD}
%X @>{\subset}>>{\cal X}\\
%@VVV @VVV \\
%{\mab A}_{s}(a,d) @>{\subset}>>
%{\mab A}_{{\cal W}_{2n}(s)}(a,d). 
%\end{CD} 
%\end{equation*} 
%We fix the morphism 
%${\cal X}\lo {\mab A}_{{\cal W}_{2n}(s)}(a,d)$. 
Because we have the following commutative diagram 
\begin{equation*} 
\begin{CD}
@. 0 @.  0@. 0\\
@. @VVV @VVV @VVV \\
0@>>> \Om^{\bul}_{{\cal Y}/{\cal W}_n(s)}[-1]
@>{\theta \wedge}>> \Om^{\bul}_{{\cal Y}/{\cal W}_n}
@>>> \Om^{\bul}_{{\cal Y}/{\cal W}_n(s)}@>>> 0\\
@. @V{p^n}VV @V{p^n}VV @V{p^n}VV \\
0@>>> \Om^{\bul}_{{\cal Y}/{\cal W}_{2n}(s)}[-1]
@>{\theta \wedge}>> \Om^{\bul}_{{\cal Y}/{\cal W}_{2n}}
@>>> \Om^{\bul}_{{\cal Y}/{\cal W}_{2n}(s)}@>>> 0\\
@. @VVV @VVV @VVV \\
0@>>> \Om^{\bul}_{{\cal Y}/{\cal W}_n(s)}[-1]
@>{\theta \wedge}>> \Om^{\bul}_{{\cal Y}/{\cal W}_n}
@>>> \Om^{\bul}_{{\cal Y}/{\cal W}_n(s)}@>>> 0\\
@. @VVV @VVV @VVV \\
@. 0 @.  0@. 0, 
\end{CD}
%\tag{4.3.1}\label{cd:kac}
\end{equation*}
we have the following commutative diagram 
of exact sequences:  
\begin{equation*} 
\begin{CD}
\cdots @>>> {\cal H}^{i-1}(\Om^{\bul}_{{\cal Y}/{\cal W}_{n}(s)})
@>{\theta \wedge}>> {\cal H}^i(\Om^{\bul}_{{\cal Y}/{\cal W}_{n}})
@>>> {\cal H}^i(\Om^{\bul}_{{\cal Y}/{\cal W}_{n}(s)})\lo \cdots \\
@. @V{p^{-n}d}VV @V{p^{-n}d}VV @V{p^{-n}d}VV \\
\cdots @>>> {\cal H}^{i}(\Om^{\bul}_{{\cal Y}/{\cal W}_{n}(s)})
@>{\theta \wedge}>> {\cal H}^{i+1}(\Om^{\bul}_{{\cal Y}/{\cal W}_{n}})
@>>> {\cal H}^{i+1}(\Om^{\bul}_{{\cal Y}/{\cal W}_{n}(s)})\lo \cdots.  
\end{CD}
%\tag{4.3.1}\label{cd:kac}
\end{equation*}
We may also assume that there exists a splitting 
$\iota \col \Om^{\bul}_{{\cal Y}/{\cal W}_{2n}(s)}\lo 
\Om^{\bul}_{{\cal Y}/{\cal W}_{2n}}$ of 
the projection 
$\Om^{\bul}_{{\cal Y}/{\cal W}_{2n}}
\lo \Om^{\bul}_{{\cal Y}/{\cal W}_{2n}(s)}$ 
as in the proof of (\ref{lemm:pte}). 
Hence the following sequence 
\begin{align*} 
0\lo {\cal H}^{i-1}(\Om^{\bul}_{{\cal Y}/{\cal W}_m(s)})
\lo {\cal H}^i(\Om^{\bul}_{{\cal Y}/{\cal W}_m})
\lo {\cal H}^i(\Om^{\bul}_{{\cal Y}/{\cal W}_m(s)})\lo 0 
\quad (m=n, 2n)
\end{align*} 
is split. 
Hence the following diagram is commutative: 
\begin{equation*} 
\begin{CD}
{\cal H}^{i}(\Om^{\bul}_{{\cal X}/{\cal W}_n})
@=\iota({\cal H}^{i}(\Om^{\bul}_{{\cal X}/{\cal W}_n(s)}))\oplus 
\theta \wedge (\iota({\cal H}^{i-1}(\Om^{\bul}_{{\cal X}/{\cal W}_n(s)})))\\
@V{p^{-n}d}VV @VV{\iota(p^{-n}d)\oplus 
\theta \wedge \iota(p^{-n}d)}V\\
{\cal H}^{i+1}(\Om^{\bul}_{{\cal X}/{\cal W}_n})
@=\iota({\cal H}^{i+1}(\Om^{\bul}_{{\cal X}/{\cal W}_n(s)}))\oplus 
\theta \wedge (\iota({\cal H}^{i}(\Om^{\bul}_{{\cal X}/{\cal W}_n(s)}))). 
\end{CD}
\tag{4.3.1}\label{cd:kac}
\end{equation*}
Consequently 
$$Z{\cal H}^{i}(\Om^{\bul}_{{\cal X}/{\cal W}_{n}})=
\iota(Z{\cal H}^{i}(\Om^{\bul}_{{\cal X}/{\cal W}_{n}(s)}))\oplus 
\theta \wedge(Z{\cal H}^{i-1}(\Om^{\bul}_{{\cal X}/{\cal W}_{n}(s)})),$$  
$$B{\cal H}^{i}(\Om^{\bul}_{{\cal X}/{\cal W}_{n}})=
\iota(B{\cal H}^{i}(\Om^{\bul}_{{\cal X}/{\cal W}_{n}(s)}))\oplus 
\theta \wedge (B{\cal H}^{i-1}(\Om^{\bul}_{{\cal X}/{\cal W}_{n}(s)}))$$ 
and 
$${\cal H}^{i}({\cal H}^{\bul}(\Om^*_{{\cal X}/{\cal W}_{n}}))=
\iota({\cal H}^{i}({\cal H}^{\bul}(\Om^*_{{\cal X}/{\cal W}_{n}(s)})))\oplus 
\theta \wedge({\cal H}^{i-1}({\cal H}^{\bul}(\Om^*_{{\cal X}/{\cal W}_{n}(s)}))).$$  
\end{proof}

\begin{rema}\label{rema:secpf}
We can give another proof of (\ref{prop:pwte}) by using the following 
inverse log Cartier isomorphisms (cf.~\cite[(11.1)]{ndw}) 
and by using the argument in the proof of (\ref{prop:ris}): 
\begin{align*}
C^{-n}\col {\cal W}_n\Om^i_X\os{\sim}{\lo} {\cal H}^i({\cal W}_n\Om^{\bul}_X)
\end{align*}
and 
\begin{align*}
C^{-n}\col {\cal W}_n\wt{\Om}^i_X\os{\sim}{\lo} {\cal H}^i({\cal W}_n\wt{\Om}^{\bul}_X) 
\end{align*}
fitting into the following commutative diagram: 
\begin{equation*}
\begin{CD}
0 @>>> {\cal W}_n{\Om}_Y^{i-1} @>{\theta  \wedge}>>  
{\cal W}_n\wt{{\Om}}^i_Y @>>> {\cal W}_n{\Om}_Y^i@>>> 0\\ 
@. @V{C^{-n}}V{\simeq}V  
@V{C^{-n}}V{\simeq}V @V{C^{-n}}V{\simeq}V \\
0 @>>> {\cal H}^{i-1}({\cal W}_n{\Om}_Y^{\bul})
@>{\theta \wedge}>> 
{\cal H}^i({\cal W}_n\wt{\Om}^{\bul}_Y) @>>> 
{\cal H}^i({\cal W}_n{{\Om}}_Y^{\bul})@>>> 0. 
\end{CD}
%\tag{2.2.10.2;$n$}\label{cd:orexttlam}
\end{equation*}
\end{rema}

\section{Lower semi-continuity of log genera}\label{sec:llss} 
In this section we give applications of 
the finite length version of 
the spectral sequence (\ref{ali:infhwpwt}). 
This has been proved in \cite{ndw}:

\begin{theo}[{\bf \cite[(4.1.1;${\bf n}$)]{ndw}}]\label{theo:hwawt}
Let $i$ be a nonnegative fixed integer. 
Then there exists the following spectral sequence$:$ 
\begin{align*}
E_1^{-k, h+k}&=
\us{j\geq {\rm max}\{-k,0\}}{\bigoplus}
H^{h-i-j}(\os{\circ}{X}{}^{(2j+k)},
{\cal W}_n\Om^{i-j-k}_{\os{\circ}{X}{}^{(2j+k)}})
(-j-k)
\tag{5.1.1}\label{ali:hwnpwta}\\
&\Lo H^{h-i}(X,{\cal W}_n\Om_X^i) \quad (h\in {\mab N}). 
\end{align*}  
\end{theo}

\par 
Let $X$ be a proper log smooth scheme over $s$ of pure dimension $d$. 
Let $K_0$ be the fraction field of ${\cal W}$. 
%Next we define the (resp.\  {\it log}) {\it plurigenus of level} $n$ $(\, 
%n \in {\mab N}_{>0}\,)$ as follows:
%$$p_g(Y;n;r):={\rm length}_{W_{n}}
%(H^0(Y, \os{r}{\us{W_n({\cal O}_{Y})}{\otimes}}
%W_n\Om^d_{Y})),$$
Let $({\cal T},{\cal A})$ be a ringed topos. 
For an ${\cal A}$-module ${\cal F}$ of ${\cal T}$ 
and for a positive integer $r$, 
denote  
$\underset{r~{\rm pieces}}{\underbrace{{\cal F}\otimes_{\cal A} 
\cdots \otimes_{\cal A}{\cal F}}}$ by 
$\os{r}{\us{{\cal A}}{\otimes}}{\cal F}$. 
%${\cal F}^{\os{r}{\us{\cal A}{\otimes}}}$. 

\begin{defi}\label{defi:pgssn}  
(1) (cf.~\cite[\S11.2]{ii}) We call 
\begin{align*} 
p_g(X/s,n,r):={\rm length}_{{\cal W}_n}
(H^0(X,\os{r}{\us{{\cal W}_n({\cal O}_{X})}{\otimes}}
{\cal W}_n\Om^d_X))
\tag{5.2.1}\label{ali:gxnr} 
\end{align*} 
the {\it log plurigenus} of $X/s$ of level $n$. 
When $n=1$, we call $p_g(X/s,n,r)$ the {\it log plurigenus} of $X/s$.  
When $r=1$, we call $p_g(X/s,n,r)$ the 
{\it log genus} of $X/s$ of level $n$. 
When $r=1$ and $n=1$, we call $p_g(X/s,n,r)$ the 
{\it log genus} of $X/s$. 
%Denote $p_g(X/s,n)$ simply by $p_g(X/s,n,r)$. 
\par  
(2) (the ``vertical'' log version of the Iitaka-Kodaira dimension defined in \cite{ii2} 
and \cite[\S10.5, \S11.2]{ii})
Set 
\begin{align*} 
\kap(X/s,n):=
\us{r\lo \infty}{\varlimsup}\dfrac{\log p_g(X,n,r)}{{\log r}}.
\end{align*} 
We call $\kap(X/s,n)$ the {\it log Iitaka-Kodaira dimension} of $X/s$ 
of level $n$. 
We call $\kap(X/s,1)$ the {\it log Iitaka-Kodaira dimension} of $X/s$ and 
we denote it by $\kap(X/s)$.  
\par 
(3) We call 
\begin{align*} 
p_g(X/{\cal W}(s),r):={\rm rank}_{{\cal W}}
(H^0(X,\os{r}{\us{{\cal W}({\cal O}_{X})}{\otimes}}{\cal W}\Om^d_X))
\tag{5.2.2}\label{ali:gxrnr} 
\end{align*} 
the {\it log Witt plurigenus} of $X/s$. 
%of level $\infty$. 
When $r=1$, we call $p_g(X/{\cal W}(s),r)$ the 
{\it log Witt genus} of $X/s$. 
\par 
(4) 
Set 
\begin{align*} 
\kap(X/{\cal W}(s)):=
\us{r\lo \infty}{\varlimsup}\dfrac{\log p_g(X,\infty,r)}{{\log r}}.
\end{align*} 
We call $\kap(X/{\cal W}(s))$ 
the {\it log Witt-Iitaka-Kodaira dimension} of $X/{\cal W}(s)$.  
\end{defi} 

First we check that $p_g(X/{\cal W}(s),r)\not= \infty$. 
To prove this, we recall the following:

\begin{prop}[{\bf \cite[(3.10)]{nlfc}}]\label{prop:fabzg}
Let $t$ be a fine log scheme whose underlying scheme is ${\rm Spec}(\kap)$. 
Let $Z/t$ be a log smooth scheme of Cartier type.  
Let $F_{{\cal W}_n(Z)} \col {\cal W}_n(Z) \lo {\cal W}_n(Z)$ be 
the Frobenius endomorphism of ${\cal W}_n(Z)$. 
Then the following hold$:$
\par 
$(1)$ The following sequence 
\begin{align*} 
F_{{\cal W}_n(Z)*}({\cal W}_n\Om^i_Z)
\os{V}{\lo} {\cal W}_{n+1}\Om^i_Z
\os{F^n}{\lo} 
%F^n_{Z*}(Z_n{\mathfrak W}_1\Om^{i+1})
Z_n{\cal W}_1\Om^i_Z
\lo 0
\tag{5.3.1}\label{ali:mvee}
\end{align*} 
is an exact sequence of 
${\cal W}_{n+1}({\cal O}_Z)'$-modules. 
\par 
$(2)$  
The following sequence  
\begin{align*} 
{\cal W}_{n+1}\Om^i_Z
\os{F}{\lo} F_{{\cal W}_n(Z)*}({\cal W}_n\Om^i_Z)
\os{F_{{\cal W}_n(Z)*}(F^{n-1}d)}{\lo} 
%F_{Z*}^n(B_n{\mathfrak W}_1\Om^{i+1})
B_n{\cal W}_1\Om^{i+1}_Z\lo 0
\tag{5.3.2}\label{ali:mafee}
\end{align*}
is an exact sequence of 
${\cal W}_{n+1}({\cal O}_Z)'$-modules. 
\end{prop}

The following is a log version of \cite[II (2.16)]{idw}: 

\begin{coro}\label{coro:ijd}
Let $i,j$ be nonnegative integers. 
If $\dim_{\kap}H^j(Y,Z_n\Om^i_{Y/t})$ $(n\in {\mab Z}_{\geq 0})$ is bounded, 
then $\dim_{\kap}H^j(Y,{\cal W}\Om^i_Y)/VH^j(Y,{\cal W}\Om^i_Y)<\infty$. 
Furthermore, $\dim_{\kap}H^j(Y,B_n\Om^{i+1}_{Y/t})$ 
$(n\in {\mab Z}_{\geq 0})$ is bounded, 
then $H^j(Y,{\cal W}\Om^i_Y)/VH^j(Y,{\cal W}\Om^i_Y)$ is a finitely generated 
${\cal W}$-module. 
\end{coro}
\begin{proof} 
The proof is the same as that of \cite[II (2.16)]{idw} by using (\ref{prop:fabzg}). 
\end{proof} 

\begin{coro}\label{coro:fg}
Let $i$ be a nonnegative integer. 
Then $H^0(Y,{\cal W}\Om^i_Y)$ is a free ${\cal W}$-module of finite type. 
\end{coro} 
\begin{proof} 
The proof is the same as that of \cite[II (2.17)]{idw} by using (\ref{coro:ijd}). 
\end{proof}

\par
Consider the spectral sequence (\ref{ali:hwnpwta}) for the case $i=d$: 
\begin{align*}
E_1^{-k, h+k}&=
\us{j\geq {\rm max}\{-k,0\}}{\bigoplus}
H^{h-d-j}(\os{\circ}{X}{}^{(2j+k)},
{\cal W}_n\Om^{d-j-k}_{\os{\circ}{X}{}^{(2j+k)}})
(-j-k)
\tag{5.5.1}\label{ali:hwnpwt}\\
&\Lo H^{h-d}(X,{\cal W}_n\Om_X^d) \quad (h\in {\mab N}). 
\end{align*}  

Let $\Gam(\os{\circ}{X})$ be 
the dual graph of the simple normal crossing vareity
$\os{\circ}{X}/\kap$.

\begin{theo}\label{theo:pgg} 
$(1)$ 
The following inequalities hold$:$ 
\begin{align*} 
p_g(\os{\circ}{X}{}^{(0)}/\kap,n,1)\leq 
p_g(X/s,n,1) &\leq p_g(\os{\circ}{X}{}^{(0)}/\kap,n,1)+ 
& \! \! \! \! \! \! \! \! \! \! \! 
\sum_{k=1}^{d}{\rm length}_{{\cal W}_{n}}
{\rm Ker}(H^0(\os{\circ}{X}{}^{(k)},
{\cal W}_n\Om_{\os{\circ}{X}{}^{(k)}}^{d-k})(-k) 
\tag{5.6.1}\label{ali:pgx0n}\\
& &\lo 
H^1(\os{\circ}{X}{}^{(k-1)},{\cal W}_n\Om_{\os{\circ}{X}{}^{(k-1)}}^{d-k+1})(-k+1)) \\
& \leq \sum_{s=0}^{d}p_g(\os{\circ}{X}{}^{(s)}/\kap,n,1).  
\end{align*}  
\par
$(2)$ Let $n$ be a positive integer.
If $\dim \os{\circ}{X}=1$, then 
$p_g(X/s,n,1)=p_g(\os{\circ}{X}{}^{(0)}/\kap,n,1)+{\rm 
length}_{{\cal W}_{n}}(H^1(\Gam(\os{\circ}{X}),{\cal W}_n))$.
\par 
$(3)$ Let $n$ be a positive integer.
Set 
\begin{align*} 
\Psi^1(\os{\circ}{X},n):= {\rm  length}_{{\cal W}_n}
{\rm Coker}(H^1(\os{\circ}{X}{}^{(0)},
{\cal W}_n({\cal O}_{\os{\circ}{X}{}^{(0)}})) 
\lo H^1(\os{\circ}{X}{}^{(1)}, {\cal W}_n({\cal O}_{\os{\circ}{X}{}^{(1)}}))).
\end{align*}   
If $\dim \os{\circ}{X}=2$ 
and if the boundary map
$d_2^{-2,4} \col E_2^{-2,4} \lo E_2^{0,3}$ of 
{\rm (\ref{ali:hwnpwt})} is zero,  
then the following formula holds$:$
$$p_g(X/s,n,1)=p_g(\os{\circ}{X}{}^{(0)}/\kap,n,1)+
\Psi^1(\os{\circ}{X},n)+
{\rm length}_{{\cal W}_{n}}(H^2(\Gam(\os{\circ}{X}),{\cal W}_n)).$$ 
\end{theo}
\begin{proof}
(1): Let $\{E_{\infty}^{-k,k+d}\}_{k\in {\mab Z}_{\geq 0}}$ 
be the set of the $E_{\infty}$-terms of 
the spectral sequence (\ref{ali:hwnpwt}).  
It is clear that 
$p_g(X/s,n,1)=
\sum_{k=0}^{d}{\rm length}_{{\cal W}_{n}}(E_{\infty}^{-k,d+k})$.
Because $\dim \os{\circ}{X}{}^{(2j+k)}=d-(2j+k)$, 
the non-vanishing term 
$H^{h-d-j}(\os{\circ}{X}{}^{(2j+k)},
{\cal W}_n\Om^{d-j-k}_{\os{\circ}{X}{}^{(2j+k)}})
(-j-k)$ can arise only in the case 
$d-(2j+k)\geq d-j-k$. 
This inequality implies that  $j=0$. 
Hence the spectral sequence 
(\ref{ali:hwnpwt}) is equal to the following spectral sequence 
\begin{align*}
E_1^{-k, h+k}&= 
H^{h-d}(\os{\circ}{X}{}^{(k)},
{\cal W}_n\Om^{d-k}_{\os{\circ}{X}{}^{(k)}})
(-k)\Lo H^{h-d}(X,{\cal W}_n\Om_X^d) \quad (h\in {\mab N}). 
\tag{5.6.2}\label{ali:hwnspwt}\\
\end{align*}  
If $k<0$ or $h<d$, then $E_1^{-k, h+k}=0$.  
Hence $E_{\infty}^{-k, d+k}$ $(k\geq 0)$ is a submodule of 
\begin{align*} 
E_2^{-k, d+k}=
{\rm Ker}(H^0(\os{\circ}{X}{}^{(k)},
{\cal W}_n\Om_{X^{(k)}}^{d-k})(-k) 
\lo H^1(\os{\circ}{X}{}^{(k-1)},{\cal W}_n\Om_{\os{\circ}{X}{}^{(k-1)}}^{d-k+1})(-k+1)).
\tag{5.6.3}\label{ali:hwnspewt}\\ 
\end{align*} 
Furthermore, 
$E_{\infty}^{0, d}=E_1^{0,d}=
H^0(\os{\circ}{X}{}^{(0)},{\cal W}_n\Om_{\os{\circ}{X}{}^{(0)}}^{d})$.
Hence the following natural morphism 
\begin{align*}
H^0(X,{\cal W}_n\Om_{X}^{d})\lo 
H^0(\os{\circ}{X}{}^{(0)},{\cal W}_n\Om_{\os{\circ}{X}{}^{(0)}}^{d}) 
\tag{5.6.4}\label{ali:hwndwt}
\end{align*} 
is surjective.
In conclusion, we obtain the inequalities in (\ref{ali:pgx0n}). 
\par
(2): 
By (\ref{ali:hwnspwt})  
we have the following exact sequence
\begin{align*} 
0 &\lo H^0(\os{\circ}{X}{}^{(0)},{\cal W}_n\Om^1_{\os{\circ}{X}{}^{(0)}}) \lo 
H^0(X,{\cal W}_n\Om^{1}_{X})\\
&\lo {\rm Ker}(H^0(X^{(1)}, {\cal W}_n{\cal O}_{X^{(1)}})(-1) \lo 
H^1(\os{\circ}{X}{}^{(0)},{\cal W}_n\Om^1_{\os{\circ}{X}{}^{(0)}})) \lo 0.
\end{align*} 
The boundary morphism $H^0(X^{(1)}, {\cal W}_n{\cal O}_{X^{(1)}})(-1)
\lo H^1(\os{\circ}{X}{}^{(0)},{\cal W}_n\Om^1_{\os{\circ}{X}{}^{(0)}})$ 
is the \v{C}ech Gysin morphism (\cite[4.9]{msemi}). 
Note that 
$H^0(X^{(1)}, {\cal W}_n{\cal O}_{X^{(1)}})=
H^0_{\rm crys}(X^{(1)}/{\cal W}_n)$ and 
$H^1(\os{\circ}{X}{}^{(0)}, {\cal W}_n\Om^1_{\os{\circ}{X}{}^{(0)}})=
H^2_{\rm crys}(\os{\circ}{X}{}^{(0)}/{\cal W}_n)$. 
Hence 
\begin{align*} 
{\rm Ker}(H^0(X^{(1)}, {\cal W}_n{\cal O}_{X^{(1)}})(-1) \lo 
H^1(\os{\circ}{X}{}^{(0)},{\cal W}_n\Om^1_{\os{\circ}{X}{}^{(0)}}))
=H^1(\Gam(\os{\circ}{X}),{\cal W}_n)(-1)
\end{align*} 
(cf.~[loc.~cit.~5.3]). 
Thus we can complete the proof of (1). 
\par
(3): Because $\dim \os{\circ}{X}=2$, we obtain the following equalities: 
$$E_1^{0,2}=
H^0(\os{\circ}{X}{}^{(0)},
{\cal W}_n\Om^{2}_{\os{\circ}{X}{}^{(0)}}), \quad 
E_1^{-1,3}=
H^0(\os{\circ}{X}{}^{(1)},
{\cal W}_n\Om^1_{\os{\circ}{X}{}^{(1)}})(-1),$$ 
$$E_1^{0,3}=
H^1(\os{\circ}{X}{}^{(0)},{\cal W}_n\Om^2_{\os{\circ}{X}{}^{(0)}}),\quad 
E_1^{-2,4}=
H^0(\os{\circ}{X}{}^{(2)},
{\cal W}_n({\cal O}_{\os{\circ}{X}{}^{(2)}}))(-2),$$ 
$$E_1^{-1,4}=
H^1(\os{\circ}{X}{}^{(1)},
{\cal W}_n\Om^1_{\os{\circ}{X}{}^{(1)}})(-1),\quad 
E_1^{0,4}=
H^2(\os{\circ}{X}{}^{(0)},
{\cal W}_n\Om^2_{\os{\circ}{X}{}^{(0)}}).$$ 
The other $E_1$-terms of (\ref{ali:hwnspwt}) are zero. 
Hence 
$E_{\infty}^{-1,3}=E_{2}^{-1,3}$.  
By the assumption, 
$E_{\infty}^{-2,4}=E_{2}^{-2,4}=H^2(\Gam(\os{\circ}{X}),{\cal W}_n)(-2)$.  
By the duality of Ekedahl (\cite[(2.2.23)]{ek}), 
$E_{2}^{-1,3}$ is the dual of 
${\rm Coker}(H^1(\os{\circ}{X}{}^{(0)},{\cal W}_n({\cal O}_{\os{\circ}{X}{}^{(0)}})) 
\lo H^1(X^{(1)}, {\cal W}_n({\cal O}_{X^{(1)}})))$. 
Hence (3) follows.
\end{proof}

\begin{coro}\label{coro:dk0} 
The natural morphism 
\begin{align*} 
H^0(X,{\cal W}\Om^d_{X})\lo 
H^0(\os{\circ}{X}{}^{(0)},{\cal W}\Om^d_{\os{\circ}{X}{}^{(0)}})
\tag{5.7.1}\label{ali:nms} 
\end{align*} 
is surjective. 
In particular, 
\begin{align*} 
p_g(X/{\cal W}(s),1)\geq p_g(\os{\circ}{X}{}^{(0)}/{\cal W},1).
\tag{5.7.2}\label{ali:nmks} 
\end{align*} 
\end{coro} 
\begin{proof} 
Set $K_n:={\rm Ker}(H^0(X,{\cal W}_n\Om_{X}^{d})\lo 
H^0(\os{\circ}{X}{}^{(0)},{\cal W}_n\Om_{\os{\circ}{X}{}^{(0)}}^{d}))$. 
By the surjectivity of the morphism (\ref{ali:hwndwt}), 
we obtain the following exact sequence 
\begin{align*}
0\lo K_n \lo H^0(X,{\cal W}_n\Om_{X}^{d})\lo 
H^0(\os{\circ}{X}{}^{(0)},{\cal W}_n\Om_{\os{\circ}{X}{}^{(0)}}^{d})\lo 0 
\tag{5.7.3}\label{ali:hkdwt}
\end{align*} 
of ${\cal W}_n$-modules. 
Because $H^0(X,{\cal W}_n\Om_{X}^{d})$ is 
an artrinian ${\cal W}_n$-modules, so is $K_n$, and  
hence the projective system $\{K_n\}_{n=1}^{\infty}$ 
satisfies the Mittag-Leffler condition. 
Hence the following sequence 
$0\lo \vpl_nK_n\lo 
H^0(X,{\cal W}\Om_{X}^{d})\lo 
H^0(\os{\circ}{X}{}^{(0)},{\cal W}\Om_{\os{\circ}{X}{}^{(0)}}^{d})
\lo 0$ is exact. 
\end{proof} 

\begin{rema}
(1) As in (\ref{theo:pgg}), 
even if $\kap$ is any field of any characteristic,
there is the following spectral sequence of abelian groups
$$E_1^{-k, h+k}=H^{h-d}(X^{(k+1)},\Om^{d-k}_{X^{(k+1)}/\kap}) \Lo
H^{h-d}(X,\Om_{X/{\kap}}^{d}). $$
Thus 
$p_g(X/s,1,1)=\sum_{k=0}^{d}{\rm dim}_{\kap}E_{\infty}^{-k,k+d}$.
In particular, the inequalities (\ref{ali:pgx0n}) 
for the case $n=1$ holds. 
\par 
(2) Assume that $\os{\circ}{s}={\rm Spec}({\mab C})$ and 
that $X/s$ is a proper log analytic SNCL space. 
Let $i$ be a nonnegative fixed integer. 
Then we have the following spectral sequence 
by using the Poincar\'{e} residue isomorphism (cf.~(\ref{eqn:mprrn}))$:$ 
\begin{align*}
E_1^{-k, h+k}&=
\us{j\geq {\rm max}\{-k,0\}}{\bigoplus}
H^{h-i-j}(\os{\circ}{X}{}^{(2j+k)},\Om^{i-j-k}_{\os{\circ}{X}{}^{(2j+k)}})
(-j-k)
\tag{5.8.1}\label{ali:anwt}\\
&\Lo H^{h-i}(X,\Om_{X/s}^i) \quad (h\in {\mab N}). 
\end{align*}  
Here $(-j-k)$ is the usual Tate twist. 
Assume that each irreducible components of $\os{\circ}{X}$ is 
K\"{a}hler or the analytification of a proper scheme over ${\mab C}$. 
By using theory of mixed Hodge structures in 
(\cite {dl}, \cite{dh2}) 
and by \cite{fn}, this spectral sequence degenerates at $E_2$. 
%which can be ignored in (\ref{ali:anwt}). 
In particular, the following equality holds$:$ 
\begin{align*} 
p_g(X/s,1) = p_g(\os{\circ}{X}{}^{(0)}/{\mab C},1)+ 
\sum_{k=1}^{d}{\rm dim}_{\mab C}~
{\rm Ker}&(H^0(\os{\circ}{X}{}^{(k)},\Om_{\os{\circ}{X}{}^{(k)}}^{d-k})(-k) \tag{5.8.2}\label{ali:pgxckn}\\
&
\lo 
H^1(\os{\circ}{X}{}^{(k-1)},\Om_{\os{\circ}{X}{}^{(k-1)}}^{d-k+1})(-k+1)). 
\end{align*}  
\end{rema}

Using the Lefschetz' principle, we obtain the following: 

\begin{coro}\label{coro:al0}
Assume that the characteristic of the base field $\kap$ is $0$. 
Let $X/s$ be a proper SNCL scheme. 
Then 
\begin{align*} 
p_g(X/s,1) = p_g(\os{\circ}{X}{}^{(0)}/\kap,1)+ \sum_{k=1}^{d}{\rm dim}_{\kap}~
{\rm Ker}&(H^0(\os{\circ}{X}{}^{(k)},\Om_{\os{\circ}{X}{}^{(k)}}^{d-k}) \tag{5.9.1}\label{ali:pg0xkn}\\
&
\lo 
H^1(\os{\circ}{X}{}^{(k-1)},\Om_{\os{\circ}{X}{}^{(k-1)}}^{d-k+1})). 
\end{align*}  
\end{coro}

\begin{coro}\label{coro:pxan}
%$(1)$ 
%Assume that $\dim \os{\circ}{X}=1$. 
Let the notations be as in {\rm (\ref{theo:pgg})}.  
Set $d:=\dim \os{\circ}{X}$. 
Let $\os{\circ}{X}_{\lam}$ be an irreducible component of $\os{\circ}{X}$. 
Assume that the spectral sequence {\rm (\ref{ali:hwnspwt})} for the case $n=1$ 
degenerates at $E_2$. 
Then, if $p_g(X/s,1,1)=1$ and if $H^d(\Gam(\os{\circ}{X}),\kap)\not =0$, then 
$p_g(\os{\circ}{X}_{\lam},1,1)=0$ and 
$H^d(\Gam(\os{\circ}{X}),\kap)\simeq \kap$. 
%\par 
%$(2)$ Assume that $\dim \os{\circ}{X}=2$. 
%Then, 
%if $p_g(X/s,n,1)=1$ and 
%$H_1(\Gam(\os{\circ}{X}),{\cal W}_n)\not =0$, then 
%all the irreducible component of $\os{\circ}{X}$ is rational and 
%$\Gam(\os{\circ}{X})$ is $m$-gon for some $m\geq 2$. 
\end{coro}
\begin{proof} 
Consider the spectral sequence (\ref{ali:pgx0n}). 
By (\ref{ali:hwnspewt}), 
\begin{align*} 
E_{\infty}^{-d, 2d}=E_2^{-d, 2d}&=
{\rm Ker}(H^0(\os{\circ}{X}{}^{(d)},{\cal O}_{\os{\circ}{X}{}^{(d)}})(-d) 
\lo H^1(\os{\circ}{X}{}^{(d-1)},\Om_{\os{\circ}{X}{}^{(d-1)}}^{1})(-d+1))\\
%\tag{5.3.3}\label{ali:hkerspewt}\\ 
&= H^d(\Gam(\os{\circ}{X}),\kap)(-d). 
\end{align*} 
\end{proof} 

\section{Lower semi-continuity of log plurigenera}\label{sec:llgpss} 
Let $\{\os{\circ}{X}_i\}_{i\in I}$ be a finite set of 
the irreducible component of $\os{\circ}{X}$.
Next we would like to prove that the following inequality
\begin{align*} 
\sum_{i\in I}p_g(\os{\circ}{X}_i/\kap,n,r) \leq p_g(X/s,n,r) \quad 
(r\in {\mab Z}_{\geq 1})
\tag{6.0.1}\label{ali:ine}
\end{align*} 
holds.
%, which has been announced in \cite[(4.4)]{ndw}.
However, for general $n$, I have not been able to prove 
this inequality. 
What I have been able to prove is (\ref{ali:ine}) only 
in the case $n=1$ for the case $r\geq 2$ 
(see (\ref{theo:temfc}) below).  
Even in the case $n=1$, 
because we do not know whether 
the following induced morphism 
by (\ref{ali:hwndwt}) 
\begin{align*}
H^0(X,\os{r}{\us{{\cal O}_X}
{\otimes}}(\Om_{X}^d)^{\otimes r})\lo 
H^0(\os{\circ}{X}{}^{(0)},
\os{r}{\us{{\cal O}_{\os{\circ}{X}{}^{(0)}}}{\otimes}}
\Om_{\os{\circ}{X}{}^{(0)}}^{d}) 
\end{align*} 
is surjective for general $r\geq 2$, 
we need a new argument to prove this inequality. 
%Because we do not know whether 
%the following induced morphism by (\ref{ali:hwndwt}) 
%\begin{align*}
%H^0(X,\os{r}{\us{{\cal W}_n({\cal O}_X)}
%{\otimes}}({\cal W}_n\Om_{X}^d)^{\otimes r})\lo 
%H^0(\os{\circ}{X}{}^{(0)},
%\os{r}{\us{{\cal W}_n({\cal O}_{\os{\circ}{X}{}^{(0)}})}{\otimes}}
%{\cal W}_n\Om_{\os{\circ}{X}{}^{(0)}}^{d}) 
%\end{align*} 
%is surjective for general $r\geq 2$, 
%we need a new argument to prove this inequality. 
Let $r$ be a positive integer. 
Set $a:=a^{(0)}\col \os{\circ}{X}{}^{(0)}\lo \os{\circ}{X}$. 
Let $n$ be a positive integer or nothing. 
The key for the proof of this inequality is 
to construct the following morphism and 
to prove that it is injective: 

\begin{theo}\label{theo:costm} 
There exists a morphism 
\begin{align*} 
\Psi_{X,n,r}\col 
%\bigoplus_{i\in I}
%a_{i*}(\os{r}{\us{{\cal W}_n({\cal O}_{\os{\circ}{X}_i})}{\otimes}}
%{\cal W}_n\Om^d_{\os{\circ}{X}_i}) \lo 
a_*(\os{r}{\us{{\cal W}_n({\cal O}_{\os{\circ}{X}{}^{(0)}})}{\otimes}}
{\cal W}_n\Om^d_{\os{\circ}{X}{}^{(0)}}) \lo 
\os{r}{\us{{\cal W}_n({\cal O}_{X})}{\otimes}}{\cal W}_n\Om^d_X
\tag{6.1.1}\label{ali:fsiws} 
\end{align*} 
of ${\cal W}_n({\cal O}_X)$-modules, 
which shall be shown to be injective in {\rm (\ref{theo:temfc})} below 
for the case $n=1$. 
\end{theo}  
\begin{proof} 
By \cite[3.15]{msemi} and \cite[(6.28) (9)]{ndw} 
(cf.~\cite[(6.29) (1)]{ndw}), the morphism 
$\theta_n\wedge  \col 
{\cal W}_n\Om^d_X \lo 
{\cal W}_n\wt{\Om}^{d+1}_X/P_0{\cal W}_n\wt{\Om}^{d+1}_X$ 
of ${\cal W}_n({\cal O}_X)$-modules 
is an isomorphism of ${\cal W}_n({\cal O}_X)$-modules: 
\begin{align*} 
\theta_n\wedge  \col 
{\cal W}_n\Om^d_X \os{\sim}{\lo}  
{\cal W}_n\wt{\Om}^{d+1}_X/P_0{\cal W}_n\wt{\Om}^{d+1}_X. 
\tag{6.1.2}\label{ali:dp0} 
\end{align*} 
Consider the following exact sequence of 
of ${\cal W}_n({\cal O}_X)$-modules: 
\begin{align*} 
0\lo 
{\rm gr}^P_1{\cal W}_n\wt{\Om}^{d+1}_X 
\lo 
{\cal W}_n\wt{\Om}^{d+1}_X/P_0{\cal W}_n\wt{\Om}^{d+1}_X 
\lo 
{\cal W}_n\wt{\Om}^{d+1}_X/P_1{\cal W}_n\wt{\Om}^{d+1}_X 
\lo 0. 
\tag{6.1.3}\label{ali:dpa0} 
\end{align*} 
By \cite[(3.7)]{msemi} 
we have the following isomorphism 
\begin{align*} 
{\rm Res}\col 
{\rm gr}^P_1{\cal W}_n\wt{\Om}^{d+1}_X
\os{\sim}{\lo} a_*({\cal W}_n\Om^{d}_{\os{\circ}{X}{}^{(0)}}) 
\tag{6.1.4}\label{ali:dppa0} 
\end{align*} 
of ${\cal W}_n({\cal O}_X)$-modules. 
Hence we have the following composite morphism 
of ${\cal W}_n({\cal O}_X)$-modules: 
\begin{align*} 
a_*({\cal W}_n\Om^d_{\os{\circ}{X}{}^{(0)}}) 
\os{\os{\rm Res}{\simeq}}{\longleftarrow} 
{\rm gr}_1^P{\cal W}_n\tilde{\Om}^{d+1}_{X/{\kap}}\os{\subset}{\lo}
{\cal W}_n\tilde{\Om}^{d+1}_{X}/P_0{\cal W}_n\tilde{\Om}^{d+1}_{X}
\os{\os{\theta_n}{\simeq}}{\longleftarrow}{\cal W}_n{\Om}^{d}_{X}.
\tag{6.1.5}\label{ali:awod} 
\end{align*} 
Consequently we have the following morphism of ${\cal W}_n({\cal O}_X)$-modules: 
\begin{align*} 
%a_*(\os{r}{\us{{\cal W}_n({\cal O}_{\os{\circ}{X}{}^{(0)}})}{\otimes}}
%{\cal W}_n\Om^d_{\os{\circ}{X}{}^{(0)}}) 
%=
\os{r}{\us{{\cal W}_n({\cal O}_X)}{\otimes}}
a_*({\cal W}_n\Om^{d}_{\os{\circ}{X}{}^{(0)}}) 
\lo \os{r}{\us{{\cal W}_n({\cal O}_X)}{\otimes}}
{\cal W}_n\Om^d_X. 
\tag{6.1.6}\label{ali:dppn} 
\end{align*} 
Let $a_i\col \os{\circ}{X}_i \lo \os{\circ}{X}$ be the following composite morphism 
$\os{\circ}{X}_i \os{\subset}{\lo} \os{\circ}{X}{}^{(0)} \os{a}{\lo} \os{\circ}{X}$. 
Because 
\begin{align*} 
\os{r}{\us{{\cal W}_n({\cal O}_X)}{\otimes}}
a_*({\cal W}_n\Om^{d}_{\os{\circ}{X}{}^{(0)}})&=
\os{r}{\us{{\cal W}_n({\cal O}_X)}{\otimes}}
(\us{i}{\bigoplus}a_{i*}({\cal W}_n\Om^d_{\os{\circ}{X}_{i}/{\kap}}))\\
&=
\us{i_{1},\ldots,i_{r}}{{\bigoplus}}
a_{i_{1}*}({\cal W}_n\Om^d_{\os{\circ}{X}_{i_{1}}/{\kap}})
\otimes_{{\cal W}_n({\cal O}_X)} 
\cdots \otimes_{{\cal W}_n({\cal O}_X)} 
a_{i_{r}*}({\cal W}_n\Om^d_{\os{\circ}{X}_{i_{r}}/{\kap}})
\end{align*}  
and the last sheaf contains 
$\us{i}{\bigoplus}a_{i*}(({\cal W}_n\Om^d_{X_i/{\kap}})^{\otimes\,r})$,
there exists the following morphism of ${\cal W}_n({\cal O}_X)$-modules:
\begin{align*} 
\bigoplus_{i\in I}
a_{i*}(\os{r}{\us{{\cal W}_n({\cal O}_{\os{\circ}{X}_i})}{\otimes}}
{\cal W}_n\Om^d_{\os{\circ}{X}_i}) 
\lo 
\os{r}{\us{{\cal W}_n({\cal O}_X)}{\otimes}}{\cal W}_n\Om^{d}_{X}.
\tag{6.1.7}\label{ali:wnoo}
\end{align*}  
Since $a_*(\os{r}{\us{{\cal W}_n({\cal O}_{\os{\circ}{X}{}^{(0)}})}{\otimes}}
{\cal W}_n\Om^d_{\os{\circ}{X}{}^{(0)}})=
\bigoplus_{i\in I}
a_{i*}(\os{r}{\us{{\cal W}_n({\cal O}_{\os{\circ}{X}_i})}{\otimes}}
{\cal W}_n\Om^d_{\os{\circ}{X}_i})$, 
we obtain the desired morphism (\ref{ali:fsiws}).  
\end{proof} 

\begin{rema}\label{rema:ag}
Because the morphism (\ref{ali:awod}) is injective, 
we have the following injective morphism 
\begin{align*} 
\Psi_{X,n,1} \col H^0(\os{\circ}{X}{}^{(0)},
{\us{{\cal W}_n({\cal O}_{\os{\circ}{X}{}^{(0)}})}{\otimes}}
{\cal W}_n\Om^d_{\os{\circ}{X}{}^{(0)}}) \os{\sus}{\lo} 
H^0(X,\os{r}{\us{{\cal W}_n({\cal O}_{X})}{\otimes}}{\cal W}_n\Om^d_X). 
\tag{6.2.1}\label{ali:fsaiws} 
\end{align*} 
Hence we obtain the inequality again 
\begin{align*} 
p_g(\os{\circ}{X}{}^{(0)}/\kap,n,1)\leq p_g(X/s,n,1),  
\end{align*} 
which has been proved in (\ref{ali:pgx0n}). 
\end{rema}

Now we can prove the following. 
This result is similar to those of N.~Nakayama 
(\cite[Theorem 11]{nn}, \cite[(6.3)]{nn2}) 
(cf.~\cite[(4.2)]{cl}, \cite[\S6]{morr}). 

\begin{theo}\label{theo:temfc} 
The morphism {\rm (\ref{ali:fsiws})} is injective. 
%For the case $n=1$, the following induced morphism by 
%{\rm (\ref{ali:fsiws})} is injective$:$ 
%\begin{align*} 
%H^0(\os{\circ}{X}{}^{(0)},
%\os{r}{\us{{\cal O}_{\os{\circ}{X}{}^{(0)}}}{\otimes}}
%\Om^d_{\os{\circ}{X}{}^{(0)}/\kap}) \os{\subset}{\lo} 
%H^0(X, \os{r}{\us{{\cal O}_{X}}{\otimes}}\Om^d_{X/\kap}).
%\tag{6.6.2}\label{ali:fsios} 
%\end{align*} 
Consequently 
\begin{align*} 
p_g(X/s,1,r)\geq p_g(\os{\circ}{X}{}^{(0)}/\kap,1,r)
\tag{6.3.1}\label{ali:fswpgs} 
\end{align*} 
and 
\begin{align*} 
\kap(X,1)\geq {\rm max}\{\kap(X_i,1) \vert X_i:
\,{\rm an}~{\rm irreducible}~{\rm component}~{\rm of}~X\}
\tag{6.3.2}\label{ali:fscgs} 
\end{align*} 
$($cf.~{\rm \cite{ueno}, \cite{ni}}$)$.  
\end{theo} 
\begin{proof} 
For the time being, we do not assume that $n=1$. 
Let $\os{\circ}{X}_i$ be an irreducible component of $\os{\circ}{X}$.
%Let $a_i\col \os{\circ}{X}_i \lo \os{\circ}{X}$ be the following composite morphism 
%$\os{\circ}{X}_i \os{\subset}{\lo} \os{\circ}{X}{}^{(0)} \os{a}{\lo} \os{\circ}{X}$. 
Let $X_{\rm sm}$ be the smooth locus of $\os{\circ}{X}/\kap$. 
We claim that the morphism (\ref{ali:fsiws}) 
is an isomorphism on the nonsingular points 
$X_{\rm sm}$ of $\os{\circ}{X}$. 
Indeed, it suffices to show that the morphism 
$a_*({\cal W}_n\Om^d_{\os{\circ}{X}{}^{(0)}}) 
\lo {\cal W}_n\Om^{d}_{X}$ is an isomorphism on $X_{\rm sm}$ 
since $\os{\circ}{X}{}^{(0)}\vert_{X_{\rm sm}}=X_{\rm sm}$. 
%Consider the morphism (\ref{ali:fsiws}). 
Locally on $X_{\rm sm}$ we have a system of local parameters 
$x_1, \ldots, x_d$. We may assume that there is a local isomorphism 
$X_{\rm sm} \simeq {\rm Spec}\,\kap[t,x_1,\ldots, x_d]/(t)$ 
and we have a local admissible lift $
\ol{\cal X}:={\rm Spec}\,{\cal W}[t, x_1,\ldots, x_d]$ 
of $X_{\rm sm}$ over ${\cal W}[t]$. 
Set ${\cal X}_n:=\ol{\cal X}\times_{{\cal W}[t]}{\cal W}_n$. 
Then $\theta \in {\cal W}_n\wt{\Om}^1_X={\cal H}^0(\Om^{\bul}_{{\cal X}_n})$ 
is the class $[d\log t]$ of $d\log t$. 
The inverse of the Poincar\'{e} residue isomorphism 
$a_*({\cal W}_n\Om^d_{\os{\circ}{X}{}^{(0)}})\os{\simeq}{\lo}
{\rm gr}_1^P{\cal W}_n\tilde{\Om}^{d+1}_{X}$
is defined by $\om \lom \om [d\log t]$. 
Because $a_*({\cal W}_n\Om^d_{\os{\circ}{X}{}^{(0)}})\vert_{X_{\rm sm}}
= {\cal W}_n\Om^d_{X_{\rm sm}}={\cal W}_n{\Om}^{d}_X\vert_{X_{\rm sm}}$,
the restriction of (\ref{ali:fsiws}) to $X_{\rm sm}$ is equal to 
the identity ${\rm id}_{{\cal W}_n\Om^d_{X_{\rm sm}}}$.  
Thus we have proved that our claim holds. 
\par 
By \cite[(1.17)]{lodw} the sheaf 
$\os{r}{\us{{\cal W}_n({\cal O}_X)}{\otimes}}{\cal W}_n\Om^d_X$ 
is a coherent ${\cal W}_n({\cal O}_X)$-module. 
In particular, 
$a_*(\os{r}{\us{{\cal W}_n({\cal O}_{\os{\circ}{X}{}^{(0)}})}{\otimes}}
{\cal W}_n\Om^d_{\os{\circ}{X}{}^{(0)}})$ is also a coherent 
${\cal W}_n({\cal O}_X)$-module. 
\par 
Now consider the case $n=1$ (but $r$ is general). 
Let $F_{\os{\circ}{X}_{i}}\col \os{\circ}{X}_{i}\lo \os{\circ}{X}_{i}$ 
be the Frobenius endomorphism of $\os{\circ}{X}_{i}$. 
Since $(C^{-1})^{-1} \col {\cal W}_1\Om^d_{\os{\circ}{X}_{i}}
=F_*{\cal H}^d(\Om^{\bul}_{\os{\circ}{X}_{i}/{\kap}})
\os{\sim}{\lo} \Om^d_{\os{\circ}{X}_{i}/{\kap}}$, 
${\cal W}_1\Om^d_{\os{\circ}{X}_{i}}$ is 
an invertible ${\cal W}_1({\cal O}_{\os{\circ}{X}_{i}})
(\simeq {\cal  O}_{X_{i}})$-module. 
Hence $a_*(\os{r}{\us{{\cal W}_1({\cal O}_{\os{\circ}{X}{}^{(0)}})}{\otimes}}
{\cal W}_1\Om^d_{\os{\circ}{X}{}^{(0)}/{\kap}})$ 
is an invertible ${\cal W}_1({\cal O}_{\os{\circ}{X}{}^{(0)}})$-module. 
Because the morphism (\ref{ali:fsiws}) 
is an isomorphism on a dense open subset and 
because an \'{e}tale algebra over 
$\kap[x_1,\ldots,x_{d}]$ has no non-zero divisor, 
the morphism (\ref{ali:fsiws}) is injective. 
\end{proof}

In the rest of this section, 
we give the compatibility 
with fundamental operators. 
(This may be useful for checking whether 
$\Psi_{X,n,r}$ is injective.)

\begin{prop}\label{prop:ra0}
Let $R\col 
{\cal W}_{n+1}\Om^d_{\os{\circ}{X}{}^{(0)}} \lo 
{\cal W}_{n}\Om^d_{\os{\circ}{X}{}^{(0)}}$ 
and 
$R\col {\cal W}_{n+1}\Om^d_X\lo {\cal W}_n\Om^d_X$ 
be the projections. 
$($These are morphisms of 
${\cal W}_{n+1}({\cal O}_{\os{\circ}{X}{}^{(0)}})$-modules 
and ${\cal W}_{n+1}({\cal O}_X)$-modules, respectively.$)$
Then the following diagram is commutative$:$ 
\begin{equation*} 
\begin{CD} 
a_*(\os{r}{\us{{\cal W}_{n+1}({\cal O}_{\os{\circ}{X}{}^{(0)}})}{\otimes}}
{\cal W}_{n+1}\Om^d_{\os{\circ}{X}{}^{(0)}}) 
@>{\Psi_{X,n+1,r}}>>
\os{r}{\us{{\cal W}_{n+1}({\cal O}_{X})}{\otimes}}{\cal W}_{n+1}\Om^d_X\\
@V{a_*(\otimes^rR})VV @VV{\otimes^rR}V \\
a_*(\os{r}{\us{{\cal W}_n({\cal O}_{\os{\circ}{X}{}^{(0)}})}{\otimes}}
{\cal W}_n\Om^d_{\os{\circ}{X}{}^{(0)}})  @>{\Psi_{X,n,r}}>>
\os{r}{\us{{\cal W}_n({\cal O}_{X})}{\otimes}}
{\cal W}_n\Om^d_X.
\end{CD} 
\tag{6.4.1}\label{cd:fsiwas} 
\end{equation*} 
\end{prop}
\begin{proof} 
By \cite[(8.4.3), (11.1)]{ndw} we have the following commutative diagram 
\begin{equation*} 
\begin{CD} 
a_*({\cal W}_{n+1}\Om^d_{\os{\circ}{X}{}^{(0)}}) 
@<{{\rm Res},\,\simeq}<< 
{\rm gr}_1^P{\cal W}_{n+1}\tilde{\Om}^{d+1}_{X}@>{\subset}>>
{\cal W}_{n+1}\tilde{\Om}^{d+1}_{X}/P_0{\cal W}_{n+1}\tilde{\Om}^{d+1}_{X}
@<{\theta_{n+1},\,\simeq}<<{\cal W}_{n+1}{\Om}^{d}_{X}\\
@V{a_*(R)}VV @V{R}VV @V{R}VV @V{R}VV \\
a_*({\cal W}_n\Om^d_{\os{\circ}{X}{}^{(0)}}) 
@<{{\rm Res},\,{\simeq}}<< 
{\rm gr}_1^P{\cal W}_n\tilde{\Om}^{d+1}_{X}@>{\subset}>>
{\cal W}_n\tilde{\Om}^{d+1}_{X}/P_0{\cal W}_n\tilde{\Om}^{d+1}_{X}
@<{\theta_n,\,\simeq}<<{\cal W}_n{\Om}^{d}_{X}. 
\end{CD}
\tag{6.4.2}\label{cd:awod} 
\end{equation*} 
Hence we have the following commutative diagram 
\begin{equation*} 
\begin{CD} 
\os{r}{\us{{\cal W}_{n+1}({\cal O}_X)}{\otimes}}
a_*({\cal W}_{n+1}\Om^d_{\os{\circ}{X}{}^{(0)}}) 
@>{\Psi_{X,n+1,r}}>>
\os{r}{\us{{\cal W}_{n+1}({\cal O}_{X})}{\otimes}}{\cal W}_{n+1}\Om^d_X\\
@V{\otimes^ra_*(R)}VV @VV{\otimes^rR}V \\
\os{r}{\us{{\cal W}_{n}({\cal O}_X)}{\otimes}}
a_*({\cal W}_{n}\Om^d_{\os{\circ}{X}{}^{(0)}}) 
@>{\Psi_{X,n,r}}>>
\os{r}{\us{{\cal W}_n({\cal O}_{X})}{\otimes}}{\cal W}_n\Om^d_X.
\end{CD} 
\tag{6.4.3}\label{cd:fsipiws} 
\end{equation*} 
Since the following diagram 
\begin{equation*} 
\begin{CD} 
a_*(\os{r}{\us{{\cal W}_{n+1}({\cal O}_{\os{\circ}{X}{}^{(0)}})}{\otimes}}
{\cal W}_{n+1}\Om^d_{\os{\circ}{X}{}^{(0)}}) 
@>{\subset}>>
\os{r}{\us{{\cal W}_{n+1}({\cal O}_X)}{\otimes}}
a_*({\cal W}_{n+1}\Om^d_{\os{\circ}{X}{}^{(0)}}) \\
@V{a_*(\otimes^rR)}VV @VV{\otimes^ra_*(R)}V \\
 a_*(\os{r}{\us{{\cal W}_{n}({\cal O}_{\os{\circ}{X}{}^{(0)}})}{\otimes}}
{\cal W}_{n}\Om^d_{\os{\circ}{X}{}^{(0)}}) 
@>{\subset}>>
\os{r}{\us{{\cal W}_{n}({\cal O}_X)}{\otimes}}
a_*({\cal W}_n\Om^d_{\os{\circ}{X}{}^{(0)}}).
\end{CD} 
\tag{6.4.4}\label{cd:fsipdws} 
\end{equation*} 
is commutative, 
we obtain the commutative diagram (\ref{cd:fsiwas}).  
\end{proof} 

\begin{defi}\label{defi:ni}
Set $\Psi_{X,\infty,r}:=\vpl_n\Psi_{X,n,r}$. 
\end{defi}

\begin{prob}\label{prob:pi}
Is $\Psi_{X,n,r}$ injective for any $n, r\in {\mab Z}_{>0}$? 
In particular, is $\Psi_{X,\infty,r}$ injective for any $r\in {\mab Z}_{>0}$ 
and $\kap(X/{\cal W}(s))\geq {\rm max}\{\kap(X_i/{\cal W})) \vert X_i:
~{\rm an}~{\rm irreducible}~{\rm component}~{\rm of}~X\}$?
\end{prob}

\begin{prop}\label{prop:rafv0}
%Let $R\col 
%{\cal W}_{n+1}\Om^d_{\os{\circ}{X}{}^{(0)}} \lo 
%{\cal W}_{n}\Om^d_{\os{\circ}{X}{}^{(0)}}$ 
%and 
%$R\col {\cal W}_{n+1}\Om^d_X\lo {\cal W}_n\Om^d_X$ 
%be the projections. 
%$($These are morphisms of 
%${\cal W}_{n+1}({\cal O}_{\os{\circ}{X}{}^{(0)}})$-modules 
%and ${\cal W}_{n+1}({\cal O}_X)$-modules, respectively.$)$
Let $F_{{\cal W}_{n}(\os{\circ}{X}{}^{(0)})}
\col {\cal W}_{n}(\os{\circ}{X}{}^{(0)})\lo 
{\cal W}_{n}(\os{\circ}{X}{}^{(0)})$  
and 
$F_{{\cal W}_{n}(X)}\col {\cal W}_{n}(X)\lo {\cal W}_{n}(X)$  
be the Frobenius endomorphisms of 
${\cal W}_{n}(\os{\circ}{X}{}^{(0)})$ and 
${\cal W}_{n}(X)$, respectively. 
Then the following hold$:$
\par 
$(1)$  The following diagram is commutative$:$ 
\begin{equation*} 
\begin{CD} 
a_*(\os{r}{\us{{\cal W}_{n+1}({\cal O}_{\os{\circ}{X}{}^{(0)}})}{\otimes}}
{\cal W}_{n+1}\Om^d_{\os{\circ}{X}{}^{(0)}}) 
@>{\Psi_{X,n+1,r}}>>
\os{r}{\us{{\cal W}_{n+1}({\cal O}_{X})}{\otimes}}{\cal W}_{n+1}\Om^d_X\\
@V{a_*(\otimes^rF})VV @VV{\otimes^rF}V \\
a_*(\os{r}{\us{{\cal W}_{n}({\cal O}_{\os{\circ}{X}{}^{(0)}})}{\otimes}}
F_{{\cal W}_n(\os{\circ}{X}{}^{(0)})*}
({\cal W}_{n}\Om^d_{\os{\circ}{X}{}^{(0)}}))  
@>{F_{{\cal W}_{n}(X)*}(\Psi_{X,n,r})}>>
\os{r}{\us{{\cal W}_n({\cal O}_{X})}{\otimes}}
F_{{\cal W}_{n}(X)*}({\cal W}_n\Om^d_X).
\end{CD} 
\tag{6.7.1}\label{cd:fsifs} 
\end{equation*} 
\par  
$(2)$ The following diagram is commutative$:$ 
\begin{equation*} 
\begin{CD} 
a_*(\os{r}{\us{{\cal W}_{n}({\cal O}_{\os{\circ}{X}{}^{(0)}})}{\otimes}}
F_{{\cal W}_{n}(\os{\circ}{X}{}^{(0)})*}
({\cal W}_{n}\Om^d_{\os{\circ}{X}{}^{(0)}}))  
@>{F_{{\cal W}_n(X)*}(\Psi_{X,n,r})}>>
\os{r}{\us{{\cal W}_n({\cal O}_{X})}{\otimes}}
F_{{\cal W}_n(X)*}({\cal W}_n\Om^d_X)\\
@V{a_*(\otimes^rV})VV @VV{\otimes^rV}V \\
a_*(\os{r}{\us{{\cal W}_{n+1}({\cal O}_{\os{\circ}{X}{}^{(0)}})}{\otimes}}
{\cal W}_{n+1}\Om^d_{\os{\circ}{X}{}^{(0)}}) 
@>{\Psi_{X,n+1,r}}>>
\os{r}{\us{{\cal W}_{n+1}({\cal O}_{X})}{\otimes}}{\cal W}_{n+1}\Om^d_X.
\end{CD} 
\tag{6.7.2}\label{cd:fsvifs} 
\end{equation*} 
\end{prop}
\begin{proof} 
(1): In \cite[(3.2)]{nlfc} we have proved that 
the morphism 
$F\col {\cal W}_{n+1}\Om^d_X \lo 
F_{{\cal W}_{n}(X)*}({\cal W}_n\Om^d_X)$ 
is a morphism of ${\cal W}_{n+1}({\cal O}_{X})$-modules. 
Let $\ol{\cal X}$ be an admissible lift of $X$ over 
${\cal W}_{n+1}[t]$. 
Set ${\cal X}_m:=\ol{\cal X}
\otimes_{{\cal W}_{m}[t]}{\cal W}_{m}$ $(m=n,n+1)$. 
Then the operators 
$F\col {\cal W}_{n+1}\tilde{\Om}^{d+1}_{X}\lo 
F_{{\cal W}_{n}(X)*}({\cal W}_{n}\tilde{\Om}^{d+1}_{X})$, 
$F\col {\cal W}_{n+1}\Om^{d}_{X}\lo 
F_{{\cal W}_{n}(X)*}({\cal W}_{n}{\Om}^{d}_{X})$ 
and 
$F\col {\cal W}_{n+1}\Om^d_{\os{\circ}{X}{}^{(0)}}\lo 
F_{{\cal W}_{n}(X)*}({\cal W}_{n}{\Om}^{d}_{\os{\circ}{X}{}^{(0)}})$ 
are induced by the projections  
$\Om^{\bul}_{{\cal X}_{n+1}/{\cal W}_{n+1}}\lo 
\Om^{\bul}_{{\cal X}_n/{\cal W}_{n}}$, 
$\Om^{\bul}_{{\cal X}_{n+1}/{\cal W}_{n+1}(s)}\lo 
\Om^{\bul}_{{\cal X}_n/{\cal W}_{n}(s)}$ 
and 
$\Om^{\bul}_{\os{\circ}{\cal X}{}^{(0)}_{n+1}/{\cal W}_{n+1}}\lo 
\Om^{\bul}_{{\cal X}{}^{(0)}_n/{\cal W}_n}$,  
respectively.
Hence the following diagram is commutative: 
\begin{equation*} 
\begin{CD} 
a_*({\cal W}_{n+1}\Om^d_{\os{\circ}{X}{}^{(0)}}) 
@<{{\rm Res},\,\simeq}<< 
{\rm gr}_1^P{\cal W}_{n+1}\tilde{\Om}^{d+1}_{X}@>{\subset}>>\\
@V{a_*(F)}VV @V{F}VV  \\
F_{{\cal W}_{n}(X)*}(a_*({\cal W}_n\Om^d_{\os{\circ}{X}{}^{(0)}})) 
@<{{\rm Res},\,{\simeq}}<< 
F_{{\cal W}_{n}(X)*}({\rm gr}_1^P{\cal W}_n\tilde{\Om}^{d+1}_{X})
@>{\subset}>>
\end{CD}
\tag{6.7.3}\label{cd:awpod} 
\end{equation*} 
\begin{equation*} 
\begin{CD} 
{\cal W}_{n+1}\tilde{\Om}^{d+1}_{X}/P_0{\cal W}_{n+1}\tilde{\Om}^{d+1}_{X}
@<{\theta_{n+1},\,\simeq}<<{\cal W}_{n+1}{\Om}^{d}_{X}\\
@V{F}VV @VV{F}V \\
F_{{\cal W}_{n}(X)*}(
{\cal W}_n\tilde{\Om}^{d+1}_{X}/P_0{\cal W}_n\tilde{\Om}^{d+1}_{X})
@<{\theta_n,\,\simeq}<<F_{{\cal W}_{n}(X)*}({\cal W}_n{\Om}^{d}_{X}). 
\end{CD}
\end{equation*} 
The rest of the proof is the same as that of (\ref{prop:ra0}). 
\par 
(2): The proof of (2) is similar to that of (1). 
\end{proof}

To state the contravariant functoriality, 
let us recall the following: 

\begin{prop}[{\rm {\bf \cite[(4.3)]{nh3}}}]\label{prop:mmoo} 
Let $g\col Y \lo Y'$ be a morphism of 
fs log $($formal$)$ schemes satisfying the condition 
\begin{equation*} 
M_{Z,y}/{\cal O}^*_{Z,y}\simeq {\mab N}^r
\tag{6.8.1}\label{eqn:epynr}
\end{equation*}
for any point $y$ of $\os{\circ}{Z}$ 
and for some $r\in {\mab N}$ depending on $y$. 
Let $b^{\star}{}^{(k)} \col \os{\circ}{D}{}^{(k)}(M_{Y^{\star}}) \lo \os{\circ}{Y}{}^{\star}$ 
$(k\in {\mab Z})$ $(\star=$nothing or $')$ be the morphism defined before 
{\rm \cite[(4.3)]{nh3}}. 
Assume that, for each point $y\in \os{\circ}{Y}$ 
and for each member $m$ of the minimal generators 
of $M_{Y,y}/{\cal O}^*_{Y,y}$, there exists 
a unique member $m'$ of the minimal generators of 
$M_{Y',\os{\circ}{g}(y)}/{\cal O}^*_{Y',\os{\circ}{g}(y)}$ 
such that $g^*(m')\in m^{{\mab Z}_{>0}}$. 
Then there exists a canonical morphism 
$\os{\circ}{g}{}^{(k)}\col \os{\circ}{D}{}^{(k)}(M_Y) \lo 
\os{\circ}{D}{}^{(k)}(M_{Y'})$ fitting 
into the following commutative diagram of schemes$:$
\begin{equation*} 
\begin{CD} 
\os{\circ}{D}{}^{(k)}(M_Y) @>{\os{\circ}{g}{}^{(k)}}>> 
\os{\circ}{D}{}^{(k)}(M_{Y'}) \\ 
@V{b^{(k)}}VV @VV{b'{}^{(k)}}V \\ 
\os{\circ}{Y} @>{\os{\circ}{g}}>> \os{\circ}{Y}{}'. 
\end{CD} 
\tag{6.8.2}\label{cd:dmgdm}
\end{equation*} 
\end{prop}

The following has been proved in \cite[(1.3.20)]{nb} (cf.~\cite[(4.8)]{nh3}): 

\begin{prop}[{\bf The contravariant functoriality of the Poincar\'{e} residue morphism}]\label{prop:rescos} 
Let $S$ be as above and let 
$u \col S\lo S'$ 
be a morphism of family of log points.  
Let $X'/S'$ be an SNCL scheme.  
Let  $X' \os{\subset}{\lo} {\cal P}{}'$ be  
an immersion into a log smooth scheme over $S'$ 
fitting into the following commutative diagram over the morphism $S\lo S':$ 
\begin{equation*} 
\begin{CD} 
X @>{\subset}>> {\cal P} \\
@V{g\vert_{X}}VV @VV{g}V \\ 
X' @>{\subset}>> {\cal P}'. \\  
\end{CD} 
%\tag{1.3.19.1}\label{eqn:pvpq} 
\end{equation*}   
Let ${\cal P}^{\rm ex}$ and ${\cal P}'{}^{\rm ex}$ be the exactification of 
the immersions $X \os{\subset}{\lo} {\cal P}$ and 
$X' \os{\subset}{\lo} {\cal P}'$, respectively. 
Let $g^{\rm ex}\col {\cal P}^{\rm ex}\lo 
{\cal P}'{}^{\rm ex}$ be the induced morphism by $g$. 
Let $a'{}^{(k)}\col \os{\circ}{X}{}'^{(k)}\lo \os{\circ}{X}{}'$ $(k\in {\mab N})$
be the natural morphism of schemes over $\os{\circ}{S}{}'_0$.  
%Assume also that the morphism 
%$g\col (\os{\circ}{\cal P},{\cal M})\lo 
%(\os{\circ}{\cal P}{}',{\cal M}')$ 
%satisfies the condition in {\rm (\ref{prop:mmoo})}. 
Assume that, for each point 
$x\in \os{\circ}{\cal P}{}^{\rm ex}$ 
and for each member $m$ of the minimal generators 
of $M_{{\cal P}^{\rm ex},x}/{\cal O}^*_{{\cal P}^{\rm ex},x}$, 
there exists a unique member $m'$ of 
the minimal generators of 
$M_{{\cal P}{}'^{\rm ex},\os{\circ}{g}(x)}
/{\cal O}^*_{{\cal P}{}'^{\rm ex},\os{\circ}{g}(x)}$ 
such that $g^*(m')= m$ and such that the image of 
the other minimal generators of 
$M_{{\cal P}{}'^{\rm ex},{\os{\circ}{g}(x)}}
/{\cal O}^*_{{\cal P}{}'^{\rm ex},\os{\circ}{g}(x)}$ by $g^*$ 
are the trivial element of 
$M_{{\cal P}^{\rm ex},{x}}/{\cal O}^*_{{\cal P}^{\rm ex},x}$. 
Let 
$a'{}^{(k)}\col \os{\circ}{X}{}'^{(k)}\lo \os{\circ}{X}{}'$ 
be an analogous morphism to $a^{(k)}$ for $X'$.  
Then the following diagram is commutative$:$ 
\begin{equation*} 
\begin{CD}  
g^{\rm ex}_*({\rm Res}): 
g^{\rm ex}_*({\rm gr}_k^P
({\Om}^{\bul}_{{\cal P}^{\rm ex}/\os{\circ}{S}}))
@>{\sim}>> \\ 
@AAA \\
{\rm Res}: 
{\rm gr}_k^P({\Om}^{\bul}_{{\cal P}'{}^{\rm ex}/\os{\circ}{S}{}'})
@>{\sim}>> 
\end{CD}
\tag{6.9.1}\label{eqn:regdmrp} 
\end{equation*} 
\begin{equation*} 
\begin{CD}  
g^{\rm ex}_*a^{(k-1)}_*
({\Om}^{\bul}_{\os{\circ}{\cal P}{}^{{\rm ex},(k-1)}/\os{\circ}{S}}
\otimes_{\mab Z}
\vp^{(k-1)}_{\rm zar}(\os{\circ}{X}/\os{\circ}{S}_0)\{-k\})\\
@AAA \\
a'{}^{(k-1)}_*({\Om}^{\bul}_{\os{\circ}{\cal P}{}^{'{\rm ex},(k-1)}/\os{\circ}{S}{}'}
\otimes_{\mab Z}
\vp^{(k-1)}_{\rm zar}(\os{\circ}{X}{}'/\os{\circ}{S}{}'_0)\{-k\}). 
\end{CD} 
\end{equation*} 
\end{prop}

\begin{prop}\label{prop:nn}
Let $s\lo t$ be a morphism of 
the log points of perfect fields of characteristic $p>0$.
%Let $e$ be the image of $1$ of the following composite morphism 
%\begin{align*}
%{\mab N}=M_t/{\cal O}^*_t\lo M_s/{\cal O}^*_s={\mab N}. 
%\end{align*} 
Let $Y/t$ is a proper SNCL scheme and let 
$b\col \os{\circ}{Y}{}^{(0)}\lo \os{\circ}{Y}$ be an analogue of 
$a \col \os{\circ}{X}{}^{(0)}\lo \os{\circ}{X}$.  
Let $g\col X\lo Y$ be a morphism of log schemes satisfying 
the assumption 
in {\rm (\ref{prop:mmoo})}. 
Then the following diagram is  commutative$:$ 
\begin{equation*} 
\begin{CD} 
g_*a_*(\os{r}{\us{{\cal W}_n({\cal O}_{\os{\circ}{X}{}^{(0)}})}{\otimes}}
{\cal W}_n\Om^d_{\os{\circ}{X}{}^{(0)}}) @>{g_*(\Psi_{X,n,r})}>> 
g_*(\os{r}{\us{{\cal W}_n({\cal O}_{X})}{\otimes}}{\cal W}_n\Om^d_X)\\
@A{g^{(0)*}}AA @AA{g^*}A \\
b_*(\os{r}{\us{{\cal W}_n({\cal O}_{\os{\circ}{Y}{}^{(0)}})}{\otimes}}
{\cal W}_n\Om^d_{\os{\circ}{Y}{}^{(0)}}) @>{\Psi_{Y,n,r}}>> 
\os{r}{\us{{\cal W}_n({\cal O}_Y)}{\otimes}}{\cal W}_n\Om^d_Y.
\end{CD}
\tag{6.10.1}\label{cd:fsiws} 
\end{equation*} 
\end{prop} 
\begin{proof} 
This immediately follows from the following functoriality of the Poincar\'{e} residue isomorphism. 
\end{proof}

\parno
\begin{center}
{{\rm \Large{\bf Appendix}}}
\end{center}

\section{Generalization of Illusie's Poincar\'{e} residue isomorphism}\label{sec:gir}
Let the notations be as in the beginning of \S\ref{sec:hpc}. 
Especially let ${\cal E}$ be a locally generated by horizontal sections of 
$\nabla \col {\cal E}\lo {\cal E}\otimes_{{\cal O}_X}\Om^{\bul}_{X/\os{\circ}{S}}$ 
satisfying the conditions (I)$\sim$(IV). 
Set 
\begin{align*} 
P_k({\cal E}\otimes_{{\cal O}_X}\Om^{\bul}_{X/S}):=
{\rm Im}({\cal E}\otimes_{{\cal O}_X}P_k\Om^{\bul}_{X/\os{\circ}{S}}\lo 
{\cal E}\otimes_{{\cal O}_X}\Om^{\bul}_{X/S}). 
\end{align*} 
In this section we generalize Illusie's  Poincar\'{e} residue isomorphism in 
\cite[Appendix (2.2)]{io}. 
That is, we prove that there exists the following isomorphism:  
\begin{align*} 
{\rm gr}^P_k({\cal E}\otimes_{{\cal O}_X}\Om^{\bul}_{X/S})
\os{\sim}{\lo} K^{\bul-k}_{k}({\cal E}) \quad (k\in {\mab N}).  
\tag{7.0.1}\label{ali:cpsips} 
\end{align*}  
We also generalize \cite[Appendice (2.6)]{io}. 
First we begin with the following proposition:

\begin{prop}[{\rm {\bf cf.~\cite[p.~399]{io}}}]\label{prop:ki}
Let $Y/S$ be as in the beginning of {\rm \S\ref{sec:hpc}}. 
Let ${\cal G}$ be a flat quasi-coherent ${\cal O}_Y$-module. 
%Set $P_k({\cal F}\otimes_{{\cal O}_X}\Om^i_{Y/\os{\circ}{S}})
%:={\cal F}\otimes_{{\cal O}_X}P_k\Om^i_{Y/\os{\circ}{S}}$. 
Then the following morphism 
\begin{align*} 
 \theta \wedge  \col {\cal G}\otimes_{{\cal O}_Y}\Om^{i -1}_{Y/\os{\circ}{S}}\owns 
e\otimes \om \lom e\otimes(\theta \wedge \om)\in   
{\cal G}\otimes_{{\cal O}_Y}\Om^i_{Y/\os{\circ}{S}} 
\quad (i\in {\mab N})
\tag{7.1.1}\label{ali:cfefs} 
\end{align*} 
of ${\cal O}_Y$-modules is strictly compatible with $P$'s in the following sense$:$ 
for nonnegative integers $i$ and $k$, 
\begin{align*} 
(\theta \wedge ) ({\cal G}\otimes_{{\cal O}_Y}\Om^{i-1}_{Y/\os{\circ}{S}})\cap 
P_k({\cal G}\otimes_{{\cal O}_Y}\Om^i_{Y/\os{\circ}{S}})=
( \theta \wedge  )(P_{k-1}({\cal G}\otimes_{{\cal O}_Y}\Om^{i-1}_{Y/\os{\circ}{S}})). 
\tag{7.1.2}\label{ali:ixs} 
\end{align*} 
\end{prop} 
\begin{proof} 
It suffices to prove the inclusion ``$\subset$'' in (\ref{ali:ixs}). 
The problem is local. 
%Hence we may assume that there exists 
%a solid morphism $X\lo {\mab A}_S(a,b)$ over $S$. 
%Because we think that one may give the complicated proof of 
%the inclusion even in this simple case, 
%we give the complete simple proof as follows. 
%\par 
%The point is to consider the base 
Let $m_1,\ldots, m_r$ be local sections of $M_Y$ such that 
$\{\theta,d\log m_1,\ldots,d\log m_r\}$ is a local basis 
of $\Om^1_{Y/\os{\circ}{S}}$ 
over ${\cal O}_Y$.  
%We may assume that there exists $1\leq a\leq r$ such that 
%the image of $m_i$ in $M_Y/{\cal O}_Y^*$ $(1\leq i\leq a)$ 
%is nontrivial and 
%the image of $m_i$ in $M_Y/{\cal O}_Y^*$ $(a< i\leq r)$ 
%is trivial. 
%and not to consider 
%$\{d\log x_0,\ldots,d\log x_a,dx_{a+1},\ldots, dx_b\}$. 
Let $\om$ be a local section of $\Om^{i-1}_{Y/\os{\circ}{S}}$ 
such that $\theta \wedge \om  \in P_k\Om^i_{Y/\os{\circ}{S}}$.  
%Assume that $\om \not\in P_{k-1}\Om^i_{Y/\os{\circ}{S}}$.
Express  
$\om = \theta \wedge \eta +\om'$, 
where $\om' \not\in  (\theta \wedge ) (\Om^{i-1}_{Y/\os{\circ}{S}})$ 
and $\eta\in \Om^{i-1}_{Y/\os{\circ}{S}}$. 
Then $\theta \wedge \om =\theta \wedge \om'$. 
Because $\theta \wedge \om' \in P_k\Om^i_{Y/\os{\circ}{S}}$ 
and $\om' \not\in  (\theta \wedge ) (\Om^{i-1}_{Y/\os{\circ}{S}})$, 
$\om'\in P_{k-1}\Om^i_{Y/\os{\circ}{S}}$.  
Hence we obtain the following equality: 
\begin{align*} 
(\theta \wedge ) (\Om^{i-1}_{Y/\os{\circ}{S}})\cap 
P_k\Om^i_{Y/\os{\circ}{S}}=
(\theta \wedge ) P_{k-1}\Om^{i-1}_{Y/\os{\circ}{S}}. 
\tag{7.1.3}\label{ali:itrxs} 
\end{align*} 
and we see that the following sequence 
\begin{align*} 
0\lo P_{k-1}\Om^{i-1}_{Y/S}
\os{\theta \wedge}{\lo} \Om^i_{Y/\os{\circ}{S}} 
\lo \Om^i_{Y/\os{\circ}{S}}/
((\theta \wedge ) (\Om^{i-1}_{Y/\os{\circ}{S}})\cap 
P_k\Om^i_{Y/\os{\circ}{S}})\lo 0 
\end{align*} 
of ${\cal O}_Y$-modules is exact. 
Because ${\cal G}$ is a flat ${\cal O}_Y$-module, 
the following sequence 
\begin{align*} 
0\lo {\cal G}\otimes_{{\cal O}_Y}P_{k-1}\Om^{i-1}_{Y/S}
\os{\theta \wedge }{\lo} {\cal G}\otimes_{{\cal O}_Y}\Om^i_{Y/\os{\circ}{S}} 
\lo 
{\cal G}\otimes_{{\cal O}_Y}
[\Om^i_{Y/\os{\circ}{S}}/
((\theta \wedge ) (\Om^{i-1}_{Y/\os{\circ}{S}})\cap 
P_k(\Om^i_{Y/\os{\circ}{S}}))]\lo 0 
\tag{7.1.4}\label{ali:oox}
\end{align*} 
is exact.  
Because 
\begin{align*} 
{\cal G}\otimes_{{\cal O}_Y}
((\theta \wedge ) (\Om^{i-1}_{Y/\os{\circ}{S}})\cap 
P_k(\Om^i_{Y/\os{\circ}{S}}))
=(\theta \wedge ) ({\cal G}\otimes_{{\cal O}_Y}\Om^{i-1}_{Y/\os{\circ}{S}})
\cap ({\cal G}\otimes_{{\cal O}_Y}P_k\Om^i_{Y/\os{\circ}{S}}),
\end{align*} 
the exact sequence (\ref{ali:oox}) implies the equality (\ref{ali:ixs}). 
\end{proof}

\begin{coro}\label{coro:appstc} 
Let $k$ be an integer. 
The following sequences 
\begin{align*} 
P_{k-1}({\cal G}\otimes_{{\cal O}_X}\Om^{\bul}_{X/\os{\circ}{S}})[-1]
\os{\theta \wedge }{\lo} 
P_k({\cal G}\otimes_{{\cal O}_X}\Om^{\bul}_{X/\os{\circ}{S}})
\lo 
P_k({\cal G}\otimes_{{\cal O}_X}\Om^{\bul}_{X/S})\lo 0, 
\tag{7.2.1}\label{ali:cpfs} 
\end{align*}
\begin{align*} 
0\lo P_{k-1}({\cal G}\otimes_{{\cal O}_X}\Om^{\bul}_{X/S})[-1]
\os{\theta \wedge }{\lo} 
P_k({\cal G}\otimes_{{\cal O}_X}\Om^{\bul}_{X/\os{\circ}{S}})
\lo 
P_k({\cal G}\otimes_{{\cal O}_X}\Om^{\bul}_{X/S})\lo 0
\tag{7.2.2}\label{ali:cpfps} 
\end{align*}
%and 
%\begin{align*} 
%\cdots \os{\theta \wedge}{\lo}  
%P_{k-2}\Om^{\bul}_{X/\os{\circ}{S}}[-2]
%\os{\theta \wedge}{\lo} 
%P_{k-1}\Om^{\bul}_{X/\os{\circ}{S}}[-1]
%\os{\theta \wedge}{\lo} 
%P_k\Om^{\bul}_{X/\os{\circ}{S}}\lo 
%P_k\Om^{\bul}_{X/S}\lo 0
%\tag{7.2.3}\label{ali:cpfps} 
%\end{align*}
of complexes of $f^{-1}({\cal O}_S)$-modules 
are exact. 
\end{coro} 

\begin{prop}\label{prop:ess}
For each $i$, the resulting sequences of 
{\rm (\ref{ali:cpfps})} by the operations $Z^i$, $B^i$ and ${\cal H}^i$ 
$(i\in {\mab Z}_{\geq 0})$ are exact. 
\end{prop}
\begin{proof} 
The analogous proof to the first proof of (\ref{lemm:pte}) works. 
\end{proof}

\begin{prop}[{\rm {\bf cf.~\cite[Appendice (2.2)]{io}}}]\label{prop:ngr}
There exists the isomorphism {\rm (\ref{ali:cpsips})}. 
%\begin{align*} 
%{\rm gr}^P_k\Om^{\bul}_{X/S}\os{\sim}{\lo} K^{\bul-k}_{k} \quad (k\in {\mab N}).  
%\tag{3.7.1}\label{ali:cpsps} 
%\end{align*}   
\end{prop}
\begin{proof} 
First assume that $k=0$. 
Then ${\rm gr}^P_0({\cal E}\otimes_{{\cal O}_X}\Om^{\bul}_{X/S})=
P_0({\cal E}\otimes_{{\cal O}_X}\Om^{\bul}_{X/S})=
P_0({\cal E}\otimes_{{\cal O}_X}\Om^{\bul}_{X/\os{\circ}{S}})
=K^{\bul}_0({\cal E})$. 
\par 
Next assume that $k\geq 1$. 
By (\ref{ali:cpfps}) we see that the following upper sequence 
\begin{equation*} 
\begin{CD}
{\rm gr}^P_{k-1}({\cal E}\otimes_{{\cal O}_X}\Om^{\bul}_{X/\os{\circ}{S}})[-1]
@>{\theta \wedge  }>>
{\rm gr}^P_k({\cal E}\otimes_{{\cal O}_X}\Om^{\bul}_{X/\os{\circ}{S}})@>>> 
{\rm gr}^P_k({\cal E}\otimes_{{\cal O}_X}\Om^{\bul}_{X/S})\lo 0 \\
@V{{\rm Res},\,\simeq}VV @VV{{\rm Res},\,\simeq}V\\
{\cal E}\otimes_{{\cal O}_X}
\Om^{\bul -(k-1)}_{k-2}({\cal E})[-1]
@>{\iota^{(k-1)*}}>> {\cal E}\otimes_{{\cal O}_X}\Om^{\bul-k}_{k-1}({\cal E})
\end{CD}
\tag{7.4.1}\label{ali:cpfxsps}
\end{equation*} 
is exact.  
Because the sequence (\ref{eqn:cptle}) is exact, 
\begin{align*} 
{\rm gr}^P_k({\cal E}\otimes_{{\cal O}_X}\Om^{\bul}_{X/S})
={\rm Coker}(\iota^{(k-1)*}\col 
\Om^{\bul -(k-1)}_{k-2}({\cal E})[-1]\lo \Om^{\bul-k}_{k-1}({\cal E}))
=K^{\bul-k}_{k}({\cal E}). 
\end{align*} 
\end{proof} 

\begin{defi}
%(1) The morphism $\theta \col M_n\lo M_{n+1}$ in \cite[(2.2.3)]{io} 
%is not a morphism of complexes: $d\theta(\om)=-\theta(\om)$ $(\om \in M_n)$. 
%\par 
%(2) 
We call the isomorphism 
(\ref{ali:cpsips}) the {\it relative Poincar\'{e} residue isomorphism} of 
${\cal E}$ for $X/S$.  
This is a generalization of \cite[Appendice (2.1.4)]{io}. 
(Note that we do not need the local description of the isomorphism 
(\ref{ali:cpsips}) unlike in [loc.~cit.] to obtain the isomorphism (\ref{ali:cpsips}).) 
%Though we have defined  the {\it relative Poincar\'{e} residue isomorphism} of $X/S$, 
%we do not use this 
\end{defi}

\begin{prop}\label{prop:aris}
{\rm {\bf (cf.~\cite[Appendice (2.5)]{io})}}  
The resulting sequences of the following sequence  
by the operations ${\cal H}^i$, $B^i$ and $Z^i$ 
$(i\in {\mab Z}_{\geq 0})$ are exact$:$ 
\begin{align*} 
0\lo P_{k-1}({\cal E}\otimes_{{\cal O}_X}\Om^{\bul}_{X/S})
\lo P_k({\cal E}\otimes_{{\cal O}_X}\Om^{\bul}_{X/S})
\lo {\rm gr}_k^P({\cal E}\otimes_{{\cal O}_X}\Om^{\bul}_{X/S})\lo 0
\end{align*} 
\end{prop}
\begin{proof} 
The proof is the same as that of (\ref{prop:iga}). 
\end{proof}

\begin{prop}\label{prop:aaci}
Let ${\cal F}$ be as in {\rm \S\ref{sec:hpc}}. 
The isomorphism 
$C^{-1}\col {\cal F}\otimes_{{\cal O}_{X'}}\Om^i_{X'/S^{[p]}}\os{\sim}{\lo}
F_*{\cal H}^i({\cal E}\otimes_{{\cal O}_X}\Om^{\bul}_{X/S})$
induce the following isomorphism 
\begin{align*} 
C^{-1}\col {\cal F}\otimes_{{\cal O}_{X'}}P_k\Om^i_{X'/S}
\os{\sim}{\lo}
F_*{\cal H}^i({\cal E}\otimes_{{\cal O}_X}P_k\Om^{\bul}_{X/S})\quad (i,k\in {\mab N}). 
\tag{7.7.1}\label{ali:afhnox}
\end{align*}
\end{prop} 
\begin{proof} 
By (\ref{ali:cpfs}) and induction on $k$, we obtain (\ref{prop:aaci}). 
%This follows from (\ref{ali:fhoox}) and (\ref{ali:fhhox}). 
\end{proof} 

\begin{prop}[{\rm {\bf cf.~\cite[Appendice (2.6)]{io}}}]\label{prop:aahc}
Consider the following conditions: 
\par 
{\rm (a)} $R^jf_*(B^iK^{\bul}_k({\cal E}))=0$ for $\forall i, \forall j, \forall n$. 
\par 
{\rm (b)} $R^jf_*(B^i\Om^{\bul}_k({\cal E}))=0$ for $\forall i, \forall j, \forall n$. 
\par 
{\rm (c)} $R^jf_*(B^i{\rm gr}^P_k({\cal E}\otimes_{{\cal O}_X}\Om^{\bul}_{X/S}))=0$ for $\forall i, \forall j, \forall k$. 
%\par 
%{\rm (d)} $R^jf_*(B{\rm gr}^P_n\Om^i_{X/\os{\circ}{S}})=0$ 
%for $\forall i, \forall j, \forall n$. 
\par 
{\rm (d)} $R^jf_*(B^i(P_k({\cal E}\otimes_{{\cal O}_X}\Om^{\bul}_{X/S})))=0$ 
and $R^jf_*(B^i({\cal E}\otimes_{{\cal O}_X}\Om^{\bul}_{X/S}))=0$ 
for $\forall i, \forall j, \forall k$.
\par 
{\rm (e)} $R^jf_*(B^i(P_k({\cal E}\otimes_{{\cal O}_X}\Om^{\bul}_{X/\os{\circ}{S}})))=0$ 
and $R^jf_*(B^i({\cal E}\otimes_{{\cal O}_X}\Om^{\bul}_{X/\os{\circ}{S}}))=0$ 
for $\forall i, \forall j, \forall k$.
\parno 
Then the following hold$:$
\par 
$(1)$ {\rm (a)},  {\rm (b)} and {\rm (c)} are equivalent. 
\par 
$(2)$ {\rm (c)},  {\rm (d)} and {\rm (e)} are equivalent.  
%{\rm (f)}. 
%\par 
%$(3)$ {\rm (d)} implies {\rm (f)}. 
\par 
$($Consequently {\rm (a)},  {\rm (b)}, {\rm (c)},  
{\rm (d)} and {\rm (e)} are equivalent.$)$  
\end{prop} 
\begin{proof} 
(1): (1) follows from (\ref{prop:iga}), (\ref{prop:ess}) and (\ref{prop:ngr}).  
\par 
(2):  Set $T:=S$ or $\os{\circ}{S}$. 
Note that the sheaf is the associated sheaf to the following presheaf 
for open sub-log scheme $U$ of $B$: 
$U \lom R^jf_{U*}(B\Om^i_{X\times_TU/U})$, 
where $f_U\col X\times_BU\lo U$ is the structural morphism. 
Hence, to prove that $R^jf_*(B\Om^i_{X/T})=0$, it suffices to 
assume that  $B\Om^i_{X/T}=P_kB\Om^i_{X/T}$ for some $k\in {\mab N}$. 
Now (2) follows from (\ref{prop:aris}) and (\ref{prop:ess}).  
\end{proof}

%\begin{lemm}\label{lemm:ci0}
%The Cartier morphism 
%$C^{-1} \col \Om^i_{\os{\circ}{X}/\os{\circ}{S}}\lo 
%F_*({\cal H}^i(\Om^{\bul}_{\os{\circ}{X}/\os{\circ}{S}}))$ 
%induces the following isomorphism 
%\begin{align*} 
%C^{-1} \col P_0\Om^i_{X'/\os{\circ}{S}}\os{\sim}{\lo} 
%F_*{\cal H}^i(P_0\Om^{\bul}_{X/\os{\circ}{S}}).
%\tag{7.8.1}\label{ali:ommmb} 
%\end{align*} 
%\end{lemm} 
%\begin{proof} 
%This follows from (\ref{eqn:cptle}) and the following commutative diagram 
%\begin{equation*} 
%\begin{CD} 
%\Om_0@>{\iota^{(0)*}}>> \Om_1\\
%@V{C^{-1}}VV @VV{C^{-1}}V \\
%\Om'{}^{\bul}_{\!\!0}@>{\iota^{(0)*}}>>\Om'{}^{\bul}_{\!\!1}. 
%\end{CD}
%\end{equation*} 
%\end{proof} 

\section{An analogue of Hyodo's criterion for quasi-F-split schemes}\label{sec:hfc}
In this section we give an analogue of 
Hyodo's criterion for quasi-F-split schemes. 
This answers Illusie's question stated in the Introduction.

\par 
Let $Y$ be a closed subvariety of an $F$-split variety $X$. 
Let ${\cal I}$ be the ideal sheaf of $Y$ in $X$. 
Recall that $Y$ is compatibly split if 
there exists a splitting 
$\rho \col F_*({\cal O}_X)\lo {\cal O}_X$ 
of the operator $F\col {\cal O}_X\lo F_*({\cal O}_X)$ 
such that $\rho(F_*({\cal I}))\subset {\cal I}$.  
It is clear that 
this is equivalent to the existence of 
a splitting $\rho \col F_*({\cal O}_X)\lo {\cal O}_X$ 
of the operator $F\col {\cal O}_X\lo F_*({\cal O}_X)$ 
which induces a splitting 
$\rho \col F_*({\cal O}_Y)\lo {\cal O}_Y$ 
of the operator $F\col {\cal O}_Y\lo F_*({\cal O}_Y)$.  

\begin{defi}\label{defi:esp}
Let $Y$ be an SNC scheme. 
We say that $Y$ is {\it compatibly quasi-}$F$-{\it split} if 
there exists a splitting $\rho \col 
F_*({\cal W}_n({\cal O}_{Y^{(0)}}))\lo {\cal W}_n({\cal O}_{Y^{(0)}})$ 
of the operator $F\col {\cal W}_n({\cal O}_{Y^{(0)}})\lo 
F_*({\cal W}_n({\cal O}_{Y^{(0)}}))$ for some $n\in {\mab Z}_{\geq 1}$ 
which induces a splitting 
$\rho \col F_*({\cal W}_n({\cal O}_{Y^{(1)}}))\lo {\cal W}_n({\cal O}_{Y^{(1)}})$ 
of the operator $F\col {\cal W}_n({\cal O}_{Y^{(1)}})\lo 
F_*({\cal W}_n({\cal O}_{Y^{(1)}}))$. 
\end{defi}

\begin{prop}\label{prop:cqs}
Let $Y$ be an SNC scheme. 
Then $Y$ is quasi-$F$-split if and only if 
$Y$ is compatibly quasi-$F$-split. 
\end{prop}
\begin{proof} 
By \cite[Theorem 1]{rs} the following sequence 
\begin{align*} 
0\lo {\cal W}_n({\cal O}_X)\lo {\cal W}_n({\cal O}_{\os{\circ}{X}{}^{(0)}})\lo 
{\cal W}_n({\cal O}_{\os{\circ}{X}{}^{(1)}})\lo \cdots \quad (n\in {\mab Z}_{\geq 1}) 
\tag{8.2.1}\label{ali:csfp}
\end{align*}
is exact. Hence (\ref{prop:cqs}) follows. 
\end{proof} 

%\begin{exem}
%\end{exem}

\section{Ordinarity  at 0 of proper SNC schemes}\label{sec:em}
Let $S$ be a scheme of characteristic $p>0$. 
In this section let $X$ be a 
proper SNC(=simple normal crossing) scheme over $S$. 
Let $f\col X\lo S$ be the structural morphism. 
In this section we consider only trivial log structures. 
Let $F \col X\lo X$ be the Frobenius endomorphism. 
We say that $X/S$ is {\it ordinary at $0$} if 
$F^*\col R^hf_*({\cal O}_X)\lo R^hf_*({\cal O}_X)$ is bijective 
for any $h\in {\mab N}$. 
When $S$ is the underlying scheme of a family of log points and 
when $X/S$ is the underlying morphism of a morphism of SNCL schemes $Y/T$ 
with structural morphism $g\col Y\lo T$, 
this definition is equivalent to the variation of the ordinarity at $0$ defined in 
\cite[(6.1)]{nlfc}: $R^hg_*(B\Om^1_{Y/T})=0$. 
This follows from the following tautological exact sequence 
\begin{align*} 
0\lo {\cal O}_Y\lo F_*({\cal O}_Y)\os{d}{\lo} B_1\Om^1_{Y/T}\lo 0. 
\end{align*} 
The following proposition may be of independent interest (cf.~\cite[Conjecture 
${\textrm N_n}$]{st}): 

\begin{prop}\label{prop:ii} 
If $X^{(i)}/S$ is ordinary at $0$ for all $i$, 
then $X/S$ is ordinary at $0$. 
\end{prop}
\begin{proof} 
Because the following sequence 
\begin{align*} 
0\lo {\cal O}_X\lo {\cal O}_{X^{(0)}}\lo {\cal O}_{X^{(1)}}\lo \cdots 
\end{align*} 
is exact and this exact sequence is compatible with the operator $F$, 
we have the following sequence 
\begin{align*} 
E_1^{-k, h+k}=R^{h+k}f_*({\cal O}_{X^{(-k+1)}})
\Lo
R^hf_*({\cal O}_X) \quad (-k,q\in {\mab Z}_{\geq 0}),  
\tag{9.1.1}\label{ali:wrrzn}
\end{align*} 
By the assumption $F\col E_1^{-k, h+k}\lo E_1^{-k,h+k}$ is bijective. 
Hence $F\col R^hf_*({\cal O}_X)\lo R^hf_*({\cal O}_X)$ is bijective. 
\end{proof}

\bigskip
\bigskip
\parno
Yukiyoshi Nakkajima 
\parno
Department of Mathematics,
Tokyo Denki University,
5 Asahi-cho Senju Adachi-ku,
Tokyo 120--8551, Japan. 
\parno
{\it E-mail address\/}: 
nakayuki@cck.dendai.ac.jp

\end{document}